\documentclass{amsart}

\usepackage{amsmath}
\usepackage{mathrsfs}
\usepackage{amsthm}
\usepackage[utf8]{inputenc}
\usepackage{amsfonts}
\usepackage{amssymb}
\usepackage{amsthm}
\usepackage{caption}
\usepackage{amssymb, amsmath} 
\usepackage{graphicx}
\usepackage{tikz}
\usepackage[upright]{fourier}
\usepackage{tkz-graph}
\usetikzlibrary{arrows}
\usepackage{color}
\usepackage{color,soul}
\usepackage[normalem]{ulem}
\usepackage{enumitem}
\usepackage[toc,page]{appendix}
\usetikzlibrary{arrows,positioning,automata,shapes,calc,3d,graphs}
\usetikzlibrary{decorations.markings}
\usetikzlibrary{arrows, shadows}
\usepackage[underline=true,rounded corners=false]{pgf-umlsd}
\usepackage{xspace}
\usepackage{verbatim}
\usepackage[margin=1in]{geometry}
\usepackage{xassoccnt}
\tikzstyle{vertex}=[circle, draw, inner sep=0pt, minimum size=10pt]
 \usepackage{graphicx, scalefnt}
 \usepackage{pgfplots}
 \usepackage{etex}
 \usetikzlibrary{decorations.pathmorphing}
\usetikzlibrary{shapes,arrows}
\tikzstyle{decision} = [diamond, draw, text width=5em, text badly centered, node distance=3cm, inner sep=0pt]
\tikzstyle{block} = [rectangle, draw, text width=5em, text centered, rounded corners, minimum height=4em]
\usetikzlibrary{intersections}

\newtheorem{theorem}{Theorem}[section]

\newtheorem{corollary}[theorem]{Corollary}
\newtheorem{lemma}[theorem]{Lemma}

\newtheorem{proposition}[theorem]{Proposition}
\newtheorem{problem}{Problem}
\newtheorem{conjecture}{Conjecture}

\newcommand{\pf}{\bf Proof:}
\newcommand{\epf}{\Box\vspace{0.15in}}

\theoremstyle{definition}
\newtheorem{definition}[theorem]{Definition}
\newtheorem{example}[theorem]{Example}

\numberwithin{equation}{section}

\newcommand{\sone}{\textsf{S}_1}
\newcommand{\gone}{{\sf G}_1}
\newcommand{\sfin}{{\sf S}_{fin}}

\newcommand{\open}{\mathcal{O}}

\newtheorem*{repp@theorem}{\repp@title (reformulated)}
\newcommand{\newrepptheorem}[2]{%
\newenvironment{repp#1}[1]{%
 \def\repp@title{#2 \ref{##1}}%
 \begin{repp@theorem}}%
 {\end{repp@theorem}}}
\makeatother
\newrepptheorem{theorem}{Theorem}

\theoremstyle{definition}
\makeatletter
\newtheorem*{rep@theorem}{\rep@title \ continued}
\newcommand{\newreptheorem}[2]{%
\newenvironment{rep#1}[1]{%
 \def\rep@title{#2 \ref{##1}}%
 \begin{rep@theorem}}%
 {\end{rep@theorem}}}
\makeatother
\newreptheorem{example}{Example}

\newcommand{\integers}{{\mathbb Z}}

\newcommand{\egp}{{\sf E}}

\newcommand{\naturals}{{\mathbb N}}

\begin{document}

\title{Ramsey Theory and the Borel Conjecture}
\author{Marion Scheepers}
\address{Department of Mathematics\\ Boise State University\\ Boise, Idaho 83725}
\email{mscheepe@boisestate.edu}
\subjclass[2000]{03E02, 03E05, 05D10, 22A05, 54D20, 54G10, 54H11 }
\keywords{Ramsey theory, Rothberger bounded, strong measure zero, uniformizable space, topological group, infinite game}

\date{}

\begin{abstract}
The Borel covering property, introduced a century ago by E. Borel, is intimately connected with Ramsey theory, initiated ninety years ago in an influential paper of F.P. Ramsey. The current state of knowledge about the connection between the Borel covering property and Ramsey theory is outlined in this paper. Initially the connection is established for the situation when the set with the Borel covering property is a proper subset of a $\sigma$-compact uniform space. Then the connection is explored for a stronger covering property introduced by Rothberger. After establishing the fact that in this case several landmark Ramseyan theorems are characteristic of this stronger covering property, the case when the space with this stronger covering property is in fact $\sigma$-compact is explored.
\end{abstract}

\maketitle

In his 1919 paper \cite{Borel} E. Borel, studying the notion of Lebesgue measure zero, introduced a covering property for subsets of the real line. A set $X$ of real numbers has Borel's covering property if for each sequence $(\epsilon_n:n\in\mathbb{N})$ of positive real numbers there is a sequence $(x_n:n \in\mathbb{N})$ of elements of $X$ such that the set $\{(x_n-\epsilon_n,\; x_n+\epsilon_n):n\in\mathbb{N}\}$ of open intervals of real numbers covers $X$. Here, and throughout the paper, $\mathbb{N}$ denotes the set of positive integers, also called natural numbers.
 
Borel's definition has been adapted in several ways to other contexts. For a subset $X$ of a metric space $(M,d)$ the corresponding definition is as follows: For each sequence $(\epsilon_n:n\in\mathbb{N})$ of positive real numbers there is a sequence $(x_n:n\in\mathbb{N})$ of elements of $X$ such that the set $\{B_d(x_n,\;\epsilon_n):n\in\mathbb{N}\}$ is a cover for $X$. Here, $B_d(x,\epsilon)$ denotes the set $\{y\in X: d(x,y)<\epsilon\}$.  For a subset $X$ of a topological group $(G,\odot)$ the corresponding definition is as follows: For each sequence $(U_n:n\in\mathbb{N})$ of neighborhoods of the identity element $id$ of $G$, there is a sequence $(x_n:n\in\mathbb{N})$ of elements of $X$ such that the set $\{x_n\odot U_n:\; n\in\mathbb{N}\}$ is a cover for $X$. Here the symbol $x\odot N$ denotes the set $\{x\odot y: y\in N\}$.

A remark made in \cite{Borel} provided an initial stimulus for investigating Borel's covering property: Borel noted that every countable set of real numbers has Borel's covering property, and he speculated that the only sets of real numbers that have this covering property are the countable sets. This speculation became known as Borel's Conjecture. By 1976 it was known that Borel's conjecture is independent of the Zermelo-Fraenkel axioms, denoted \textsf{ZFC}. Despite the status of Borel's Conjecture as a statement undecidable in  the axiomatic system \textsf{ZFC}, research on Borel's covering property and its relatives continues unabated and has affected numerous branches of mathematics.  

In \cite{rothberger38} Rothberger modified the definition of Borel's covering property for the context of general topological spaces: For a topological space $(S,\tau)$ and a subset $X$ of $S$, the analogous covering property is as follows: For each sequence $(\mathcal{O}_n:n\in\mathbb{N})$ of open covers of $S$ there is a sequence $(T_n:n\in\mathbb{N})$ of open sets such that for each $n$ it is the case that $T_n\in\mathcal{O}_n$, and $\{T_n:n\in\mathbb{N}\}$ is a cover of $X$. This was a step towards a much broader framework, selection principles in mathematics, in which the Borel covering property and its modifications would be special cases. In this paper we take this broader approach, but we confine ourselves to instances of topological covering properties intimately related to the Borel covering property.
 
Ten years after Borel's paper \cite{Borel} F.P. Ramsey reported \cite{Ramsey} a significant extension of the classical pigeon hole principle. This work was taking place independently of considerations related to the Borel covering property, and was an initial stimulus for the development of the field now known as Ramsey theory.

The purpose of this paper is to give a coherent expository survey on the connection between covering properties related to the Borel covering property, on the one hand, and Ramsey theory on the other hand. As such the reader should not expect a large volume of new results, but instead a reorganization of several known results into one narrative. As this paper is only one among several in this volume, each dedicated to a different aspect of mathematics inspired by the Borel covering property, the reader can expect an overlap of some basic information among these papers. For example, in the first section of this paper, to set the stage, we describe the broader selection principles context for Borel's covering property. In the second section of the paper we connect this broader context with an essential tool, an infinite two-person game inspired by work of F. Galvin, for making the connection in section 3 with Ramsey theory. With this initial outline established, the rest of the paper narrates the deep connections between the covering properties under consideration, and several landmark developments in Ramsey theory.  Ramseyan equivalents of several other covering properties besides the Borel covering property and its relatives have been explored. A very interesting recent investigation along these lines has been published by Tsaban in \cite{Tsaban}. 

Without further notice all topological spaces in this paper are assumed to be Hausdorff spaces - meaning that for any two distinct points $x$ and $y$, there are disjoint open sets $U$ and $V$ such that $x\in U$ and $y\in V$.

\section{Borel's Covering Property as a Selection Principle}

Rothberger's approach in \cite{rothberger38} to the Borel covering property suggests a common framework for treating Borel's covering property and its variations. To describe this framework let $\mathcal{A}$ and $\mathcal{B}$ be families of sets. The symbol
\[
 \sone(\mathcal{A},\mathcal{B})
\]
denotes the statement (called a selection principle) that there is for each sequence $(O_n:n\in\mathbb{N})$ of elements of $\mathcal{A}$ a corresponding sequence $(T_n:n\in\mathbb{N})$ such that for each $n\in\mathbb{N}$ we have $T_n\in O_n$, and $\{T_n:\;n\in\mathbb{N}\}$ is an element of $\mathcal{B}$. To see that the selection principle $\sone(\mathcal{A},\mathcal{B})$ subsumes the already mentioned variations on Borel's covering property, consider the following.

\begin{example}\label{ex:metric}
For a metric space $(M,d)$ define for $\epsilon>0$ the set $\mathcal{U}_{\epsilon} = \{B_d(x,\epsilon):x\in M\}$. Then $\mathcal{U}_{\epsilon}$ is an open cover of $M$. Use the symbol $\mathcal{O}_{diam}$ to denote the set $\{\mathcal{U}_{\epsilon}:\;\epsilon>0\}$ of open covers of $M$ arising in this way. For a subset $X$ of $M$, let $\mathcal{O}_X$ denote the set of $\mathcal{V}$ where $\mathcal{V}$ consists of open subsets of $M$ and $\mathcal{V}$ covers $X$.  Then $X$ has the Borel covering property exactly when the selection principle $\sone(\mathcal{O}_{diam},\mathcal{O}_X)$ holds.
\end{example}

\begin{example}\label{ex:topgp}
For a topological group $(G,\odot)$ define for a open neighborhood $N$ of the identity element $i_G$ the set $\mathcal{U}_{N} = \{x\odot N:\; x\in G\}$. Then $\mathcal{U}_{N}$ is an open cover of $G$. Use the symbol $\mathcal{O}_{nbd}$ to denote the set $\{\mathcal{U}_{N}:\; N \mbox{ a neighborhood of } i_G\}$ of open covers of $G$ arising in this way. For a subset $X$ of $G$, let $\mathcal{O}_X$ denote the set of $\mathcal{V}$ where $\mathcal{V}$ consists of open subsets of $G$ and $\mathcal{V}$ covers $X$.  Then $X$ has the Borel covering property exactly when the selection principle $\sone(\mathcal{O}_{nbd},\mathcal{O}_X)$ holds.
\end{example}

Borel's covering property arose for the real line $\mathbb{R}$ with its standard algebraic operations and its Euclidean metric. For this specific example Borel's covering property can be interpreted both as a property of the metric space $(\mathbb{R}, \vert\cdot\vert)$, and (independently) as a property of the topological group $(\mathbb{R},\; +)$.  A natural common generalization is obtained by formulating the covering property in terms of uniform spaces. See for example \cite{willard}, Chapter 9 for more details about uniform spaces.  
 
\begin{definition}\label{def:uniformity}
For a set $X$, a collection $\Psi$ of subsets of $X\times X$  is said to be a uniformity on $X$ if $\Psi$ has the following properties:
\begin{itemize}
\item[(U1)]{For each $U\in\Psi$ it is the case that $\{(x,x):x\in X\}\subseteq U$}
\item[(U2)]{If $U$ and $V$ are members of $\Psi$, then $U\cap V$ is a member of $\Psi$}
\item[(U3)]{If $U$ is a member of $\Psi$ and $V$ is a subset of $X\times X$ such that $U\subseteq V$, then $V$ is an element of $\Psi$}
\item[(U4)]{If $U$ is a member of $\Psi$, then also the set $\{(y,x):\; (x,y)\in U\}$ is a member of $\Psi$}
\item[(U5)]{For each $U\in\Psi$ there is a $V\in\Psi$ such that whenever $(x,y)\in V$ and $(y,z)\in V$, then it is the case that $(x,z)\in U$}
\end{itemize}
\end{definition}

Let $\Psi$ is a uniformity on a set $X$. The pair $(X,\Psi)$ is said to be a uniform space. An element of the uniformity $\Psi$ is said to be an entourage. Convenient notation for the properties stated in (U4) and in (U5) is as follows: For an entourage $U$, $U^{-1}$ denotes $\{(y,x):\; (x,y)\in U\}$. For a subset $U$ of $X\times X$ the symbol $U\circ U$ denotes the set $\{(x,z)\in X\times X:\; \mbox{ there is a }y\in X \mbox{ such that } (x,y)\in U \mbox{ and } (y,z)\in U\}$. An entourage $U$ is said to be symmetric if $U = U^{-1}$. 
Properties (U1), (U2) and (U3) state that a uniformity is a filter on the set $X\times X$ with the property that each element of the filter contains the set $\Delta_X = \{(x,x):x\in X\}$, while property (U4) states that the filter is closed under the operation $U^{-1}$, and property (U5) states that for each element $U$ of the filter there is an element $V$ of the filter such that $V\circ V\subseteq U$. The blanket assumption of this paper is that topological spaces considered here are Hausdorff spaces: This requirement is imposed on the uniformities relevant to this paper by requiring that a uniformity $\Psi$ is \emph{separating}, which is that $\Delta_X = \bigcap\{U:U\in\Psi\}$.
 
A uniformity $\Psi$ on a set $X$ gives rise to a topology on $X$: For each $U\in\Psi$, and for each $x\in X$ define
\[
  U(x) = \{y\in X:\;(x,y)\in U\}.
\]
Then for $x\in X$ the set $\mathcal{N}_{\Psi}(x) = \{U(x): U\in\Psi\}$ defines a neighborhood basis of $x$ in $X$, and each subset of $X$ of the form $U(x)$ is said to be a neighborhood of $x$. Given a set $\mathcal{N}$ of neighborhood bases for the elements of a set $X$, there is a standard topology on $X$, denoted $\tau_{\mathcal{N}}$, giving rise to $\mathcal{N}$ -- see for example \cite{willard}, Section 5. Accordingly, given a uniformity $\Psi$ on the set $X$, there is a naturally associated topology, denoted $\tau_{\Psi}$, on $X$. By classical results of Tychonoff and of Weil \cite{weil} the topology of a topological space is uniformizable (i.e., there is a uniformity $\Psi$ such that the topology is of the form $\tau_{\Psi}$) if, and only if, the space is completely regular. Since the spaces in this paper are all assumed to be Hausdorff spaces, being uniformizable is equivalent to being a Tychonoff space. 

\begin{example}\label{ex:uniform}
For uniformity $\Psi$ on the set $X$ define a special subfamily of the family of all open covers of the space $(X,\tau_{\Psi})$: For any $N\in \Psi$, the open cover $\{Int_X(N(x)):x\in X\}$ of $X$ is denoted $\mathcal{O}(N)$. The symbol $\mathcal{O}_{\Psi}$ denotes the set $\{\mathcal{O}(N):N\in\Psi\}$ of these special open covers of $X$. When $Y$ is a subset of $X$, then the symbol $\mathcal{O}_Y$ denotes the set of all families $\mathcal{U}$ for which each element of $\mathcal{U}$ is an open set in the topological space $(X,\tau_{\Psi})$, and for which $Y\subseteq \bigcup\mathcal{U}$. Borel's covering property translates to $\sone(\mathcal{O}_{\Psi},\mathcal{O}_Y)$.
\end{example}

Generalizing further, Nachbin introduced the notion of a quasi uniformity \cite{Nachbin}:
\begin{definition}[Nachbin 1948]
A set $\mathcal{U}\subseteq X\times X$ is a quasi-uniformity if 
\begin{enumerate}
\item{$\mathcal{U}$ is a filter,}
\item{For all $U\in\mathcal{U}$ the set $\{(x,x):x\in X\}$ is a subset of $U$,}
\item{For all $U\in\mathcal{U}$ there is a $V\in\mathcal{U}$ such that $V\circ V\subseteq U$.}
\end{enumerate}
If $\mathcal{U}$ is a quasi-uniformity on $X$, then $(X,\mathcal{U})$ is said to be a quasi-uniform space.
\end{definition}
Observe that a quasi-uniformity is a uniformity whenever for each $U\in\mathcal{U}$ it is also the case that $U^{-1}\in\mathcal{U}$. For a quasi-uniformity $\mathcal{U}$, the symbol $\mathcal{U}^{*}$ denotes the set $\{U^{-1}:U\in\mathcal{U}\}$. Note that $\mathcal{U}^{*}$ also is a quasi-uniformity and is called the conjugate quasi-uniformity of $\mathcal{U}$.
As in the case of uniformities, if $U$ is a member of a quasi-uniformity $\mathcal{U}$, then $U(x)$ denotes the set $\{y\in X:\; (x,y)\in U\}$. There is a natural topology associated with a quasi-uniformity: If $\mathcal{U}$ is a quasi-uniformity on the set $X$, define
\[
  \tau_{\mathcal{U}} = \{G\subseteq X:\; (\forall x\in G)(\exists U\in\mathcal{U})(U(x)\subseteq G)\}.
\]
Then $\tau_{\mathcal{U}}$ is a topology on $X$. A topology $\tau$ on $X$ is \emph{compatible} with the quasi-uniformity $\mathcal{U}$ if $\tau = \tau_{\mathcal{U}}$.

\begin{example}\label{ex:quasiuniform}
For quasi uniformity $\mathcal{U}$ on the set $X$ define a special subfamily of the family of all open covers of the space $(X,\tau_{\mathcal{U}})$: For an element $U$ of $\mathcal{U}$, the open cover $\{U(x):x\in X\}$ of $X$ is denoted $\mathcal{O}(U)$. The symbol $\mathcal{O}_{\mathcal{U}}$ denotes the set $\{\mathcal{O}(N):N\in\mathcal{U}\}$ of these special open covers of $X$. When $Y$ is a subset of $X$, then the symbol $\mathcal{O}_Y$ denotes the set of all families $\mathcal{V}$ for which each element of $\mathcal{V}$ is an open set in the topological space $(X,\tau_{\mathcal{U}})$, and for which $Y\subseteq \bigcup\mathcal{V}$. In this example Borel's covering property translates to $\sone(\mathcal{O}_{\mathcal{U}},\mathcal{O}_Y)$.
\end{example}

\begin{example}\label{ex:Rothberger}
For a general topological space $(X,\tau)$, let $\open$ denote the collection of all open covers of $X$. We say that $(X,\tau)$ has the \emph{Rothberger covering property} if $\sone(\open,\open)$ holds. Rothberger introduced this covering property in \cite{rothberger38}. It is evident that if $\tau$ is the topology $\tau_{\Psi}$ generated by the (quasi-) uniformity $\Psi$ on $X$, and if $(X,\tau_{\Psi})$ satisfies $\sone(\open,\open)$, then it satisfies $\sone(\open_{\Psi},\open)$. As such, the Rothberger covering property implies the Borel covering property.
\end{example}

\section{An infinite game and covering properties}

We now arrive at the crucial tool that links the considered covering properties with Ramseyan properties.
\begin{definition}\label{def:Game}
Let families of sets, $\mathcal{A}$ and $\mathcal{B}$, be given. The game $\textsf{G}_1(\mathcal{A},\mathcal{B})$, played between players ONE and TWO, is defined as follows: The game has an inning per positive integer $n$. In inning $n$ ONE first chooses an element $O_n$ of $\mathcal{A}$, and then TWO responds by choosing an element $T_n$ of $O_n$. The play
\[
  O_1,\; T_1,\; O_2,\; T_2,\; \cdots,\; O_n,\; T_n,\; \cdots
\]
is won by player TWO if $\{T_n:\;n\in\mathbb{N}\}$ is a member of $\mathcal{B}$. Else, ONE wins the play.
\end{definition}
The instance of this game when $\mathcal{A} = \mathcal{B} = \mathcal{O}$, the set of all open covers of a topological space $(X,\tau)$, was introduced by F. Galvin in \cite{Galvin}. 
Note that if player ONE does not have a winning strategy in the game $\gone(\mathcal{A},\mathcal{B})$, then the selection principle $\sone(\mathcal{A},\mathcal{B})$ holds: For let a sequence $(A_n:n\in\mathbb{N})$ of elements of the family $\mathcal{A}$ be given. Assign the strategy $F$, defined to play $A_n$ in the $n$-th inning, to player ONE of the game $\gone(\mathcal{A},\mathcal{B})$. Since ONE has no winning strategy in this game, $F$ is not a winning strategy. Consider a play of the game where ONE used the strategy $F$, but lost. It is of the form
\[
  A_1, B_1, A_2, B_2, \cdots, A_n, B_n, \cdots
\]
where for each $n$ we have $B_n\in A_n$. Since ONE lost this play, TWO won and we have $\{B_n:n\in\mathbb{N}\}\in\mathcal{B}$. But then the sequence $(B_n:n\in\mathbb{N})$ witnesses $\sone(\mathcal{A},\mathcal{B})$ for the sequence $(A_n:n\in\mathbb{N})$.

One of the fundamental questions is: When does the selection principle $\sone(\mathcal{A},\mathcal{B})$ imply that ONE has no winning strategy in the game $\gone(\mathcal{A},\mathcal{B})$? Perhaps the most delicate instance of a positive result for this question is the following theorem of Pawlikowski \cite{JP}: 
\begin{theorem}[J. Pawlikowski]\label{thm:Pawlikowski} For any topological space $(X,\tau)$, the selection principle $\sone(\open,\open)$ holds if, and only if, ONE has no winning strategy in the game $\gone(\open,\open)$.
\end{theorem}

More easily proved instances of such equivalences can be obtained by imposing additional constraints on the underlying topological space, or on the objects ONE is allowed to play during the game. In the rest of this section we prove, for the convenience of the reader, some instances relevant to the Borel covering property. The proofs are not new but can be gleaned from the proof of Theorem 2 in \cite{GMS} where it is proved for metric spaces:
 
\begin{theorem}\label{Th:commongeneralize}  For $(X ,\Psi)$ a $\sigma$ -totally bounded uniform space and $Y$ a subset of $X$, the following are equivalent:
\begin{enumerate}
\item{$(X ,\Psi)$ satisfies $\textsf{S}_1(\mathcal{O}_{\Psi},\mathcal{O}_Y )$.}
\item{ONE has no winning strategy in the game $\textsf{G}_1(\mathcal{O}_{\Psi},\mathcal{O}_Y )$.}
\end{enumerate}
\end{theorem}

\begin{proof} The proof of $(2)\Rightarrow(1)$ has been explained above. We prove that $(1)\Rightarrow (2):$ 
Let $F$ be a strategy for player ONE in the game $\textsf{G}_1(\mathcal{O}_{\Psi},\mathcal{O}_Y )$ on $X$. Since $X$ is $\sigma$-totally bounded, and write $X = \bigcup_{n\in\mathbb{N}}X_n$
where for each $n$ we have $\emptyset\neq X_n \subseteq X_{n+1}$ and each $X_n$ is totally bounded. For each $n$, putting $Y_n = Y \cap X_n$, the uniform space $(X,\Psi)$ satisfies the property $\textsf{S}_1(\mathcal{O}_{\Psi},\mathcal{O}_{Y_n} )$.

To defeat ONE's strategy TWO will in specific innings $m$ concentrate attention on specific $Y_n$'s. To this end, partition $\mathbb{N}$ into infinitely many infinite subsets $S_n$, $n\in\mathbb{N}$. For innings numbered by members of $S_n$ TWO will focus on $Y_n$.
We now use ONE's strategy $F$ to recursively define a sequence $\{N_k: k\in\mathbb{N}\}$ and an array of sets $\mathcal{U}(T_1, . . . , T_k)$ where:
\begin{enumerate}
\item{ For each k, $N_k$ is a symmetric member of the uniformity $\Psi$;}
\item{ With $n_1$ such that $1 \in S_{n_1} $, $\mathcal{U}(\emptyset)$ is a finite subset of $F(\emptyset)$ (ONE's first move) that covers $Y_1$ (as $X_1$ is totally bounded), and $N_1$ is such that for
each $x \in Y_1$ there is a $V \in \mathcal{U}(\emptyset)$ with $(N_1 \circ N_1)(x) \subseteq V$;}
\item{For each $(T_1,\dots , T_k)$ such that $T_1 \in \mathcal{U}(\emptyset)$,  $T_2 \in \mathcal{U}(T_1),\cdots$ and $T_k \in \mathcal{U}(T_1, \cdots. , T_{k-1})$ and for $n_{k+1}$ such that $k+1 \in S_{n_{k+1}}$
we have $\mathcal{U}(T_1, \cdots , T_k)$ a finite subset of $F (T_1,  \cdots , T_k)$ that covers $Y_{n_{k+1}}$. Note that there are only finitely many such $(T_1, . . . , T_k)$. $N_{k+1}$ is a member of the uniformity such that for each such sequence $(T_1, . . . , T_k)$ and for each $x\in Y_{n_{k+1}}$ there is a $U\in \mathcal{U}(T_1, . . . , T_k)$ with $(N_{k+1}\circ N_{k+1})(x)\subseteq U$.
}
\end{enumerate}

With this data available, construct a play against $F$ won by TWO as follows: 

Fix an $m \in \mathbb{N}$. Since $\textsf{S}_1(\mathcal{O}_{\Psi},\mathcal{O}_{Y_m} )$ holds, select for each $k \in S_m$ an $x_k \in Y$ such that $(N_k(x_k): k \in S_m)$ covers $Y_m$.
We may assume each $x_k$ is in $Y_m$ - for suppose an $x_k$ is not in $Y_m$. If $N_k(x_k)\cap Y_m = \emptyset$ we may with impunity replace this
$x_k$ by one from $Y_m$. 
However, if $N_k(x_k)\cap Y_m \neq \emptyset$, then let y be an element of this intersection. We claim that $N_k(x_k)\cap  Y_m \subseteq
(N_k\circ N_k)(y)$. 
For let $z\in  N_k(x_k)\cap Y_m$ be given. Note that $(x_k, y)\in  N_k$ and $(x_k, z)\in  N_k$. Since $N_k$ is symmetric we have
$(y, x_k)\in N_k$ and $(x_k, z)\in N_k$.
 But then $(y, z)\in  N_k \circ N_k$, which implies that $z \in (N_k \circ N_k)(y)$. Thus, letting this y be the new
$x_k$ we see that we may choose for each $k \in S_m$ an $x_k \in Y_k$ such that $((N_k\circ  N_k)(x_k): k\in S_m)$ covers $Y_k$.
Finally, recursively choose a sequence $(T_k: k\in\mathbb{N})$ as follows: Choose $T_1 \in \mathcal{U}(\emptyset)$ with $(N_1\circ N_1)(x_1)\subseteq T_1$. With $T_1, . . . , T_m$
chosen, choose $T_{m+1} \in \mathcal{U}(T_1, . . . , T_m)$ with $(N_{m+1}\circ N_{m+1})(x_{m+1})\subseteq T_{m+1}$. Then the sequence
$F (\emptyset), T_1, F (T_1), . . . , T_k, F(T_1, . . . , T_k), T_{k+1}, . . .$
is an $F$ -play lost by ONE.
\end{proof}

Strengthening the hypothesis on $X$ in Theorem \ref{Th:commongeneralize} from $\sigma$-totally bounded to $\sigma$-compact gives a result towards a characterization of $\textsf{S}_1(\mathcal{O}_{\Psi},\mathcal{O}_{X} )$ in terms of a Ramseyan property. The argument needs, as one of its components, the following  Theorem \ref{th:LebesgueCov} which is proven in Theorem 33 on p. 199 of \cite{Kelley}:
\begin{theorem}\label{th:LebesgueCov}
Let $(X,\Psi)$ be a uniform space and $K\subset X$ a subset compact in the uniform topology. For each open cover $\mathcal{U}$ of $X$ there is a member $W$ of the uniformity $\Psi$ such that for each $x\in K$ there is a $U\in\mathcal{U}$ such that $W(x)\subseteq U$.
\end{theorem}
Another component towards a Ramseyan characterization is next given in Theorem \ref{th:MSUniform}. 
Note that the equivalence of (2) and (3) for general topological spaces (not only $\sigma$-compact uniform spaces) holds by Pawlikowski's theorem stated above as Theorem \ref{thm:Pawlikowski}. 

\begin{theorem}\label{th:MSUniform} Let $(X ,\Psi)$ be a $\sigma$-compact uniform space and let Y be a subset of $X$. The following are equivalent:
\begin{enumerate}
\item{$(X ,\Psi)$ satisfies $\textsf{S}_1(\mathcal{O}_{\Psi},\mathcal{O}_Y )$.}
\item{$(X ,\Psi)$ satisfies $\textsf{S}_1(\mathcal{O},\mathcal{O}_Y )$.}
\item{ONE has no winning strategy in the game $\textsf{G}_1(\mathcal{O},\mathcal{O}_Y )$.}
\item{ONE has no winning strategy in the game $\textsf{G}_1(\mathcal{O}_{\Psi},\mathcal{O}_Y )$.}
\end{enumerate}
\end{theorem}
\begin{proof}
The equivalence of (1) and (4) was proven in Theorem \ref{Th:commongeneralize}. We must show that (1) implies (2), (2) implies (3) and (3) implies (4). Some of these implications have been proven in prior work, as we shall point out. For the convenience of the reader we give proofs here.

{\flushleft{\underline{Proof that (1) implies (2):}}}
Let $(\mathcal{O}_n:n\in\mathbb{N})$ be a sequence of open covers of $X$. Partition $\mathbb{N}$ into infinitely many infinite subsets $S_n,\; n\in\mathbb{N}$. Fix any $k\in\mathbb{N}$, and consider the subsequence $(\mathcal{O}_n:n\in S_k)$ of the given sequence of open covers. For a fixed $n\in S_k$, apply Theorem \ref{th:LebesgueCov} to $\mathcal{O}_n$ and the compact subset $K_k$ of $X$. Fix an element $W_n$ of $\Psi$ such that  for each $x\in K_k$ there is a $U\in\mathcal{O}_n$ for which we have $W_n(x)\subseteq U$. We may assume that $W_n$ is symmetric. Then choose $V_n\in\Psi$ symmetric such that $V_n\circ V_n\subset W_n$

Now for each $n\in S_k$ the set $\mathcal{U}_n = \{V_n(x):x\in X\}$ is an element of $\mathcal{O}_{\Psi}$. Applying (1) for $\sone(\mathcal{O}_{\Psi},\mathcal{O}_{Y_k})$ to the sequence $(\mathcal{U}_n:n\in S_k)$, choose for each $n\in S_k$ an $x_n$ such that $Y_k \subseteq \bigcup_{n\in S_k}V_n(x_n)$. For each $n\in S_k$ consider $V_n(x_n)$: If $V_n(x_n)\cap K_k \neq \emptyset$ but $x_n\not\in K_k$, then consider $y\in K_k\cap V_n(x_n)$. Then we have $K_k\cap V_n(x_n) \subseteq W_n(y)$: For pick a $z\in K_k\cap V_n(x_n)$. We have $(z,x_n)\in V_n$, and $(x_n,y)\in V_n$. Since $V_n\circ V_n\subseteq W_n$ it follows that $(z,y)\in W_n$, and thus $z\in W_n(y)$.  In this case we redefine $x_n$ to be such a $y$. Then the sequence $(W_n(x_n):n\in S_k)$ is a cover of $Y_k$. For each $n\in S_k$, when $x_n\in K_k$, choose a $U_n\in\mathcal{O}_n$ for which $W_n(x_n)\subseteq U_n$, and otherwise choose $U_n\in\mathcal{O}_n$ arbitrarily. Then the sequence $(U_n:n\in S_k)$ covers $Y_k$.

It follows that the sequence $(U_n:n\in\mathbb{N})$ where of each $n$ we have $U_n\in\mathcal{O}_n$, is an open cover of $Y$. This completes the proof that (1) implies (2).

{\flushleft{\underline{Proof that (2) implies (3):}}} (We follow the argument in the proof of $(2)\Rightarrow(3)$ in Theorem 9 of \cite{MSFinPowers})  To this end, write
\[
  X = \bigcup_{n\in\mathbb{N}} K_n
\]
where for $m<n$ we have $K_m\subset K_n$ and each $K_m$ is compact. For each $n$, put $Y_n = Y\cap K_n$. 

Fix a partition $\mathbb{N} = \cup\{S_n:n\in\mathbb{N}\}$ so that each $S_n$ is infinite and whenever $m\neq n$, then $S_m$ and $S_n$ are disjoint.  Let $F$ be a strategy of player ONE of the game $\gone(\open,\open_Y)$. To find a play of $\gone(\open,\open_Y)$ during which ONE used the strategy $F$ and yet TWO won, we proceed as follows: 

Recursively define a sequence $( W_n : n \in \mathbb{N})$ of elements of the uniformity $\Psi$,  and an array $(\mathcal{U}(U_1, \cdots, U_n): n \in \mathbb{N})$ of finite families of open sets such that:
\begin{enumerate}
\item{With $n_1$ such that $1 \in S_{n_1}$, $\mathcal{U}(\emptyset)$ is a finite subset of $F(\emptyset)$ which is an open cover of $Y_{n_1}$, and $W_1$ is a member of $\Psi$ such that for each $x\in Y_1$ we have $W_1(x)\subset U$ for some $U\in \mathcal{U}(\emptyset)$;}
\item{For each $(U_1,\cdots, U_k)$ such that $U_1 \in \mathcal{U}(\emptyset)$ and $U_{i+1} \in \mathcal{U}(U_1... , U_i)$, $i < k$,
and for $n_{k+1}$ such that $k + 1 \in S_{n_{k+1}}$, the set $\mathcal{U}(U_1,\cdots, U_k)$ is a finite subset of
$F(U_1,\cdots,U_k)$ which covers $Y_{n_{k+1}}$ and $W_{n_{k+1}}\in\Psi$ is such that $W_{n_{k+1}}\subset W_{n_k}$ and for each $x\in Y_{n_{k+1}}$ there is a $U\in\mathcal{U}(U_1,\cdots,U_{k})$ with $W_{n_{k+1}}(x)\subseteq U$.}
\end{enumerate}

Suppose we have defined these objects up to the $m$-th stage. Here is how stage $m + 1$'s definition is made: For each sequence $(U_1, \cdots, U_n)$ of open sets such
that $U_1 \in \mathcal{U}(\emptyset)$, $U_2 \in \mathcal{U} (U_1), \cdots$. and $U_m \in \mathcal{U} (U_1, \cdots, U_{m-1})$, apply ONE's strategy $F$ to obtain an open cover $F(U_1,\cdots, U_m)$ of $X$. Determine the $n_{m+ 1}$ or which
$m + 1 \in S_{n_{m+1}}$ and then let $\mathcal{U}(U_1,\cdots,U_{m+1})$ be a finite subset of $F(U_1,\cdots,U_{m+1})$ which covers $Y_{n_{m+1}}$. For this finite cover of $Y_{n_{m+1}}$ choose a $Z_{n_{m+1}}\subset W_{n_m}$ in $\Psi$ such that for each $x\in Y_{n_{m+1}}$ there is a $U\in \mathcal{U}(U_1,\cdots,U_{n_{m+1}})$ for which $Z_{n_{m+1}}(x)\subseteq U$. Do this for all the finitely many possible $(U_1, \cdots, U_{m+1})$, and finally let
$W_{n_{m+1}}$ be an element of $\Psi$ contained in each $Z_{n_{m+1}}$.

Note that for any fixed $m$ the selection principle $\sone(\open,\open_{Y_m})$ holds. Thus for the sequence $(W_n : n \in S_m )$ choose a corresponding sequence $(x_n : n \in S_m )$ of elements of $Y_m$ such that $Y_m$ is covered by the set $\{W_n(x_n):n\in S_m)$. Then the sequence $(W_n(x_n)):n\in\mathbb{N})$ covers $Y$. 

Recursively choose a sequence $(U_n: n \in \mathbb{N})$ as follows: Choose $U_1 \in \mathcal{U}(\emptyset)$ with $W_1(x_1) \subset U_1$, and with $U_1, \cdots, U_m$ chosen, choose $U_{m+1} \in \mathcal{U}(U_1, \cdots, U_m)$ with $W_{n_{m+1}}(x_{n_{m+1}}) \subset U_{m+1}$. These choices are possible on account of the way we chose the sets  $W_{n_m}$. Then
\[
  F(\emptyset), U_1, F(U_1), U_2, F(U_1, U_2), \cdots
\]
is an $F$-play of $\gone(\open, \open_Y)$ which is lost by ONE.

{\flushleft\underline{(3)$\Rightarrow$(4):}} Since $\open_{\Psi}\subset \open$ it follows that if ONE has no winning strategy in $\gone(\open,\open_Y)$, then ONE has no winning strategy in $\gone(\open_{\Psi},\open)$.
\end{proof}

\begin{example}\label{ex:KWilliams}
The equivalence of (1) and (2) in Theorem \ref{th:MSUniform} requires hypotheses on the underlying space. To demonstrate, consider the following example (communicated to me by Kameryn J. Williams): For any uncountable set $X$ there is a quasi-uniformity $\Psi$ on $X$ such that the selection principle $\sone(\open_{\Psi},\open)$ holds, yet the topology generated by $\Psi$ is the discrete topology. For let an uncountable set $X$ be given.  
For each countably infinite subset $F$ of $X$ define
\[
  U_F = \Delta_X \bigcup ((X\setminus F) \times X)
\]
Then let $\Psi$ be the filter generated on $X\times X$ by the family $\{U_F: F\in\lbrack X\rbrack^{\aleph_0}\}$. Then $\Psi$ is a quasi-uniformity on $X$, and the topology $\tau_{\Psi}$ generated on $X$ by $\Psi$ is the set 
\[
  \{V\subseteq X: (\forall x\in V)(\exists U\in\mathcal{U})(U(x)\subseteq V)\}.
\]
To see that for each $x\in X$ the set $\{x\}$ is a member of $\tau_{\Psi}$, note that for any countable set $F\subset X$ with $x\in F$ we have, for $U_F = \Delta_X\bigcup((X\setminus F)\times X)$ that $U_F(x) = \{x\}$.
To see that $(X,\Psi)$ satisfies the selection principle $\sone(\open_{\Psi},\open)$, let a sequence $(U_n:n\in\mathbb{N})$ of elements of $\mathcal{U}$ be given. Note that each $U_n$ contains a set of the form $U_{F_n}$ where $F_n$ is a countably infinite subset of $X$. Choose, for each $n$ an $x_n\not\in \bigcup\{F_m:m\in\mathbb{N}\}$. Then the sequence $(U_n(x_n):n\in\mathbb{N})$ covers $X$. An even stronger fact holds: TWO has a winning strategy in the game $\gone(\open_{\Psi},\open)$: For when ONE plays the open cover $\{U(x):x\in X\}$ for some $U$ in $\Psi$, then there is an $x\in X$ for which $U(x) = X$. Yet, $X$ does not satisfy the property $\sone(\open,\open)$, for take for each $n$ the open cover $\open_n = \{\{x\}:x\in X\}$ of $X$. Since $X$ is uncountable, no countable subset of the set $\{\{x\}:x\in X\}$ covers $X$.
\end{example}

On the other hand, the hypotheses in Theorem \ref{th:MSUniform} are not the optimal hypotheses under which the statements of the theorem are true. For example
\begin{example}\label{ex:MSUniformNotOptimal}
In \cite{MSUniform} Proposition 6 it is shown that there is for each uncountable cardinal $\kappa$ a $\textsf{T}_0$ topological group that has cardinality $\kappa$, does not embed as a closed subgroup into any $\sigma$-compact group, and yet TWO has a winning strategy in the game $\gone(\open_{nbd},\open)$ in this group. Even more, TWO has a winning strategy in the game $\gone(\open,\open)$ in this group. More information about this example is given in Section 8.
\end{example}

Towards connecting covering properties with Ramseyan statements we now introduce a concept the significance of which will be clear soon: Following \cite{GN} we say that an open cover $\mathcal{U}$ of a topological space $(X,\tau)$ is an $\omega$-cover if $X$ is not a member of $\mathcal{U}$, but for each finite subset $F$ of $X$ there is a $U\in\mathcal{U}$ such that $F\subseteq U$. The symbol $\Omega$ denotes the collection of all $\omega$-covers of the space $X$. 
\begin{theorem}\label{thm:OmegaEquiv} Let $(X,\tau)$ be a topological space and let $Y$ be a subset of $X$. The following are equivalent:
\begin{enumerate} 
\item{$(X ,\tau)$ satisfies $\textsf{S}_1(\mathcal{O},\mathcal{O}_Y )$.}
\item{$(X ,\tau)$ satisfies $\textsf{S}_1(\Omega,\mathcal{O}_Y )$.}
\end{enumerate}
\end{theorem}
\begin{proof} (2)$\Rightarrow$(1) requires a proof.
Thus, let a sequence $(\mathcal{U}_n:n\in\mathbb{N})$  of elements of $\mathcal{O}$ be given.

Next, write $\mathbb{N} = \bigcup\{S_n:n\in\mathbb{N}\}$ where each $S_n$ is infinite, and for $m\neq n$ we have $S_m\cap S_n = \emptyset$. Then, for each $n$ define 
\[
   \mathcal{V}_n = \{\bigcup\{U_x:x\in F\}: F\subset S_n \mbox{ nonempty, finite and } U_x\in\mathcal{U}_x\}
\]
Observe that each $\mathcal{V}_n$ is an $\omega$ cover. Applying $\sone(\Omega,\mathcal{O}_Y)$ to the sequence $(\mathcal{V}_n:n\in\mathbb{N})$, choose for each $n$ a $V_n\in\mathcal{V}_n$ such that $\{V_n:n\in\mathbb{N}\}$ is an open cover of $Y$.

Now for each $n$, choose a finite subset $F_n$ of $S_n$ such that $V_n = \cup\{U_x:x\in F_n\}$, where each corresponding $U_x$ is a member of the original open cover $\mathcal{U}_x$ of $X$. Then, for each $m\in F = \bigcup\{F_n:n\in\mathbb{N}\}$ we have a $U_m\in\mathcal{U}_m$ such that the set $\{U_m:m\in F\}$ is a cover of $Y$. By choosing $U_m\in\mathcal{U}_m$ arbitrary when $m\in\mathbb{N}\setminus F$, we find a sequence $(U_n:n\in\mathbb{N})$ that witnesses $\sone(\mathcal{O},\mathcal{O}_Y)$ for the given sequence $(\mathcal{U}_n:n\in\mathbb{N})$.
\end{proof}

At this point it is worth observing that when the topology arises from a uniformity $\Psi$ for which $X$ is not totally bounded, we obtain a specific subset of the collection of $\omega$-covers as follows: For $U\in\Psi$ we define $\Omega(U)$ to be the set $\{\cup_{x\in F} U(x) : F\subset X \mbox{ finite}\}$. When $X$ is not a member of $\Omega(U)$, then $\Omega(U)$ is indeed an $\omega$-cover of $X$. The symbol $\Omega_{\Psi}$ denotes the open $\omega$-covers of $X$ arising in this way from uniformity $\Psi$. In light of the statement of Theorem \ref{thm:OmegaEquiv} one might speculate that for a subspace $Y$ of a uniform space $(X,\Psi)$ the statements $\textsf{S}_1(\mathcal{O}_{\Psi},\mathcal{O}_Y )$ and $\textsf{S}_1(\Omega_{\Psi},\mathcal{O}_Y )$ are equivalent. This, however, is not so.
\begin{example}\label{ex:nbdomegacover} The real line with its standard metric topology satisfies the statement $\sone(\Omega_{nbd},\open)$, but does not satisfy the statement $\sone(\open_{nbd},\open)$. For let a sequence $(\epsilon_n:n\in\mathbb{N})$ of positive real numbers be given. Then for each $n$ the set $\mathcal{U}_n = \{\cup\{(x-\epsilon_n,x+\epsilon_n):x\in F, F\subset\mathbb{R} \mbox{ finite}\}$ is a member of $\Omega_{nbd}$ for the real line. For each $n$, by the compactness of the closed interval $\lbrack-n, n\rbrack$, choose a finite subset $F_n$ of $\mathbb{R}$ such that $\lbrack-n, n\rbrack\subset S_n = \bigcup \{(x-\epsilon_n,x+\epsilon_n):x\in F_n\}$. Note that $S_n\in \mathcal{U}_n$, and the set $\{S_n:n\in\mathbb{N}\}$ is an open cover of the real line. On the other hand, the real line does not have the property $\sone(\open_{nbd},\open)$. 
\end{example}
More information about the differences between $\sone(\Omega_{\Psi},\open_Y)$ and $\sone(\open_{\Psi},\open_Y)$ can be gleaned from \cite{coc11}.

\section{A Theorem of F.P. Ramsey}

We now arrive at the introduction of the other main concept of the paper, partition relations. Partition relations, as well as corresponding symbolic representations, were originally introduced by Erd\"{o}s and Rado \cite{ER} as various Ramseyan statements.  
The following theorem, entering the mathematics literature ten years after the Borel covering property, is the original stimulus for the development of Ramsey theory as a subfield of mathematics: 
\begin{theorem}[Ramsey \cite{Ramsey}]\label{Th:Ramsey}  
For all positive integers $m$ and $k$, and for each function $f:\lbrack \mathbb{N}\rbrack^m\longrightarrow \{1,\; 2,\; \cdots,\;k\}$ there are an infinite subset $M$ of $\mathbb{N}$, and an element $i$ of $\{1,\; 2,\; \cdots,\; k\}$ such that $f$ is constant, of value $i$, on the set $\lbrack M\rbrack^m$.
\end{theorem}
The statement of Theorem \ref{Th:Ramsey} uses notation requiring introduction:  
For an arbitrary set $S$ and for a positive integer $m$ the symbol $\lbrack S\rbrack^m$ denotes the set of all $m$-element subsets of the set $S$. The following notation will also be needed: When $\lambda$ is an infinite cardinal number and $S$ is a set, the symbol $\lbrack S\rbrack^{<\lambda}$ denotes the set consisting of subsets of cardinality less than $\lambda$ of the set $S$, while $\lbrack S\rbrack^{\lambda}$ denotes the set whose elements are the subsets $A$ of $S$ for which $\vert A\vert = \lambda$.

In terms of the notation that was introduced by Erd\"{o}s and collaborators, Theorem \ref{Th:Ramsey} can be stated as follows:
\[
  \mbox{For all positive integers m and k, }\aleph_0\longrightarrow(\aleph_0)^m_k.
\]
More generally, for cardinal numbers $\kappa$ and $\lambda$, the symbolic statement
\[
  \mbox{For positive integers m and k, }\kappa\longrightarrow(\lambda)^m_k.
\]
denotes that for any set $X$ of cardinality $\kappa$, and for any function
\[
  f:\lbrack X\rbrack^m\rightarrow \{1,\;\cdots,\;k\}
\]
there is an $i\in\{1,\;\cdots,\;k\}$ and a subset $Y$ of $X$ such that $\vert Y\vert = \lambda$, and $f$ is constant of value $i$ on the set $\lbrack Y\rbrack^m$.
In this paper we adopt the following more general notation: 
\begin{definition}\label{def:Ramsey}
Let $k$ and $m$ be positive integers. Let $\mathcal{A}$, $\mathcal{B}_1$, $\cdots$ $\mathcal{B}_k$ be families of sets. The symbol
\[
  \mathcal{A}\longrightarrow(\mathcal{B}_1,\cdots,\mathcal{B}_k)^m
\]
denotes that for each $A\in\mathcal{A}$, and for each function $f:\lbrack A\rbrack^m\rightarrow\{1,\cdots,k\}$, there is an $i\in\{1,\cdots,k\}$ and a $B_i\subset A$ such that $B_i\in\mathcal{B}_i$, and $f$ is constant of value $i$ on $\lbrack B_i\rbrack^m$.
\end{definition}

When for $1\le i\le k$ it is the case that $\mathcal{B}_i = \mathcal{B}$, then $ \mathcal{A}\longrightarrow(\mathcal{B}_1,\cdots,\mathcal{B}_k)^m$  is denoted $\mathcal{A}\longrightarrow(\mathcal{B})^m_k$.
If in Definition \ref{def:Ramsey} we take the special instances $\mathcal{A} = \lbrack\mathbb{N}\rbrack^{\aleph_0}$ and $\mathcal{B} = \mathcal{A}$, then the symbol $\mathcal{A}\longrightarrow(\mathcal{B})^m_k$ denotes the instance for $m$ and $k$ of Ramsey's theorem. If, instead, we take $\mathcal{A}$ to be $\Omega$, the set of $\omega$-covers of a topological space, then the following holds:
\begin{lemma}\label{lemma:OmegaisRamsey} Let $(X,\tau)$ be a topological space which has  
infinite $\omega$-covers. Then for each positive integer $k$ the partition relation
\[
  \Omega \longrightarrow(\Omega)^1_k
\]
holds.
\end{lemma}
\begin{proof}
Suppose that on the contrary $\mathcal{U}$ is an $\omega$-cover of $X$, and that there is a partition, say
\[
  \mathcal{U} = \mathcal{U}_1\bigcup \cdots \bigcup\mathcal{U}_k
\]
of $\mathcal{U}$ into $k>1$ pieces such that none of the $\mathcal{U}_i$ is an $\omega$-cover of $X$. For each $1\le i\le k$ choose a finite set $F_i\subset X$ for which there is no $U\in\mathcal{U}_i$ such that $F_i\subset U$. Then the set $F = \bigcup\{F_i:\;1\le i \le k\}$ is a finite subset of $X$, and yet no element of $\mathcal{U}$ covers $F$, contradicting the fact that $\mathcal{U}$ is an $\omega$-cover of $X$.
\end{proof}

Later on in the paper we will need the following consequence of Lemma \ref{lemma:OmegaisRamsey}:
\begin{corollary}\label{cor:bdedomegaisRamsey}
Assume that for each sequence $(\mathcal{U}_n:n\in\naturals)$ of $\omega$-covers of $X$ there is a positive integer $k$ and a sequence $(\mathcal{F}_n:n\in\naturals)$ such that for each $n$ it is the case that $\mathcal{F}_n\subset \mathcal{U}_n$ and $\vert\mathcal{F}\vert\le k$, and $\bigcup\{\mathcal{F}_n:n\in\naturals\}$ is an $\omega$-cover of $X$. Then $\sone(\Omega,\Omega)$ holds. 
\end{corollary}
\begin{proof}
For let a sequence $(\mathcal{U}_n:n\in\naturals)$ of $\omega$-covers of $X$ be given, and fix a positive integer $k$ as in the hypothesis of the Corollary. For each $n$ fix a $\le k$-element set $\mathcal{F}_n$ of $\mathcal{U}_n$ such that 
$\mathcal{F} = \bigcup\{\mathcal{F}_n:n\in\naturals\}$ is an $\omega$-cover of $X$. For each $n$ enumerate $\mathcal{F}_n$ bijectively as $\{F^n_1,\cdots,F^n_{j_n}\}$ where $j_n\le k$.
For $i\le k$ set $\mathcal{G}_i =\{F^n_i:n\in\naturals \mbox{ and }i\le j_n\}$. Then the sets $\mathcal{G}_1,\cdots,\mathcal{G}_k$ are pairwise disjoint, and $\mathcal{G} = \bigcup\{\mathcal{G}_i:i\le k\}$ is an $\omega$-cover. By Lemma \ref{lemma:OmegaisRamsey}, for an $i\le k$ the set $\mathcal{G}_i$ is an $\omega$-cover. But $\mathcal{G}_i$ results from at most one selection from each $\mathcal{U}_n$, confirming $\sone(\Omega,\Omega)$ for the given sequence of $\omega$-covers of $X$. Since the given sequence was an arbitrary sequence of $\omega$-covers of $X$, the corollary follows.
\end{proof}
Recall that a topological space is said to be Lindel\"of if each open cover of the space has a countable subset that is a cover of the space. The following classical result is another theoretical component towards establishing the connection between Ramsey theory and variations of the Borel covering property.
\begin{theorem}[Arkhangelskii, Pytkeev]\label{Th:ArkhPytk}
If $(X,\tau)$ is a topological space for which each finite power is Lindel\"of, then each $\omega$-cover of $X$ has a countable subset that is an $\omega$-cover.
\end{theorem}
A proof of Therorem \ref{Th:ArkhPytk} can be found in Theorem II.1.1 of \cite{Arkhangelskii}. If a space is $\sigma$-compact, it is Lindel\"of in all finite powers, and thus each of its $\omega$-covers has, by the Arkhangel'skii-Pytkeev Theorem, a countable subset that is an $\omega$-cover of $X$.
We now arrive at our first result, Theorem \ref{th:RealsRamsey}, that establishes a Ramseyan equivalence of the Borel covering property. This result was proven in \cite{MSFinPowers}, Theorem 9, for the special context of $\sigma$-compact metric spaces. Its generalization to Theorem \ref{th:RealsRamsey} was briefly explained previously in Theorem 3 of \cite{MSUniform}. For the reader's convenience we provide a proof of Theorem \ref{th:RealsRamsey} here. 
\begin{theorem}\label{th:RealsRamsey}
Let $(X,\Psi)$ be a $\sigma$-compact uniform space and let $Y$ be a subset of $X$. Then the following statements are equivalent:
\begin{enumerate}
\item{The selection principle $\sone(\open,\open_Y)$ holds.} 
\item{For each positive integer $k$ the partition relation $\Omega \rightarrow(\mathcal{O}_Y)^2_k$ holds.}
\end{enumerate}
\end{theorem}
\begin{proof} \underline{$(1)\Rightarrow (2):$} Let $(X,\tau)$ be a $\sigma$-compact uniform space for which the selection principle $\sone(\open,\open_Y)$ holds. Let $\mathcal{U}$ be a given $\omega$-cover of $X$. By Theorem \ref{Th:ArkhPytk} we may assume that $\mathcal{U}$ is countable. Fix an enumeration of $\mathcal{U}$, say $(U_n:n\in\mathbb{N})$. Also, let a coloring
\[
  f:\lbrack \mathcal{U}\rbrack^2 \longrightarrow \{1,\cdots, k\}
\]
be given.

Now recursively construct two sequences $(\mathcal{U}_n:n\in\mathbb{N})$ and $(i_n:n\in\mathbb{N})$ such that
\begin{enumerate}
\item{$\mathcal{U}_1 := \{ U_m: m > 1 \mbox{ and }f (\{ U_1, U_m\}) = i_1\}$ is an $\omega$-cover of $X$ and}
\item{For each $n$,
      $\mathcal{U}_{n+l }:= \{ U_m \in \mathcal{U}_n : m > n + 1 \mbox{ and } f (\{ U_{n+ I}, U_m\}) = i_{n+l}\} $
    is an $\omega$-cover of $Y$.}
\end{enumerate}
To obtain $\mathcal{U}_1$ and $i_1$, observe that $\mathcal{V} = \mathcal{U}\setminus\{U_1\}$ is an $\omega$-cover of $X$. Setting $\mathcal{V}_j = \{U\in\mathcal{V}: f(\{U_1,U\} = j\}$, we obtain a partition
\[
  \mathcal{V} = \mathcal{V}_1\cup\cdots\cup\mathcal{V}_k
\]
of the $\omega$-cover $\mathcal{V}$ into finitely many parts. By Lemma \ref{lemma:OmegaisRamsey} at least one of the parts is an $\omega$-cover. Select $i_1$ so that $\mathcal{V}_{i_1}$ is an $\omega$-cover, and put $\mathcal{U}_1 = \mathcal{V}_{i_1}$. Assuming that the $\omega$-cover $\mathcal{U}_n$ and $i_n$ have been determined, proceed in the same way to obtain the $\omega$-cover $\mathcal{U}_{n+1}$ and $i_{n+1}$ from $\mathcal{U}_n$, using the function $f$.
Observe that for each $n$ we have $\mathcal{U}_{n+1}\subset \mathcal{U}_n$. 

Next, define for $j\in\{1,\cdots,k\}$ the set $\mathcal{W}_j = \{U_n:i_n = j\}$. Then for each $n$ we have the partition
\[
  \mathcal{U}_n = (\mathcal{U}_n\cap\mathcal{W}_1) \cup \cdots \cup (\mathcal{U}_n\cap\mathcal{W}_k).
\]
Applying Lemma \ref{lemma:OmegaisRamsey} to each of these partitions we fix for each $n$ a $j_n\in\{1,\cdots,k\}$ for which $\mathcal{U}_n\cap\mathcal{W}_{j_n}$ is an $\omega$-cover.  Since for each $n$ we have $\mathcal{U}_{n+1}\subset \mathcal{U}_n$, we may assume that the selected $j_n$'s are all of the same value, say $j$. Thus, for each $n$, $\mathcal{U}_n\cap\mathcal{W}_j$ is an $\omega$-cover of $X$.

With the sequence of $\mathcal{U}_n\cap\mathcal{W}_j$ selected, define the following strategy, $F$, for ONE in the game $\gone(\mathcal{O},\mathcal{O}_Y)$ played on $X$. ONE's first move is $F(\emptyset) = \mathcal{U}_1\cap\mathcal{W}_j$. When TWO plays $U_{n_1}\in F(\emptyset)$, ONE responds with $F(U_{n_1}) = \mathcal{U}_{n_1}\cap\mathcal{W}_j$. When TWO responds with a $U_{n_2}\in F(U_{n_1})$, ONE responds with $F(U_{n_1},U_{n_2}) = \mathcal{U}_{n_2}\cap\mathcal{W}_j$, and so on. By hypothesis $\sone(\mathcal{O},\mathcal{O}_Y)$ holds. Since $X$ is $\sigma$-compact, Theorem \ref{th:MSUniform} implies that ONE does not have a winning strategy in the game $\gone(\open,\open_Y)$. Thus, the strategy $F$ just defined for player ONE is not a winning strategy in the game $\gone(\open,\open_Y)$. Choose an $F$-play
\[
  F(\emptyset), U_{n_1}, F(U_{n_1}), U_{n_2}, \cdots, U_{n_m}, F(U_{n_1},\cdots,U_{n_m}), U_{n_{m+1}},\cdots
\]
that is lost by ONE. Then the set $\mathcal{H} = \{U_{n_j}:j\in\mathbb{N}\}$ is an element of $\open_Y$ (that is, a cover of $Y$), and for all $u<v$ we have $f(\{U_{n_u},U_{n_v}\}) = j$. But then $\mathcal{H}\subseteq\mathcal{U}$ witnesses that $\Omega\rightarrow(\open_Y)^2_k$ holds for the function $f$. Since the $\omega$-cover $\mathcal{U}$ of $X$, and the function $f:\lbrack\mathcal{U}\rbrack^2\rightarrow\{1,\cdots,k\}$ were arbitrary, this completes the proof of the implication $(1)\Rightarrow (2)$.

{\flushleft${(2)\Rightarrow(1)}$} Since, by Theorem \ref{thm:OmegaEquiv}, for any topological space and a subspace $Y$ of it the statements $\sone(\open,\open_Y)$ and $\sone(\Omega,\open_Y)$ are equivalent, it suffices to show that if the space satisfies the statement $\Omega\longrightarrow(\open_Y)^2_k$ for each positive integer $k$, then indeed that space satisfies the statement $\sone(\Omega,\open_Y)$.

Thus, let a sequence $(\mathcal{U}_n:n\in\mathbb{N})$ of $\omega$-covers of the space $X$ be given. By the $\sigma$-compactness of $X$ and Theorem \ref{Th:ArkhPytk} we may assume that each $\mathcal{U}_n$ is countable. Fix for each $n$ a repetitionfree enumeration of $\mathcal{U}_n$, say $(U^n_m:m\in\mathbb{N})$. Then define $\mathcal{U} = \{U^1_m\cap U^m_n:m,n\in\mathbb{N}\}$.  Then $\mathcal{U}$ is an $\omega$-cover of $X$, for let a finite subset $F$ of $X$ be given. Since $\mathcal{U}_1$ is an $\omega$-cover of $X$, choose $U^1_m\in\mathcal{U}_1$ with $F\subseteq U^1_m$. Then, since $\mathcal{U}_m$ is an $\omega$-cover of $X$, choose $U^m_n\in\mathcal{U}_m$ with $F\subseteq U^m_n$. Then $F$ is a subset of the element $U^1_m\cap U^m_n$ of $\mathcal{U}$. 

Next define a function $f:\lbrack\mathcal{U}\rbrack^2\rightarrow\{0,1\}$ as follows:
\[
  f(\{U^1_m\cap U^m_n, U^1_k\cap U^k_{\ell}\}) = \left\{ \begin{tabular}{ll}
                                                                                          0 & if $m = k$\\
                                                                                          1 & otherwise
                                                                                          \end{tabular}
                                                                                          \right. 
\]
By the hypothesis that $\Omega\longrightarrow(\open_Y)^2_2$ holds, select a subset $\mathcal{S}$ of $\mathcal{U}$ and an $i\in\{0,1\}$ such that $\mathcal{S}$ is a cover of $Y$, and for any two $A, B\in\mathcal{S}$, we have $f(\{A,B\}) = i$.

If $i=0$, then there is a fixed $k$ such that each element of $\mathcal{V}$ is of the form $U^1_k\cap U^k_m$. It follows that in this case $Y\subseteq U^1_k$. But then to obtain an element of $\open_Y$ that witnesses $\sone(\Omega,\open_Y)$ for the sequence $(\mathcal{U}_n:n\in\mathbb{N})$, select $U^1_k$ from $\mathcal{U}_1$, and arbitrary elements from $\mathcal{U}_m$ when $m>1$.

On the other hand, if $i=1$, then $\mathcal{V}$ is of the form $\{U^1_{m_k}\cap U^{m_k}_{n_k}:k\in\mathbb{N}\}$ where $m_k\neq m_{\ell}$ wheneve $k\neq \ell$. In this case choose elements $V_j\in\mathcal{U}_j$ so that $V_1 = U^1_{m_1}$, $V_{m_k} = U^{m_k}_{n_k}$ when $j=m_k$, and for all other values of $j$, $V_j\in\mathcal{U}_j$ are arbitrarily chosen.

In either case we obtain an element of $\open_Y$.
\end{proof}

With the basic result that in $\sigma$-compact uniformizable spaces Borel's covering property is a Ramseyan property now established, the rest of the paper is dedicated to a deeper analysis of the Ramsey-theoretic aspects of covering properties analogous to Borel's covering property.


\section{More Partition Relations for the Borel Covering Property.}

Let a topological space $(X,\tau)$ and subspace $Y$ be given. We now explore strengthening the partition relation $\Omega\longrightarrow(\open_Y)^2_k$: Two immediate targets are the exponent $2$, and the type of open cover of $Y$ 
that can be obtained in the partition relation. To this end let $\Omega_Y$ denote the set of $\omega$-covers of the subspace $Y$, using sets that are open in $X$. Our first quest is identifying  
the conditions under which the stronger partition relation $\Omega\rightarrow(\Omega_Y)^2_k$ holds. This section encapsulates some of the results obtained in \cite{MSFinPowers}. 

For the finite powers $X^n$ and $Y^n$ with the product topologies, use the following special notation for clarity.\\ 
\begin{center}
\begin{tabular}{cl}
 $\open(n)$          & The collection of open covers of $X^n$ \\
 $\Omega(n)$     & The collection of open $\omega$-covers of $X^n$ \\
$\open_{Y^n}$     & The collection of covers of $Y^n$ by sets open in $X^n$ \\
$\Omega_{Y^n}$ & The collection of $\omega$-covers of $Y^n$ by sets open in $X^n$\\
\end{tabular}
\end{center}

If $(X,\Psi)$ is a $\sigma$-compact uniformizable space then by Theorem \ref{th:MSUniform} the following equivalence holds: For a positive integer $n$, $X^n$ satisfies $\sone(\open(n),\open_{Y^n})$ if, and only if, the partition relation $\Omega(n) \longrightarrow (\open_{Y^n})^2_k$ holds for each positive integer $k$.

The following sequence of lemmas will be used in the main argument. 
\begin{lemma}\label{lemma:omegapowers1}
Let $X$ be a topological space and let $\mathcal{U}$ be an $\omega$-cover for the finite power $X^n$. Then there is an $\omega$-cover $\mathcal{V}$ of  $X$ such that for each $V\in\mathcal{V}$ there is a $U\in\mathcal{U}$ with $V^n\subseteq U$, and $\{V^n:V\in\mathcal{V}\}$ is an $\omega$-cover for $X^n$.
\end{lemma}
A proof of Lemma \ref{lemma:omegapowers1} can be found in Lemma 3.3 of \cite{coc2}. Lemma \ref{lemma:FinPowers}, which is also Lemma 11 of \cite{MSFinPowers}, will be left as an exercise.
\begin{lemma}\label{lemma:FinPowers} For each $n$ let $Y_n$ be a subset of the $\sigma$-compact uniform space $(X_n,\Psi_n)$, and let $\Omega_n$ be the set of $\omega$-covers of $X_n$, and $\open_n$ the set of covers of $Y_n$ by sets open in $X_n$. Let $\open_{\sum_n X_n}$ be the set of open covers of $\sum_n X _n$ and let $\open_{\sum Y_n}$ be the set of covers of $\sum_nY _n$ by sets open in $\sum_n X_n$. If each $X_n$ satisfies $\sone(\Omega_n,\open_n)$, then the uniform space $\sum_n X_n$ satisfies $\sone(\open_{\sum_n X_n},\open_{\sum_n Y_n})$
\end{lemma}
For the convenience of the reader we give the proof of the following lemma.
\begin{lemma}\label{lemma:OneStr}
Let $((X_n,\tau_n):n\in\mathbb{N})$ be a sequence of topological spaces and for each $n$, let $Y_n$ be a subspace of $X_n$. If for each $n$ player ONE does not have a winning strategy in the game $\gone(\open_{X_n},\open_{Y_n})$, then ONE has no winning strategy in the game $\gone(\open_{\Sigma X_n},\open_{\Sigma Y_n})$.
\end{lemma}
\begin{proof}
Observe that open subsets of $\Sigma_nX_n$ are sets of the form $\Sigma_nU_n$ where for each $n$ the set $U_n$ is an open subset of the set $X_n$. Let $F$ be a strategy of player ONE for the game $\gone(\open_{\Sigma X_n},\open_{\Sigma Y_n})$. Choose a well-ordering $\prec$ of the topology of $X = \Sigma_n X_n$. Also, choose a partition of $\naturals$ into infinitely many pairwise disjoint infinite sets, say $\naturals = \cup\{S_n:n\in\naturals\}$. Then for each natural number $k$ fix $m_k$ so that $k\in S_{m_k}$. Further, if $O$ is an open subset of $\Sigma X_m$ and $k$ is a positive integer, then use $O_k$ to denote the term $U_k$ in the representation $O = \Sigma_m U_m$.

Define responses of player TWO to the strategy $F$ as follows:

In the first inning, ONE plays $F(\emptyset)$, an open cover of $\Sigma X_n$. Fix $m_1$ so that $1\in S_{m_1}$. Define a strategy $F_{m_1}$ for ONE of the game $\gone(\open_{X_{m_1}},\open_{Y_{m_1}})$ as follows: 
\[
  F_{m_1}(\emptyset) = \{U_{m_1}: \Sigma U_m \in F(\emptyset)\}
\]
For a response $T$ of TWO to $F_{m_1}$, choose the $\prec$-first $U =\Sigma U_m \in F(\emptyset)$ for which $U_{m_1} = T$, and put $T_1 = U$.

In the next inning ONE plays $F(T_1)$, an open cover of $\Sigma X_n$. Take $m_2$ so that $2\in S_{m_2}$. Put $O_{m_2} = \{U_{m_2}: \Sigma U_m\in F(T_1)\}$. Then define the strategy $F_{m_2}$ for ONE of the game $\gone(\open_{X_{m_2}},\open_{Y_{m_2}})$ as follows:
If $m_2 = m_1$, then $F_{m_2}(T) = \{U_{m_2}: \Sigma U_m\in F(T_1)\}$. Else, $F_{m_2}(\emptyset) = \{U_{m_2}: \Sigma U_m\in F(T_1)\}$. For a response $T$ of TWO in the game $\gone(\open_{X_{m_2}},\open_{Y_{m_2}})$, pick the $\prec$-first $U\in F(T_1)$ with $U_{m_2} = T$, and set $T_2 = U$.

In general, suppose that $k-1$ innings of the game have been played, producing the data
\[
  F(\emptyset), T_1, F(T_1, T_2), T_3,\cdots, F(T_1,\cdots,T_{k-2}), T_{k-1}
\]
and numbers $m_1,\cdots,m_{k-1}$ such that for $1\le i\le k-1$ we have $i \in S_{m_i}$.

Now when ONE plays $F(T_1,\cdots,T_{k-1})$, the $k$-th inning is underway. Then choose the longest possible sequence $i_1<\cdots<i_{j}$ so that
\begin{itemize}
\item{$i_j = k$}
\item{$i_1<\cdots<i_j$ and }
\item{$m_{i_1} = \cdots = m_{i_j}$.}
\end{itemize}
Define the strategy $F_{m_k}$ for ONE of the game $\gone(\open_{X_{m_k}},\open_{Y_{m_k}})$ as follows:
{\flushleft \underline{ If $j=1$:}} Then we proceed as follows: 
\[
  F_{m_k}(\emptyset) = \{U_{m_k}: \Sigma U_m \in F(T_1,\cdots,T_{k-1})\}
\]
For a response $T$ of TWO to $F_{m_k}$, choose the $\prec$-first $U =\Sigma U_m \in F(T_1,\cdots,T_{k-1})$ for which $U_{m_k} = T$, and put $T_k = U$.

{\flushleft \underline{If $j>1$:}} Then we proceed as follows: The strategy $F_{m_k}$ is the same as $F_{m_{i_1}}$, and we define
\[
 F_{m_{i_1}}((T_{i_1})_{m_k},\cdots,(T_{i_{j-1}})_{m_k}) = \{U_{m_k}: U = \Sigma U_n\in F(T_1,\cdots,T_{k-1}) \}
\]
Now for a response $T$ of TWO to the move $F_{m_{i_1}}((T_{i_1})_{m_k},\cdots,(T_{i_{j-1}})_{m_k}) $, choose the $\prec$-first $T_k\in F(T_1,\cdots,T_{k-1})$ so that $(T_k)_{m_k} = T$ and let this $T_k$ be TWO's response to $F(T_1,\cdots,T_{k-1})$.

This process recursively defines for each $k\in\naturals$ the strategy $F_k$ for player ONE of the game $\gone(\open_{X_k},\open_{Y_k})$. Each strategy $F_k$ is by hypothesis not a winning strategy for ONE of the game $\gone(\open_{X_k},\open_{Y_k})$. Thus, for each there is a play lost by ONE. For each $k$ choose a play
\[
  F_k(\emptyset), T^k_1,\; F_k(T^k_1),\; T^k_2,\; \cdots,\; 
\]
lost by ONE in the game $\gone(\open_{X_k},\open_{Y_k})$.

Now consider the following $F$-play of the game $\gone(\open_{\Sigma X_m},\; \open_{\Sigma Y_m})$ obtained as follows (recall the definition of the sequence $(m_k:k\in\naturals)$):

Choose the $\prec$-first $T_1\in F(\emptyset)$ with $T_{m_1} = F_{m_1}(\emptyset)$.  Then choose the $\prec$-first $T_2\in F(T_1)$ so that $(T_2)_{m_2} \in F_{j}(A)$ where $A=\emptyset$ if $m_2\neq m_1$, and $A = T_{m_1}$. Observe that when $m_1 = m_2$, then $(T_2)_{m_2} = T^{m_1}_2$, and when $m_1\neq m_2$, then $(T_2)_{m_2} = T^{m_2}_1$.

With $T_1,\cdots,T_k$ selected, choose the $\prec$-first $T_{k+1} \in F(T_1,\cdots,T_k)$ so that 
\[
   (T_{k+1})_{m_{k+1}} = T^{m_{k+1}}_{i_j}\in F_{m_{k+1}}((T_{i_1})_{m_{k+1}}, \cdots, (T_{i_j})_{m_{k+1}})
\]
 where  $i_j$ is such that $m_{i_1} = \cdots = m_{i_j}$.

It is left to the reader to verify that 
the $F$-play 
\[
  F(\emptyset), T_1,\; F(T_1),\; T_2,\; \cdots,T_k,\; F(T_1,\cdots,T_k),\; \cdots
\]
of the game $\gone(\open_{\Sigma X_m},\open_{\Sigma Y_m})$ that comes about in this way is lost by ONE. 
\end{proof}

\begin{theorem}\label{thm:FinPowers1} Let $Y$ be a subspace of a $\sigma$-compact uniform space $(X,\Psi)$. The following are equivalent:
\begin{enumerate}
\item{For each $n$, $X^n$ satisfies $\sone(\Omega(n),\open_{Y^n})$;}
\item{$X$ satisfies $\sone(\Omega,\Omega_Y)$}
\end{enumerate}
\end{theorem}
\begin{proof} $(1)\Rightarrow(2):$ Let $(\mathcal{U}_n:n\in\mathbb{N})$ be a sequence of $\omega$-covers for the space $X$. Then for each $n$ define
\[
  \mathcal{V}_n = \{U^k: U\in\mathcal{U}_n \mbox{ and }k\in\mathbb{N}\}.
\] 
Then each $\mathcal{V}_n$ is an open cover of the $\sigma$-compact uniformizable space $\sum_m X^m$. By hypothesis (1) each $X^m$ has the property $\sone(\Omega(m),\open_{Y^m})$, and thus by Theorem \ref{thm:OmegaEquiv} each $X^m$ has the property $\sone(\open(m),\open_{Y^m})$. But then by Lemma \ref{lemma:FinPowers}, $\sum_mX^m$ satisfies the property $\sone(\open_{\sum_mX^m},\open_{\sum_mY^m})$. Thus, select for each $m$ a set$ V_m\in\mathcal{V}_m$ so that $\{V_m:m\in\mathbb{N}\}$ is an open cover of $\sum_mY^m$. For each $m$ choose a positive integer $k_m$ and a $U_m\in\mathcal{U}_m$ so that $V_m = U^{k_m}_m$. We claim that $\{U_m:m\in\mathbb{N}\}$ is an $\omega$-cover for $Y$. For let a finite subset $\{x_1,\cdots,x_p\}$ of $Y$ be given. Then the point $(x_1,\cdots,x_p)$ is an element of $Y^p$ and thus of $\sum_mY^m$. Choose an $m$ such that $(x_1,\cdots,x_p)\in V_m$. Then $k_m = p$ and $(x_1,\cdots,x_p)\in U^p_m$, meaning that $\{x_1,\cdots,x_p\}\subset U_m$. This completes the proof of $(1)\Rightarrow(2)$.

{\flushleft{$(2)\Rightarrow(1):$}} Fix a positive integer $n$, and let $(\mathcal{U}_m:m\in\mathbb{N})$ be a sequence of $\omega$-covers of $X^n$. By Lemma \ref{lemma:omegapowers1} choose for each $m$ an $\omega$-cover $\mathcal{V}_m$ of $X$ such that $\{V^n:V\in\mathcal{V}_m\}$ refines $\mathcal{U}_m$ (and is an $\omega$-cover of $X^n$). Applying the hypothesis $\sone(\Omega,\Omega_Y)$ to the sequence $(\mathcal{V}_m:m\in\mathbb{N})$ of $\omega$ covers of $X$, choose for each $m$ a $V_m\in\mathcal{V}_m$ so that $\{V_m:m\in\mathbb{N}\}$ is an $\omega$-cover of $Y$. But then $\{V^n_m:m\in\mathbb{N}\}$ is an $\omega$ cover of $Y^n$ by sets open in $X^n$. For each $m$ choose a $V_m\in\mathcal{V}_m$ so that $U^n_m\subseteq V_m$. Then the set $\{V_m:m\in\mathbb{N}\}$ is an $\omega$-cover of $Y^n$, and witnesses $\sone(\Omega(n),\Omega_{Y^n})$ for the given sequence $(\mathcal{V}_m:m\in\mathbb{N})$ of $\omega$-covers of $X^n$.
\end{proof}
The proof of Theorem \ref{thm:FinPowers1} is a rewrite for uniform spaces of the proof of $(2)\Leftrightarrow(3)$ of Theorem 12 of \cite{MSFinPowers}, originally given for metric spaces. Next we expand the equivalences to include the appropriate game theoretic version.

\begin{theorem}\label{thm:FinPowers2} Let $Y$ be a subspace of a $\sigma$-compact uniform space $(X,\Psi)$. The following are equivalent:
\begin{enumerate}
\item{$X$ satisfies $\sone(\Omega,\Omega_Y)$}
\item{On $X$ ONE has no winning strategy in the game $\gone(\Omega,\Omega_Y)$.}
\end{enumerate}
\end{theorem}
\begin{proof} $(1)\Rightarrow(2):$ Assume that the space $X$ has the property $\sone(\Omega,\Omega_Y)$. Fix a strategy $F$ for ONE in the game $\gone(\Omega,\Omega_Y)$. We must show that $F$ is not a winning strategy.

Define from the strategy $F$ for ONE in the game $\gone(\Omega,\Omega_Y)$ a corresponding strategy $G$ for ONE in the game $\gone(\open_{\sum_nX^n},\open_{\sum_nY^n})$ as follows, as depicted in Figure \ref{fig:GfromF}: Define ONE's first move, $G(\emptyset)$ in the game $\gone(\open_{\sum_nX^n},\open_{\sum_nY^n})$ as the following open cover of $\sum_nX^n$ obtained from $F(\emptyset)$: 
\[
  G(\emptyset) = \{U^k:k\in\mathbb{N}, U\in F(\emptyset)\}.
\]
For a response $T_1\in G(\emptyset)$ by TWO, define $G(T_1)$ as follows: Fix $U_1\in F(\emptyset)$ so that $T_1 = U_1^{k_1}$ for some positive integer $k_1$, and define
\[
 G(T_1) = \{U^k:k\in\mathbb{N}, U\in F(U_1)\},
\]
and so on.

\begin{figure}[h]
\begin{center}
\begin{tabular}{c|l}
\multicolumn{2}{c}{$\gone(\Omega,\Omega_Y)$ on $X$}\\ \hline\hline
\multicolumn{2}{c}{}\\
ONE              & TWO                                     \\ \hline
$F(\emptyset)$   &                                       \\
                 &                                         \\
                 & $U_1$ \\
$F(U_1)$   &                                       \\
                 &                                         \\
                 & $U_2 $ \\
$F(U_1,U_2)$   &                                       \\
                 &                                         \\
                 & $U_3 $ \\
$\vdots$         & $\vdots$ \\
\end{tabular}
\hspace{1in}
\begin{tabular}{c|l}
\multicolumn{2}{c}{$\gone(\open_{\sum_nX^n},\open_{\sum_nY^n})$ on $\sum_nX^n$}                        \\ \hline\hline
\multicolumn{2}{c}{}\\ 
ONE   & TWO                                                 \\ \hline
      &                                                     \\
$G(\emptyset) = \{U^k:U\in F(\emptyset), k\in\mathbb{N}\}$ & $T_1 = U_1^{k_1}$                                      \\
      &                                                     \\
      &                                                     \\
$G(T_1) = \{U^k: U\in F(U_1), k\in\mathbb{N}\}$ & $T_2=U_2^{k_2}$                                    \\
      &                                                     \\
      &                                                     \\
$G(T_1,T_2) = \{U^k: U\in F(U_1,U_2), k\in\mathbb{N}\}$ & $T_3=U_3^{k_3}$                                    \\
$\vdots$         & $\vdots$ \\
\end{tabular}
\end{center}
\caption{Defining the strategy $G$ from the strategy $F$}\label{fig:GfromF}
\end{figure}

Now note that by $(2)\Rightarrow(1)$ of Theorem \ref{thm:FinPowers1} and by Theorem \ref{thm:OmegaEquiv}, for each $n$ the space $X^n$ satisfies $\sone(\open(n),\open_{Y^n})$. Theorem \ref{th:MSUniform} implies that ONE has no winning strategy in the game $\gone(\open(n),\open_{Y^n})$. Lemma \ref{lemma:OneStr} implies that ONE has no winning strategy in the game $\gone(\open_{\sum_nX^n},\open_{\sum_nY^n})$.

It follows that the strategy $G$ just defined is not a winning strategy for ONE in the game $\gone(\open_{\sum_nX^n},\open_{\sum_nY^n})$. Thus, fix a $G$-play lost by ONE, say
\[
  G(\emptyset), T_1, G(T_1), T_2, G(T_1,T_2), \cdots, T_n, G(T_1,T_2,\cdots,T_n),\cdots.
\]
Since this play is lost by ONE, the set $\{T_n:n\in\mathbb{N}\}$ is a cover of $\sum_nY^n$.

Note from the definition of $G$ that there is a corresponding sequence $((U_m,k_m):m\in\mathbb{N})$ such that for each $m$
\begin{enumerate}
\item{$T_m = U_m^{k_m}$}
\item{$U_1\in F(\emptyset)$ and for each $m$, $U_{m+1}\in F(U_1,\cdots,U_m)$.}
\end{enumerate}
Thus, 
\[
  F(\emptyset), U_1, F(U_1), U_2, F(U_1,U_2),\cdots,U_m,F(U_1,\cdots,U_m),\cdots
\]
is a corresponding $F$-play of the game $\gone(\Omega,\Omega_Y)$ on $X$. As in the proof of $(1)\Rightarrow(2)$ of Theorem \ref{thm:FinPowers1}, it follows that $\{U_m:m\in\mathbb{N}\}$ is an $\omega$-cover of $Y$, and thus the corresponding $F$-play is lost by ONE. This completes the proof of $(1)\Rightarrow(2)$.

The implication $(2)\Rightarrow(1)$ follows from earlier remarks.
\end{proof}

The proof of Theorem \ref{thm:FinPowers2} is a rewrite for uniform spaces of the proof of $(3)\Leftrightarrow(4)$ of Theorem 12 of \cite{MSFinPowers}, originally given for metric spaces. 
In the case when $X = Y$, the following result is known:
\begin{theorem}\label{thm:XisYRothberger} For a topological space $(X,\tau)$ the following are equivalent:
\begin{enumerate}
\item{For each natural number $n$ the space $X^n$ has property $\sone(\open,\open)$.}
\item{The space $X$ has the property $\sone(\Omega,\Omega)$}
\item{ONE has no winning strategy in the game $\gone(\Omega,\Omega)$.}
\end{enumerate}
\end{theorem}
Observe that in Theorem \ref{thm:XisYRothberger} the hypothesis that $X$ is $\sigma$-compact is dropped, but $Y=X$ is assumed. The equivalence of (1) and (2) in Theorem \ref{thm:XisYRothberger} was discovered by M. Sakai \cite{Sakai}, while the equivalence of (2) and (3) was proven in Theorem 2 of \cite{coc3}. 

Next we expand the equivalences in Theorem \ref{thm:FinPowers2} to include the appropriate Ramsey theoretic version. The corresponding work for Theorem \ref{thm:XisYRothberger} is done in Section 7. 

\begin{theorem}\label{thm:FinPowers3}
Let $Y$ be a subspace of a $\sigma$-compact uniform space $(X,\Psi)$. The following are equivalent:
\begin{enumerate}
\item{$X$ satisfies $\sone(\Omega,\Omega_Y)$}
\item{For each $k$, $X$ satisfies $\Omega \longrightarrow(\Omega_Y)^2_k$.}
\end{enumerate}
\end{theorem}
\begin{proof}
$(1)\Rightarrow(2):$ Assume that for the subspace $Y$ of $X$ it is true that $\sone(\Omega,\Omega_Y)$ holds. Let $\mathcal{U}$ be a given $\omega$-cover of $X$. Since $X$ is $\sigma$-compact the Arkhangel'skii-Pytkeev theorem implies that $\mathcal{U}$ has a countable subset that is an $\omega$-cover of $X$. We may as well assume that $\mathcal{U}$ is countable, and enumerate it without repetitions as $(U_n:n\in\mathbb{N})$. We now use an argument analogous to that in the proof of Theorem \ref{th:RealsRamsey}.

Let $k$ be a positive integer and let a function $f:\lbrack \mathcal{U}\rbrack^2\rightarrow \{1,\cdots,k\}$ be given.
Recursively construct two sequences $(\mathcal{U}_n:n\in\mathbb{N})$ and $(i_n:n\in\mathbb{N})$ such that
\begin{enumerate}
\item{$\mathcal{U}_1 := \{ U_m: m > 1 \mbox{ and }f (\{ U_1, U_m\}) = i_1\}$ is an $\omega$-cover of $X$ and}
\item{For each $n$,
      $\mathcal{U}_{n+l }:= \{ U_m \in \mathcal{U}_n : m > n + 1 \mbox{ and } f (\{ U_{n+ I}, U_m\}) = i_{n+l}\} $
    is an $\omega$-cover of $Y$.}
\end{enumerate}
For example, to obtain $\mathcal{U}_1$ and $i_1$, note that $\mathcal{V} = \mathcal{U}\setminus\{U_1\}$ is an $\omega$-cover of $X$. Setting $\mathcal{V}_j = \{U\in\mathcal{V}: f(\{U_1,U\} = j\}$ defines a partition
\[
  \mathcal{V} = \mathcal{V}_1\cup\cdots\cup\mathcal{V}_k
\]
of the $\omega$-cover $\mathcal{V}$ into finitely many parts. By Lemma \ref{lemma:OmegaisRamsey} at least one of these parts is an $\omega$-cover. Select a $i_1$ for which $\mathcal{V}_{i_1}$ is an $\omega$-cover, and put $\mathcal{U}_1 = \mathcal{V}_{i_1}$. Assume that the $\omega$-cover $\mathcal{U}_n$ and $i_n$ have been determined, and proceed in the same way to obtain the $\omega$-cover $\mathcal{U}_{n+1}$ and $i_{n+1}$ from $\mathcal{U}_n$, using the function $f$.
Observe that for each $n$ we have $\mathcal{U}_{n+1}\subset \mathcal{U}_n$. 

Next we define for $j\in\{1,\cdots,k\}$ the set $\mathcal{W}_j = \{U_n:i_n = j\}$. Then for each $n$ we have the partition
\[
  \mathcal{U}_n = (\mathcal{U}_n\cap\mathcal{W}_1) \cup \cdots \cup (\mathcal{U}_n\cap\mathcal{W}_k).
\]
Applying Lemma \ref{lemma:OmegaisRamsey} to each of these partitions we fix for each $n$ a $j_n\in\{1,\cdots,k\}$ for which $\mathcal{U}_n\cap\mathcal{W}_{j_n}$ is an $\omega$-cover.  Since for each $n$ we have $\mathcal{U}_{n+1}\subset \mathcal{U}_n$, we may assume that the selected $j_n$'s are all of the same value, say $j$. Thus, for each $n$, $\mathcal{U}_n\cap\mathcal{W}_j$ is an $\omega$-cover of $X$.

With the sequence of $\mathcal{U}_n\cap\mathcal{W}_j$ selected, define the following strategy, $F$, for ONE in the game $\gone(\Omega,\Omega_Y)$ played on $X$. ONE's first move is $F(\emptyset) = \mathcal{U}_1\cap\mathcal{W}_j$. When TWO plays $U_{n_1}\in F(\emptyset)$, ONE responds with $F(U_{n_1}) = \mathcal{U}_{n_1}\cap\mathcal{W}_j$. When TWO responds with a $U_{n_2}\in F(U_{n_1})$, ONE responds with $F(U_{n_1},U_{n_2}) = \mathcal{U}_{n_2}\cap\mathcal{W}_j$, and so on. By hypothesis $\sone(\Omega,\Omega_Y)$ holds. Since $X$ is $\sigma$-compact, Theorem \ref{thm:FinPowers2} implies that ONE does not have a winning strategy in the game $\gone(\Omega,\Omega_Y)$. Thus, the strategy $F$ just defined for player ONE is not a winning strategy in the game $\gone(\Omega,\Omega_Y)$. Choose an $F$-play
\[
  F(\emptyset), U_{n_1}, F(U_{n_1}), U_{n_2}, \cdots, U_{n_m}, F(U_{n_1},\cdots,U_{n_m}), U_{n_{m+1}},\cdots
\]
that is lost by ONE. Then the set $\mathcal{H} = \{U_{n_j}:j\in\mathbb{N}\}$ is an element of $\Omega_Y$ (that is, an $\omega$ cover of $Y$), and for all $u<v$ we have $f(\{U_{n_u},U_{n_v}\}) = j$. But then $\mathcal{H}\subseteq\mathcal{U}$ witnesses that $\Omega\rightarrow(\Omega_Y)^2_k$ holds for the function $f$. Since the $\omega$-cover $\mathcal{U}$ of $X$, and the function $f:\lbrack\mathcal{U}\rbrack^2\rightarrow\{1,\cdots,k\}$ were arbitrary, this completes the proof of the implication $(1)\Rightarrow (2)$.

{\flushleft${(2)\Rightarrow(1)}$} 
Let a sequence $(\mathcal{U}_n:n\in\mathbb{N})$ of $\omega$-covers of the space $X$ be given. By the $\sigma$-compactness of $X$ and Theorem \ref{Th:ArkhPytk} we may assume that each $\mathcal{U}_n$ is countable. Fix for each $n$ a repetitionfree enumeration of $\mathcal{U}_n$, say $(U^n_m:m\in\mathbb{N})$. Then define $\mathcal{U} = \{U^1_m\cap U^m_n:m,n\in\mathbb{N}\}$.  Then $\mathcal{U}$ is an $\omega$-cover of $X$.

Next define a function $f:\lbrack\mathcal{U}\rbrack^2\rightarrow\{0,1\}$ as follows:
\[
  f(\{U^1_m\cap U^m_n, U^1_k\cap U^k_{\ell}\}) = \left\{ \begin{tabular}{ll}
                                                                                          0 & if $m = k$\\
                                                                                          1 & otherwise
                                                                                          \end{tabular}
                                                                                          \right. 
\]
By the hypothesis that $\Omega\longrightarrow(\Omega_Y)^2_2$ holds, select a subset $\mathcal{S}$ of $\mathcal{U}$ and an $i\in\{0,1\}$ such that $\mathcal{S}$ is an $\omega$ cover of $Y$, and for any two $A, B\in\mathcal{S}$, we have $f(\{A,B\}) = i$.

If $i=0$, then there is a fixed $k$ such that each element of $\mathcal{V}$ is of the form $U^1_k\cap U^k_m$. It follows that in this case $Y\subseteq U^1_k$, contradicting the fact that $\mathcal{S}$ is an $\omega$-cover of $Y$. 

Therefore, $i=1$ and $\mathcal{V}$ is of the form $\{U^1_{m_k}\cap U^{m_k}_{n_k}:k\in\mathbb{N}\}$ where $m_k\neq m_{\ell}$ wheneve $k\neq \ell$. In this case choose elements $V_j\in\mathcal{U}_j$ so that $V_1 = U^1_{m_1}$, $V_{m_k} = U^{m_k}_{n_k}$ when $j=m_k$, and for all other values of $j$, $V_j\in\mathcal{U}_j$ are arbitrarily chosen. The set $\{V_j:j\in\naturals\}$ is an element of $\Omega_Y$.
\end{proof}
The proof of Theorem \ref{thm:FinPowers3} is a rewrite for uniform spaces of the proof of $(4)\Leftrightarrow(5)$ of Theorem 12 of \cite{MSFinPowers}, originally given for metric spaces.
We expect that there are also improvements to be made in the exponents appearing in the partition relation $\Omega\longrightarrow(\Omega_Y)^2_k$, but currently have no theoretical evidence for this.
\begin{problem}\label{problem:exponents}
Let $(X,\Psi)$ be a $\sigma$-compact uniformizable space. Is it true that if $\Omega \longrightarrow(\Omega_Y)^2_2$, then for all finite positive integers $m$ and $k$, $\Omega\longrightarrow(\Omega_Y)^m_k$.
\end{problem}
We shall later see that the answer to Problem \ref{problem:exponents} is ``yes" for the special situation when $Y=X$..

\section{Partition Relations and Rothberger's Covering Property $\sone(\open,\open)$}

In \cite{rothberger38} Rothberger pointed out that if a set $X$ of real numbers, endowed with the relative topology inherited from $\mathbb{R}$, satisfies the selection principle $\sone(\open,\open)$, then $X$ has the Borel covering property. As it is independent of \textsf{ZFC} whether the Borel covering property and the Rothberger covering property coincide in separable metric spaces, the converse implication is not provable. In \cite{rothberger41} Rothberger proved that the Continuum Hypothesis implies the existence of a set of real numbers that has Borel's covering property, but does not have Rothberger's covering property. 
Formally, the Rothberger covering property is stronger than the Borel covering property. 

The Rothberger covering property can be characterized in terms of Ramseyan properties that are generally stronger than the Ramseyan properties characterizing the Borel covering property. 
We now describe the known relation between Ramsey theoretic properties and the Rothberger covering property. The preservation of a covering property under finite powers is significantly reflected by the Ramseyan covering properties of the underlying space. Correspondingly the treatment below is divided into two parts according to the behavior of the covering property under finite powers.

From this point on we work with a topological space $(X,\tau)$, and assume that the subspace $Y$ considered before is equal to $X$. Our first goal is the following basic theorem:
\begin{theorem}\label{thm:rothberger} Let $(X,\tau)$ be a regular topological space. The following statements are equivalent:
\begin{enumerate}
\item{$X$ has the covering property $\sone(\open,\open)$.}
\item{ONE has no winning strategy in the game $\gone(\open,\open)$ played on $X$.}
\item{For each positive integer $k$, $\Omega\longrightarrow(\open)^2_k$ holds for $X$.}
\end{enumerate}
\end{theorem}
\begin{proof}
The proof of the fact that (1) implies (2) is Theorem \ref{thm:Pawlikowski}, due to Pawlikowski, \cite{JP}. Note that this implication is stronger than the corresponding one proven in Theorem \ref{th:RealsRamsey} since in that theorem we assumed that the space $X$ is $\sigma$-compact. That (2) implies (3), and (3) implies (1) follows the argument in the proof of Theorem \ref{th:RealsRamsey}.  
\end{proof}

We now consider the corresponding equivalence for the case where we consider whether for a fixed finite power of $X$, that power has the Rothberger covering property. For convenience we introduce the following notation:
Let $(X,\tau)$ be a topological space and let $n$ be a positive integer. Let $\open_n$ denote the collection of open covers $\mathcal{U}$ of $X$ with the property that whenever $F\subset X$ has at most $n$ elements, then there is a set $U\in\mathcal{U}$ for which $F\subseteq U$. Note that when $(X,\tau)$ is a topological space such that for a positive integer $n$ the Tychonoff product space $X^n$ satisfies the property $\sone(\open,\open)$, then all smaller powers of $X$ also have this property. The reason is that the property $\sone(\open,\open)$ is preserved by homeomorphisms, a closed subset of a space inherits the property $\sone(\open,\open)$ from the ambient space, and when $m<n$ are positive integers, then the space $X^m$ embeds as a closed subspace of $X^n$. 

In the proof of Theorem \ref{thm:finitepower} we make use of the following lemma, the proof of which is left to the reader:
\begin{lemma}\label{lemma:omegapowers}
If $\mathcal{U}$ is an (open) $\omega$-cover for the space $X$, then $\{U^n:U\in\mathcal{U}\}$ is a $\omega$-cover for the space $X^n$.
\end{lemma}

\begin{theorem}\label{thm:finitepower} For a topological space $(X,\tau)$ and positive integer $n$, the following are equivalent:
\begin{enumerate}
\item{For each positive integer $k$, $\Omega \longrightarrow(\open_n)^2_k$ holds.}
\item{The $n$-th power of $X$ with the Tychonoff product topology satisfies the property $\sone(\open,\open)$.}
\end{enumerate}
\end{theorem}
\begin{proof}
$(2)\Rightarrow(1):$ Fix a positive integer $n$ and an $\omega$-cover $\mathcal{U}$ of $X$. Let $f:\lbrack\mathcal{U}\rbrack^2 \longrightarrow \{1,\cdots,k\}$ be given.
By Lemma \ref{lemma:omegapowers} the set $\mathcal{V} = \{U^n:U\in\mathcal{U}\}$ is an $\omega$-cover of $X^n$. Since $X^n$ has the property $\sone(\open,\open)$, Theorem \ref{thm:finitepower} implies that for each positive integer $k$, $\Omega\longrightarrow(\open)^2_k$ holds for $X^n$.

Define, from the given $f$, a function $g:\lbrack \mathcal{V}\rbrack^2\longrightarrow \{1,\cdots,k\}$ so that
\[
  g(\{U^n,V^n\}) = f(\{U,V\}).
\]
Apply Theorem \ref{thm:rothberger}, and fix an $i\in\{1,\cdots,k\}$, and a subset $\mathcal{W}$ of $\mathcal{V}$ such that $g$ is constant on $\lbrack\mathcal{W}\rbrack^2$, and $\mathcal{W}$ is a cover of $X^n$. Then consider the set $\mathcal{A}  = \{U\in\mathcal{U}: U^n\in\mathcal{W}\}$. For a subset $F$ of $X$ such that $\vert F\vert = n$, say $F = \{y_1,\cdots,y_n\}$ consider the element $(y_1,\cdots,y_n)$ of $X^n$. For some $U\in \mathcal{A}$ we have $(x_1,\cdots,x_n) \in U^n$, from which it follows that $F\subseteq U$. Then $\mathcal{A}$ witnesses that $(1)$ holds for $X$.

{\flushleft{$(1)\Rightarrow(2):$}}  It suffices to show that the space $X^n$ satisfies the property $\sone(\Omega\,,\open)$. Thus, for each $m$ let $\mathcal{U}_m$ be an $\omega$-cover of $X^n$. For each $m$, applying Lemma \ref{lemma:omegapowers},  choose an $\omega$-cover $\mathcal{V}_m$ of $X$ such that for each $V\in\mathcal{V}_m$ there is a $U\in\mathcal{U}_m$ with $V^n \subseteq U$. Now applying hypothesis (1) to the sequence $(\mathcal{V}_n:n\in\mathbb{N})$ of $\omega$-covers of $X$, choose for each $m$ a $V_m\in\mathcal{V}_m$ such that there is for each subset $F\subset X$ with $\vert F\vert\le n$, there is an $m$ with $F\subseteq V_m$. Next, for each $m$ choose a $U_m\in\mathcal{U}_m$ such that $V_m^n\subseteq U_m$. Then $\{U_m:m\in\mathbb{N}\}$ is an open cover of $X^n$. For let $(x_1,\cdots,x_n)\in X^n$ be given. Then $F = \{x_1,\cdots,x_n\}\subset X$ has at most $n$ elements. Pick $m$ with $F\subseteq V_m$. Then $(x_1,\cdots,x_n)\in V_m^n\subseteq U_m$. This completes the proof of $(1)\Rightarrow(2)$.
\end{proof}

\section{A menu of Ramseyan statements.}

We now introduce several new notational aids to describe the intimate connection between versions of Borel's covering property and variations on Ramsey's theorem. Each notational aid is introduced in a numbered subsection. In the next section we will state a theorem, this paper's main exhibit of the intimate connections between Ramsey theory and a strengthened version of Borel's covering property.

\subsection*{The Ellentuck theorem and $\textsf{E}(\mathcal{A},\mathcal{B})$}  We describe the Ellentuck theorem in a slightly more abstract setting than the usual: 
For $A$ an abstract countably infinite set fix a bijective enumeration $(a_n:n\in\naturals)$ of $A$.
Define for $s$ and $T$ nonempty subsets of $A$:  
\[
  s< T \mbox{ if: }a_n\in s \mbox{ and }a_m\in T \Rightarrow n < m.
\]
With the relation $s< T$ defined, define the Ellentuck topology on $[A]^{\aleph_0}$ as follows: 
For $s\in\,[A]^{<\aleph_0}$ and for $B\in[A]^{\aleph_0}$ use $s< B$ to denote that $s=\emptyset$ or $\max(s) < \min(B)$. For $s < B$  define
$
  [s,B] = \{s\cup C\in \lbrack A\rbrack^{\aleph_0}: \, s < C\subseteq B\}.
$
The family $\{[s,B]:\, s\subset A \mbox{ finite and } s < B\in[A]^{\aleph_0}\}$ forms a basis for a topology on $[A]^{\aleph_0}$.  This is the \emph{Ellentuck topology} on $[A]^{\aleph_0}$ and was introduced in \cite{El} for the special case when $A = \mathbb{N}$, the set of positive integers. If a subset $\mathcal{X}$ of $[A]^{\aleph_0}$ has the topology inherited from $[A]^{\aleph_0}$ endowed with the Ellentuck topology, we speak of ``$\mathcal{X}$ with the Ellentuck topology".  For $B\subseteq A$ and for finite set $s\subseteq A$ the symbol $B|s$ denotes $\{a_n\in B: s< \{a_n\}\}$.

Recall that a subset $N$ of a topological space is \emph{nowhere dense} if there is for each nonempty open set $U$ of the space a nonempty open subset $V\subset U$ such that $N\cap V = \emptyset$. When $N$ is the union of countably many nowhere dense sets, it is said to be \emph{meager}. A subset of the form $(U\setminus M)\bigcup(M\setminus U)$ for some open set $U$ and some meager set $M$ of a topological space is said to have the \emph{Baire property}. In this context the main result in \cite{El} can now be stated as follows:
\begin{theorem}[Ellentuck]\label{galvinprikry} Let $A$ be a countably infinite set with a fixed enumeration defining the relation $<$ among subsets of $A$. 
For a set $R\subset [A]^{\aleph_0}$ the following are equivalent:
\begin{enumerate}
\item{$R$ has the Baire property in the Ellentuck topology.}
\item{For each finite set $s\subset A$ and for each infinite set $S\subset A$ with $s< S$ there is an infinite set $T\subset S$ such that
      either $[s,T]\subset R$, or else $[s,T]\cap R = \emptyset$.}
\end{enumerate}
\end{theorem}

Next we refine the statement of Ellentuck's theorem by placing more constraints on the parameters in the theorem. This refinement motivates the first in the list of upcoming notational aids for this section. For families $\mathcal{A}$ and $\mathcal{B}$ we now define a sequence of statements:
\begin{quote}
 $\egp(\mathcal{A},\mathcal{B})$: For each countably infinite $A\in\mathcal{A}$ and for each set $R\subset [A]^{\aleph_0}\cap\mathcal{B}$ the implication (1)$\Rightarrow$(2) holds, where:
\begin{enumerate}
\item{$R$ has the Baire property in the Ellentuck topology on $[A]^{\aleph_0}\cap\mathcal{B}$.}
\item{For each $S\subset A$ with $S\in\mathcal{A}$ and each finite subset $s$ of $A$, there is an infinite $B\subset S|s$ with $B\in\mathcal{B}$ such that $[s,B]\cap\mathcal{B}\subseteq R$ or $[s,B]\cap\mathcal{B}\cap R = \emptyset$.}
\end{enumerate}
\end{quote}
Thus, for a given countably infinite set $S$, ${\sf E}([S]^{\aleph_0},[S]^{\aleph_0})$ is Ellentuck's theorem.
We shall prove that for a topological space $(X,\tau)$ the statement $\sone(\Omega,\Omega)$ is equivalent to the statement $\egp(\Omega,\Omega)$.

\subsection*{The Galvin-Prikry Theorem  and $\textsf{G}(\mathcal{A},\mathcal{B})$}
Galvin and Prikry proved a precursor of  Theorem \ref{galvinprikry}: If $R$ is a Borel set in the topology inherited from $2^A$ via representing sets by their characteristic functions, then $R$ has property (2) in Theorem \ref{galvinprikry}. Silver and Mathias subsequently gave metamathematical proofs that analytic sets (in the $2^A$-topology) have this property. Theorem \ref{galvinprikry} at once yields all these prior results. The original papers \cite{El} and \cite{G-P} give a nice overview of these facts, and more. The proof of $(1)\Rightarrow(2)$ of Theorem \ref{galvinprikry} is nontrivial but uses only the techniques of Galvin and Prikry \cite{G-P}.

One might wonder to what extent the Galvin-Prikry Theorem, which is a consequence of Ellentuck's Theorem, adapts to the more constrained context in the definition of the notation $\textsf{E}(\mathcal{A},\mathcal{B})$. Towards considering this idea, we give a corresponding abstract formulation of the Galvin-Prikry theorem, and introduce the second notational aid:
\begin{quote}
 ${\sf GP}(\mathcal{A},\mathcal{B})$: For each countably infinite $A\in\mathcal{A}$ and each $R\subset [A]^{\aleph_0}\cap \mathcal{B}$ the implication $(1)\Rightarrow(2)$ holds:
\begin{enumerate}
\item{$R$ is open in the $2^A$ topology on $[A]^{\aleph_0}\cap\mathcal{B}$.}
\item{For each $S\in[A]^{\aleph_0}\cap\mathcal{A}$ there is a set $B\in[S]^{\aleph_0}\cap\mathcal{B}$ such that either $([B]^{\aleph_0}\cap\mathcal{B})\subseteq R$, or else $[B]^{\aleph_0}\cap\mathcal{B}\cap R = \emptyset$.} 
\end{enumerate}
\end{quote}
Thus, ${\sf GP}([\naturals]^{\aleph_0},[\naturals]^{\aleph_0})$ is part of the Galvin-Prikry theorem.

\subsection*{Galvin's generalization of Ramsey's theorem and $\textsf{FG}(\mathcal{A},\mathcal{B})$} An earlier generalization of Ramsey's Theorem, due to Fred Galvin, can also be formulated in the more general context as follows: 
\begin{definition} A subset $\mathcal{S}$ of $[A]^{<\aleph_0}$ is:
\begin{enumerate}
\item{dense if for each $B\in[A]^{\aleph_0}\cap \mathcal{A}$, $\mathcal{S}\cap[B]^{<\aleph_0} \neq \emptyset$.}
\item{thin if no element of $\mathcal{S}$ is an initial segment of another element of $\mathcal{S}$.}
\end{enumerate}
\end{definition}
The following is an abstract formulation of Galvin's generalization of Ramsey's Theorem, announced in \cite{Galvinnotices} and in \cite{G-P} derived from Theorem 1 there, and is the third notational aid introduced in this section:
\begin{quote} ${\sf FG}(\mathcal{A},\mathcal{B})$: 
For each countably infinite $A\in\mathcal{A}$ and for each dense set $\mathcal{S}\subset [A]^{<\aleph_0}$ there is a $B\in[A]^{\aleph_0}\cap\mathcal{B}$ such that each $C\in[B]^{\aleph_0}\cap\mathcal{B}$ has an initial segment in $\mathcal{S}$.
\end{quote}
In this notation Galvin's generalization of Ramsey's theorem reads that ${\sf FG}([\naturals]^{\aleph_0},[\naturals]^{\aleph_0})$.

\subsection*{The Nash-Williams theorem and $\textsf{NW}(\mathcal{A},\mathcal{B})$}
Nash-Williams also discovered a generalization of Ramsey's theorem. The following is an abstract formulation of Nash-Williams' theorem, and introduces the fourth notational aid:
\begin{quote}${\sf NW}(\mathcal{A},\mathcal{B})$:
For each countably infinite $A\in\mathcal{A}$ and for each thin family $\mathcal{T}\subset [A]^{<\aleph_0}$ and for each $n$, and each partition $\mathcal{T} = \mathcal{T}_1 \cup \mathcal{T}_2 \cup \cdots \cup \mathcal{T}_n$ there is a $B\in[A]^{\aleph_0}\cap \mathcal{B}$ and an $i\in\{1,\cdots,n\}$ such that $[B]^{<\aleph_0}\cap\mathcal{T} \subseteq \mathcal{T}_i$.
\end{quote} 
In this notation Nash-Williams' theorem reads that ${\sf NW}([\naturals]^{\aleph_0},[\naturals]^{\aleph_0})$.

\subsection*{Square Bracket Partition Relations}

Square bracket partition relations for cardinal numbers were introduced by Erd\"{o}s, Hajnal and Rado in Section 18 of the paper \cite{EHR}. This partition relation has been explored extensively for other combinatorial structures also, including ultrafilters. More remarks about the work on ultrafilters will be given later in the paper.  
To formulate the square bracket partition relation in sufficient generality for our anticipated application, we introduce the following, the fifth of the notational aids for this section.
\begin{definition}
The symbol
  $\mathcal{A} \longrightarrow \lbrack \mathcal{B} \rbrack^{m}_{k/\le \ell}$
denotes the statement that:
For each $A\in\mathcal{A}$ and for each function 
  $f:\lbrack A\rbrack^m\longrightarrow\{1,\cdots,k\}$
there is a $B\subseteq A$ such that $B\in\mathcal{B}$ such that $\vert\{f(\{x_1,\cdots,x_m\}): \{x_1\cdots,x_m\}\in \lbrack B\rbrack^m\}\vert\le \ell$.
\end{definition}
In the special case when $\ell = k-1$, we use the notation   $\mathcal{A} \longrightarrow \lbrack \mathcal{B} \rbrack^{m}_{k}$ instead of   $\mathcal{A} \longrightarrow \lbrack \mathcal{B} \rbrack^{m}_{k/\le k-1}$. In the special case when $\mathcal{A} = \mathcal{B} = \Omega$, the set of $\omega$-covers of a topological space, the following two results will be used.
\begin{lemma}\label{lemma:subscriptup} Let X be a topological space in which each $\omega$-cover has a countable subset that is an $\omega$-cover. If $\Omega \rightarrow\lbrack \Omega\rbrack^2_3$, then for
every positive integer $k$, $\Omega\rightarrow\lbrack\Omega\rbrack^{2}_{k/\le 2}$.
\end{lemma}
Lemma \ref{lemma:subscriptup} is the previously proven Theorem 1 of \cite{mssquarebracket}, where a proof can be found.
There are several other partition relations, for example polarized partition relations, hybrids of the square bracket partition relation and the polarized partition relation, that we have not represented in this section. We will make some remarks in a separate section about some of the well-known Ramseyan theorems not featured in the main result. The reason for our limited choice is related to the length of this paper and time limitations rather than to the importance of the concepts and ideas.

\section{The Rothberger covering property in all finite powers}

We now focus on a specific strengthening of Borel's covering property, namely the Rothberger covering property in all finite powers. Much of the material in this section is proven in prior papers, including \cite{MSGalvinPrikry}. For the convenience of the reader we provide a proof of the implications depicted in Figure \ref{fig:MainThm}. 
In light of theorems appearing earlier in this paper, one can prove the following for the selection principle $\sone(\Omega,\Omega)$ the following:
\begin{theorem}\label{thm:RothRamsey} Let $(X,\tau)$ be a topological space. The following statements are equivalent:
\begin{enumerate}
\item{$(X,\tau)$ has the property $\sone(\Omega,\Omega)$}
\item{$\Omega\longrightarrow(\Omega)^2_2$ holds for $(X,\tau)$}
\end{enumerate}
\end{theorem}
Instead of giving a direct proof of the equivalence of the statements in Theorem \ref{thm:RothRamsey}, we follow a more roundabout way to illustrate the connection with several significant strengthenings of the classical Ramsey Theorem. Here is the statement of the target theorem:

\begin{theorem}\label{egprothberger} For a topological space $(X,\tau)$ in which each $\omega$ cover has a countable subset that is an $\omega$-cover, the following statements are equivalent:
\begin{enumerate}
\item{$\sone(\Omega,\Omega)$.}
\item{$\egp(\Omega,\Omega)$.}
\item{${\sf GP}(\Omega,\Omega)$.}
\item{${\sf FG}(\Omega,\Omega)$.}
\item{${\sf NW}(\Omega,\Omega)$.}
\item{For all $n$ and $k$, $\Omega\rightarrow(\Omega)^n_k$.}
\item{For all $n$ and $k$, $\Omega\rightarrow\lbrack \Omega\rbrack^n_{k/\le 2}$.}
\item{$\Omega\rightarrow\lbrack\Omega\rbrack^2_3$.}
\item{$\Omega\rightarrow(\Omega,4)^3$.} 
\item{$\Omega\rightarrow(\Omega)^2_2$.}
\item{$\Omega\rightarrow(\Omega,\open)^2$.}
\end{enumerate}
\end{theorem}
 The theorem's proof is organized into several subsections, each treating a specific implication, as depicted in Figure \ref{fig:MainThm}, of Theorem \ref{egprothberger}.

 \SetVertexNormal[Shape      = circle,
                 LineWidth  = 2pt]
\SetUpEdge[lw         = 1.5pt,
           color      = black,
           labelcolor = white,
           labeltext  = red,
           labelstyle = {sloped,draw,text=blue}]

 \tikzset{EdgeStyle/.style={->}}
           
 \begin{figure}[h]
\begin{center}
\begin{tikzpicture}
   \Vertex[x=0 ,y=3]{1}
   \Vertex[x=2.2 ,y=2.2]{2}
   \Vertex[x=3,y=0]{3}
   \Vertex[x=2.2 ,y=-2.2]{4}
   \Vertex[x=0 ,y=-3]{5}
   \Vertex[x=-2.2 ,y=-2.2]{6}
   \Vertex[x=-3 ,y=0]{7}
   \Vertex[x=-2.2 ,y=2.2]{8}
   \Vertex[x=-3.8 ,y=-1.4]{9}
   \Vertex[x=-4.0 ,y=1.5]{10}
   \Vertex[x=-1.5 ,y=3.7]{11}
   
   \Edge[](1)(2)
   \Edge[](2)(3)
   \Edge[](3)(4)
   \Edge[](4)(5)
   \Edge[](5)(6)
   \Edge[](6)(7)
   \Edge[](7)(8)
   \Edge[](8)(1)
   \Edge[](6)(9)
   \Edge[](9)(10)
   \Edge[](10)(11)
   \Edge[](11)(1)
\end{tikzpicture}
\end{center} \caption{The directed edges between two labeled vertices denote the statements between which direct implications are proved in the paper. It is possible to arrange the proof so that the corresponding directed graph denoting the proved implications has a single closed path. }\label{fig:MainThm}
 \end{figure}

\subsection{A proof of the implication $\sone(\Omega,\Omega)\,\Rightarrow\,\egp(\Omega,\Omega)$:}

Assume that $X$ has the property $\sone(\Omega,\Omega)$. Fix a countable $\omega$-cover $A$ of $X$ and fix a set $R\subset [A]^{\aleph_0}\cap\Omega$. 
Also fix a bijective enumeration $(a_n:n\in\naturals)$ of the set $A$. The relation $<$ among subsets of $A$ is interpreted in terms of this enumeration. Recall that sets of the form $[s,C] = \{D: s<C \mbox{ and }s \subset D \subseteq s\cup C\}$ constitute a basis for the Ellentuck topology on $[A]^{\aleph_0}$.

We now introduce terminology from the paper \cite{G-P}:
\begin{definition} For a finite set $s\subset A$ and for $B\in[A|s]^{\aleph_0}\cap\Omega$:
\begin{enumerate}
\item{$B$ accepts $s$ if $[s,B]\cap\Omega\subseteq R$.}
\item{$B$ rejects $s$ if no $C\in[B]^{\aleph_0}\cap\Omega$ accepts $s$.}
\end{enumerate}
\end{definition}
Lemma \ref{acceptreject} will be used without special reference:
\begin{lemma}\label{acceptreject} Let a finite set $s\subset A$ and a set $B\in[A|s]^{\aleph_0}\cap\Omega$ be given:
\begin{enumerate}
\item{$B$ accepts $s$ if, and only if, each $C\in[B]^{\aleph_0}\cap\Omega$ accepts $s$.}
\item{$B$ rejects $s$ if, and only if, each $C\in[B]^{\aleph_0}\cap\Omega$ rejects $s$.}
\end{enumerate}
\end{lemma}

\begin{lemma}\label{decideoneset} For each finite set $s\subset A$, there is a $B\in[A|s]^{\aleph_0}\cap \Omega$ such that $B$ accepts $s$ or $B$ rejects $s$.
\end{lemma}
$\pf$ If $A|s$ does not reject $s$, choose a $B\in[A|s]^{\aleph_0}\cap\Omega$ accepting $s$. $\epf$

\begin{lemma}\label{acceptstrongreject} Let $t\subset A$ be a finite set. Let $B\in[A]^{\aleph_0}\cap\Omega$ be such that for each finite set $s\subset (t\cup B)$, $B|s$ accepts $s$ or $B|s$ rejects $s$. If $B|t$ rejects $t$ then $C =\{u\in B: B|(t\cup\{u\}) \mbox{ rejects } t\cup\{u\}\}$ is a member of $\Omega$. 
\end{lemma}
$\pf$
Suppose not. Then $D = t\cup(B\setminus C)\in\Omega$, and for each $u\in D|t$, $B|(t\cup\{u\})$ accepts $t\cup\{u\}$. Thus for each $u\in D|t$, $D|(t\cup\{u\})$ accepts $t\cup\{u\}$. This means that $[t,D|t] = \cup_{u\in D}[t\cup\{u\},D|(t\cup\{u\})]\subseteq R$, and so $D|t$ accepts $t$. This contradicts Lemma \ref{acceptreject} (2) since $D\in [B]^{\aleph_0}\cap\Omega$ and $B|t$ rejects $t$.
$\epf$

\subsection*{$\omega$-covers accepting or rejecting all finite subsets.}

Recall that
\begin{theorem}\label{coc3gone} For a topological space $Y$ the following are equivalent:
\begin{enumerate}
\item{$Y$ has property $\sone(\Omega,\Omega)$.}
\item{ONE has no winning strategy in $\gone(\Omega,\Omega)$.}
\end{enumerate}
\end{theorem}

\begin{theorem}\label{decidedfin} If $Y$ has property $\sone(\Omega,\Omega)$, then for each finite set $t\subset A$ and for each $B\in[A|t]^{\aleph_0}\cap\Omega$ there is a $C\in[B]^{\aleph_0}\cap\Omega$ such that for each finite set $s\subset t\cup C$, $C|s$ accepts $s$ or $C|s$ rejects $s$.
\end{theorem}
$\pf$ Let $t$ and $B\in[A|t]^{\aleph_0}\cap\Omega_Y$ be given. Define a strategy $\sigma$ for ONE of $\gone(\Omega,\Omega)$ as follows:

Enumerate the set of all subsets of $t$ as $\{t_1,\cdots,t_n\}$. Using Lemma \ref{decideoneset} recursively choose $B_1\supset B_2\supset\cdots \supset B_n$ in $[B]^{\aleph_0}\cap \Omega$ such that for each $i$, $B_i$ accepts $t_i$ or $B_i$ rejects $t_i$. Then define:
\[
  \sigma(\emptyset) = B_n.
\]

If TWO now chooses $T_1\in\sigma(\emptyset)$ then use Lemma \ref{decideoneset} in the same way to choose 
\[
  \sigma(T_1) \in[\sigma(\emptyset)|\{T_1\}]^{\aleph_0}\cap\Omega
\]
such that for each set $F\subset t\cup\{T_1\}$, $\sigma(T_1)$ accepts $F$, or rejects $F$.

When TWO responds with $T_2\in F(T_1)$, enumerate the subsets of $t\cup \{T_1, T_2\}$ as $(t_1,\cdots,t_n)$ say, and choose by Lemma \ref{decideoneset} sets $B_1,\cdots B_n\in[\sigma(T_1)|\{T_2\}]^{\aleph_0}\cap \Omega$ such that $B_j$ accepts $t_j$ or $B_j$ rejects $t_j$ for $1\le j\le n$ and $B_j\subset B_{j-1}$. Finally put 
\[
  \sigma(T_1,T_2) = B_n.
\]
Note that for each finite subset $F$ of $t\cup\{T_1, T_2\}$, $\sigma(T_1,T_2)$ accepts $F$ or rejects it.

It is clear how player ONE's strategy is defined. By Theorem \ref{coc3gone} $\sigma$ is not a winning strategy for ONE. Consider a $\sigma$-play lost by ONE, say
\[
  \sigma(\emptyset), \, T_1,\, \sigma(T_1),\, T_2,\, \sigma(T_1,T_2),\, \cdots,\, T_n,\, \sigma(T_1,\cdots,T_n),\, \cdots
\]
Then $C = t\cup \{T_n:\, n\in\naturals\} \subset B$ is an element of $\Omega$.

We claim that for each finite subset $s$ of $t\cup C$, $C|s$ accepts $s$ or $C|s$ rejects $s$. For consider such a $s$. If $s\subseteq t$, then as $C\subset F(\emptyset)$ and $F(\emptyset)$ accepts or rejects $s$, also $C$ does. If $s\not\subseteq t$, then put $n = \max\{m:T_m\in s\}$. Then $s$ is a subset of $t\cup\{T_1,\cdots,T_n\}$, so that $s$ is accepted or rejected by $\sigma(T_1,\cdots,T_n)$. But $C|s \subseteq \sigma(T_1,\cdots,T_n)$, and so $C|s$ accepts or rejects $s$. 
$\epf$

\subsection*{Completely Ramsey sets}
 
The subset $R$ of $[A]^{\aleph_0}\cap \Omega$ is said to be \emph{completely Ramsey} if there is for each finite set $s\subset A$ and for each $B\in[A|s]^{\aleph_0}\cap\Omega$ a set $C\in[B]^{\aleph_0}\cap\Omega$ such that 
\begin{enumerate}
\item{either $([s,C]\cap\Omega)\subseteq R$,}
\item{or else $([s,C]\cap\Omega)\cap R = \emptyset$.}
\end{enumerate}

\begin{lemma}\label{complramseyunions} If $R$ and $S$ are completely Ramsey subsets of $[A]^{\aleph_0}\cap \Omega$, then so is $R\bigcup S$.
\end{lemma}
$\pf$ Let a finite set $s\subset A$ and $B\in[A|s]^{\aleph_0}\cap\Omega$ be given. Since $R$ is completely Ramsey, choose $C\in[B]^{\aleph_0}\cap\Omega$ such that $([s,C]\cap\Omega\subset R$, or $([s,C]\cap\Omega)\cap R = \emptyset$. If the former hold we are done. In the latter case, since $S$ is completely Ramsey, choose $D\in[C]^{\aleph_0}\cap\Omega$ such that $([s,D]\cap\Omega)\subseteq S$, or $([s,D]\cap\Omega)\cap S = \emptyset$. In either case the proof is complete. $\epf$

The following Lemma is obviously true.
\begin{lemma}\label{complramseycomplements} If $R$ is completely Ramsey, then so is $([A]^{\aleph_0}\cap\Omega)\setminus R$.
\end{lemma}

\begin{corollary}\label{intersections} If $R$ and $S$ are completely Ramsey subsets of $[A]^{\aleph_0}\cap \Omega$, then so is $R\bigcap S$.
\end{corollary}
$\pf$ Lemmas \ref{complramseyunions} and \ref{complramseycomplements}, and De Morgan's laws.$\epf$

\subsection*{Open sets in the Ellentuck topology}

We are still subject to the hypothesis that $X$ satisfies $\sone(\Omega,\Omega)$. 
\begin{lemma}\label{openprecursor} For each finite set $t\subset A$ and for each  $B\in[A|t]^{\aleph_0}\cap\Omega$ such that for each finite subset $F$ of $t\cup B$,  $B|F$ accepts, or rejects $F$ the following holds: For each finite set $s\subset t\cup B$ such that $B|s$ rejects $s$, there is a $C\in[B|s]^{\aleph_0}\cap\Omega$ such that for each finite set $F\subset C$, $C|F$ rejects $s\cup F$.
\end{lemma}
$\pf$ Fix $B$ and $s$ as in the hypotheses. Define a strategy $\sigma$ for ONE in $\gone(\Omega,\Omega)$ as follows: 
By Lemma \ref{acceptstrongreject} 
\[
  \sigma(\emptyset) =\{U\in B:\,s<\{U\}\mbox{ and $B|\{U\}$ rejects }s\cup\{U\}\} \in\Omega_X.
\]
 Notice that $\sigma(\emptyset)$ accepts or rejects each of its finite subsets, it rejects $s$, and for each $U\in \sigma(\emptyset)$, $\sigma(\emptyset)|\{U\}$ rejects $s\cup\{U\}$.

If TWO now chooses $T_1\in \sigma(\emptyset)$, then by  Lemma \ref{acceptstrongreject} 
\[
  \sigma(T_1)=\{U\in \sigma(\emptyset)\setminus\{T_1\}: \sigma(\emptyset)|F \mbox{ rejects }s\cup F \mbox{ for each finite }F\subset\{T_1,U\}\}
\]
 is in $\Omega$. As before, $\sigma(T_1)$ accepts or rejects each of its finite subsets, and for any $U\in \sigma(T_1)$, for each finite subset $F$ of $\{U,T_1\}$, $\sigma(T_1)|F$ rejects $s\cup F$.  

If next TWO chooses $T_2\in \sigma(T_1)$, then by Lemma \ref{acceptstrongreject}
\[
  \sigma(T_1,T_2) = \{U\in \sigma(T_1)\setminus\{T_2\}: \sigma(T_1)|F \mbox{ rejects } s\cup F \mbox{ for any finite } F\subset \{T_1,T_2,U\}\}
\]
is an element of $\Omega$.

Continuing in this way we define a strategy $\sigma$ for ONE in $\gone(\Omega,\Omega)$. Since $X$ satisfies $\sone(\Omega,\Omega)$, $\sigma$ is not a winning strategy for ONE. Consider a $\sigma$-play lost by ONE, say:
\[
 \sigma(\emptyset), \, T_1,\, \sigma(T_1),\, T_2,\, \sigma(T_1,T_2),\, T_3,\, \sigma(T_1,T_2,T_3),\, \cdots
\]
Put $C = \{T_n:n\in\naturals\}$. Then $C\in [B|s]^{\aleph_0}\cap\Omega$. We claim that for each finite set $F\subset C$, $C|F$ rejects $s\cup F$. 

For choose a finite set $F\subset C$. Then $F\cap s = \emptyset$. Fix $n=\max\{m:T_m\in F\}$. Then $C|F \subset \sigma(T_1,\cdots,T_n)$, and the latter rejects $s\cup F$ for all finite subsets $F$ of $\{T_1,\cdots,T_n\}$. Thus $C|F$ rejects $s\cup F$.$\epf$ 

\begin{theorem}\label{opencomplRamsey} If $X$ has property $\sone(\Omega,\Omega)$, then every open subset of $[A]^{\aleph_0}\cap\Omega$ is completely Ramsey.
\end{theorem}
$\pf$ Let $R\subset[A]^{\aleph_0}\cap\Omega$ be open in this subspace. Consider a finite set $s\subset A$ and a $B \in [A|s]^{\aleph_0}\cap\Omega$. Since $(X,d)\models\sone(\Omega,\Omega)$, choose by Theorem \ref{decidedfin} a $C\in[B]^{\aleph_0}\cap\Omega$ such that for each finite set $F\subset(s\cup C)$, $C|F$ accepts or rejects $F$.

If $C$ accepts $s$ then we have $[s,C]\cap\Omega\subseteq R$, and we are done. Thus, assume that $C$ does not accept $s$. Then $C$ rejects $s$, and we choose by Lemma \ref{openprecursor} a $D\in[C|s]^{\aleph_0}\cap\Omega$ such that for each finite subset $F$ of $D$, $D|F$ rejects $s\cup F$.

We claim that $([s,D]\cap\Omega)\cap R = \emptyset$. For suppose not. Choose $E\in[s,D] \cap\Omega \cap R$. Since $R$ is open, choose an Ellentuck neighborhood of $E$ contained in $R$, say $[t,K]\cap \Omega$. Then we have $s\subset E \subset s\cup D$ and $t\subset E \subset t\cup K$. But then $s\cup t\subset E \subset t \cup K$ and $[s\cup t,K|s]\subset R$, whence also $[s\cup t,E|(s\cup t)]\subset R$. But then $E|(s\cup t)$ accepts $s\cup t$ where $t$ is a finite subset of $s\cup D$, and $E|(s\cup t)\subset D|t$, and $D|t$ rejects $s\cup t$, a contradiction. $\epf$

\subsection*{Meager subsets in the Ellentuck topology}

If the subset $R$ of $[A]^{\aleph_0}\cap\Omega$ is nowhere dense in the topology, then for each $B\in[A]^{\aleph_0}\cap\Omega$ and for each finite set $s\subset A$, $B|s$ rejects $s$. We now examine the meager subsets of $[A]^{\aleph_0}\cap \Omega$.

\begin{lemma}\label{nwdfin} If $R$ is nowhere dense, then there is for each $B\in[A]^{\aleph_0}\cap\Omega$ and each finite set $t\subset A$ a set $C\in[B|t]^{\aleph_0}\cap\Omega$ such that for each finite set $s\subset t\cup C$, $C|s$ rejects $s$.
\end{lemma}
$\pf$ Since $R$ is nowhere dense, no $\omega$-cover contained in $A$ can accept a finite set. Thus each $\omega$-cover contained in $A$ rejects each finite subset of $A$. $\epf$

\begin{lemma}\label{clnwdcomplramsey} Assume $\sone(\Omega,\Omega)$. 
If $R$ is a closed nowhere dense subset of $[A]^{\aleph_0}\cap\Omega$ then there is for each finite subset $s\subset A$ and for each $B\in[A|s]^{\aleph_0}\cap\Omega$ a $C\in[B]^{\aleph_0}\cap\Omega$ such that $[s,C]\cap R = \emptyset$. 
\end{lemma}
$\pf$ First, note that closed nowhere dense subsets are complements of open dense sets. By Theorem \ref{opencomplRamsey}, each open set is completely Ramsey. By Lemma \ref{complramseycomplements} each closed, nowhere dense set is completely Ramsey. By Lemma \ref{nwdfin} the rest of the statement follows.
$\epf$

By taking closures, the preceding lemma implies:
\begin{corollary}\label{nwdcomplramsey} Assume $\sone(\Omega,\Omega)$. If $R$ is a nowhere dense subset of $[A]^{\aleph_0}\cap\Omega$ then there is for each finite subset $s\subset A$ and for each $B\in[A|s]^{\aleph_0}\cap\Omega$ a $C\in[B]^{\aleph_0}\cap\Omega$ such that $[s,C]\cap R = \emptyset$. 
\end{corollary}

And now we prove:
\begin{theorem}\label{meagerellentuck} Assume $\sone(\Omega,\Omega)$. For a subset $N$ of $[A]^{\aleph_0}\cap\Omega$ the following are equivalent:
\begin{enumerate}
\item{$N$ is nowhere dense.}
\item{$N$ is meager.}
\end{enumerate}
\end{theorem}
$\pf$ We must show that (2)$\Rightarrow$(1). Thus, assume that $N$ is meager and write $N=\bigcup_{n\in\naturals}N_n$, where for each $n$ we have $N_n\subseteq N_{n+1}$, and $N_n$ is nowhere dense in $[A]^{\aleph_0}\cap \Omega$. Consider any basic open set $[s,B]$ of $[A]^{\aleph_0}\cap\Omega$. 
Define a strategy $\sigma$ for ONE in the game $\gone(\Omega,\Omega)$ as follows: 

Since $N_1$ is nowhere dense, choose by Corollary \ref{nwdcomplramsey} an $O_1\in[B]^{\aleph_0}\cap \Omega$ with $[s,O_1]\cap N_1 = \emptyset$. Define $\sigma(\emptyset) = O_1$. 

When TWO chooses $T_1\in \sigma(\emptyset)$ choose by Corollary \ref{nwdcomplramsey} an $O_2\in[\sigma(\emptyset)|\{T_1\}]^{\aleph_0}\cap\Omega_X$ with $[s,O_2]\cap N_2 = \emptyset$, and define $\sigma(T_1) = O_2$.

Now when TWO chooses $T_2\in \sigma(T_1)$, find by Corollary \ref{nwdcomplramsey} an $O_3\in[\sigma(T_1)|\{T_2\}]^{\aleph_0}\cap\Omega_X$ with $[s,O_3]\cap N_3 = \emptyset$, and define $\sigma(T_1,T_2) = O_3$.

It is clear how to define ONE's strategy $\sigma$. By Theorem \ref{coc3gone} $F$ is not a winning strategy for ONE. Consider a play
\[
  \sigma(\emptyset),\, T_1,\, \sigma(T_1),\, T_2,\, \cdots,\, \sigma(T_1,\cdots,T_n),\, T_{n+1},\, \cdots
\] 
lost by ONE. Put $C = \{T_n:n\in\naturals\}$. Then $C\in[B]^{\aleph_0}\cap\Omega$. Observe that by the definition of $\sigma$ we have for each $k$ and each finite set $F\subset \{T_1,\cdots,T_k\}$ that $[s\cup F,\sigma(T_1,\cdots,T_k)]\cap N_k = \emptyset$.\\
{\flushleft{\bf Claim:}} $[s,C]\cap N = \emptyset$. \\
For suppose that instead $[s,C]\cap N \neq \emptyset$. Choose $V\in [s,C]\cap N$, and then choose $m$ so that $V\in N_m$. Choose the least $k>m$ with $T_k\in V| s$. This is possible because $s$ is finite. Observe also that $s\subseteq V\subseteq s\cup C = s\cup\{T_j:j\in\naturals\}$. Put $F = V\cap \{T_1,\cdots,T_k\}$. Thus we have that $[s\cup F,V|F]\cap N_k \neq \emptyset$, which contradicts the fact that  $V|F\subset \sigma(T_1,\cdots,T_k)$, and $[s\cup F,\sigma(T_1,\cdots,T_k)]\cap N_k = \emptyset$. This completes the proof of the claim.
$\epf$

Using Lemmas \ref{complramseyunions} and \ref{complramseycomplements} and Corollary \ref{intersections} we have:
\begin{theorem}\label{ellentuck} Suppose $X$ satisfies $\sone(\Omega,\Omega)$. Then for each $A\in\Omega$, every subset of $[A]^{\aleph_0}\cap\Omega$ which has the Baire property is completely Ramsey.
\end{theorem}

\subsection{A proof of the implication $\egp(\Omega,\Omega)\,\Rightarrow\,\textsf{GP}(\Omega,\Omega)$:}

Note that a set open in the $2^{\naturals}$ topology is also open in the Ellentuck topology. The implication $(2)\Rightarrow(3)$ of Theorem \ref{egprothberger} follows from this remark. 

\subsection{A proof of the implication $\textsf{GP}(\Omega,\Omega)\,\Rightarrow\,\textsf{FG}(\Omega,\Omega)$:}

\begin{lemma}\label{densetoinitsegm} Assume ${\sf GP}(\Omega,\Omega)$. Then ${\sf FG}(\Omega,\Omega)$ holds.
\end{lemma}
$\pf$ Let $\mathcal{S}\subset[A]^{<\aleph_0}$ be dense and define $\mathcal{I}$ to be the set $\{D\in[A]^{\aleph_0}\cap\Omega_X:\, D \mbox{ has an initial segment in }\mathcal{S}\}$. Then we have: 
\[
  \mathcal{I} = \cup\{[s,D|s]:s\in \mathcal{S},\, D\in[A]^{\aleph_0}\cap\Omega_X \mbox{ and s an initial segment of }D\}
\] 
is a $2^{\naturals}$-open subset of $[A]^{\aleph_0}\cap\Omega$. Choose a $B\in[A]^{\aleph_0}\cap\Omega$ such that $[B]^{\aleph_0}\cap\Omega\subset\mathcal{I}$, or $[B]^{\aleph_0}\cap\Omega\cap\mathcal{I} = \emptyset$. But the second alternative implies the contradiction that $[B]^{<\aleph_0}\cap\mathcal{S} = \emptyset$.  It follows that the first alternative holds. $\epf$

\subsection{A proof of the implication $\textsf{FG}(\Omega,\Omega)\,\Rightarrow\,\textsf{NW}(\Omega,\Omega)$:}

\begin{theorem}\label{thintohomogeneous} Assume ${\sf FG}(\Omega,\Omega)$. Then ${\sf NW}(\Omega,\Omega)$ holds. 
\end{theorem}
$\pf$ Fix a thin family $\mathcal{T}\subset [A]^{<\aleph_0}$ and positive integer $n$, and a partition $\mathcal{T} = \mathcal{T}_1 \cup \mathcal{T}_2 \cup \cdots \cup \mathcal{T}_n$. 
We may assume $n=2$. If $\mathcal{T}_1$ is not dense, we can choose $B\in[A]^{\aleph_0}\cap \Omega$ such that $[B]^{<\aleph_0}\cap\mathcal{T} \subseteq \mathcal{T}_2$. Thus, assume $\mathcal{T}_1$ is dense. Choose, by the hypothesis, 
a $B\in[A]^{\aleph_0}\cap\Omega$ such that for each $C\in[B]^{\aleph_0}\cap\Omega$, some initial segment of $C$ is in $\mathcal{T}_1$.

Consider any $s\in\mathcal{T}\cap[B]^{<\aleph_0}$, and put $D = s\cup (B|s)$. Then $s$ is an initial segment of $D$, and $D\in[B]^{\aleph_0}\cap\Omega$, and so some initial segment of $D$, say $t$, is in $\mathcal{T}_1$. Since both $t$ and $s$ are initial segments of $D$ and are both in $\mathcal{T}$, and since $\mathcal{T}$ is thin, we have $s=t$, and so $s\in\mathcal{T}_1$. Consequently we have $[B]^{<\aleph_0}\cap\mathcal{T}\subseteq \mathcal{T}_1$.
$\epf$

\subsection{A proof of the implication $\textsf{NW}(\Omega,\Omega)\,\Rightarrow\,(\forall n)(\forall k)(\Omega\longrightarrow (\Omega)^n_k)$}

\begin{theorem}\label{ramsey}  Assume that ${\sf NW}(\Omega,\Omega)$ holds. Then: For each $n$ and $k$ we have $\Omega\rightarrow(\Omega)^n_k$.
\end{theorem}
$\pf$ Let $A\in\Omega$ be countable. Let positive integers $n$ and $k$ be given. Put $\mathcal{T}=[A]^n$. Then $\mathcal{T}$ is thin. Apply the hypothesis.
$\epf$

\subsection{A proof of the implication $(\forall n)(\forall k)(\Omega\longrightarrow (\Omega)^n_k\,\Rightarrow\,(\forall n)(\forall k)(\Omega\longrightarrow \lbrack\Omega\rbrack^n_{k/\le 2}$)}

Examination of the definitions reveals that this implication is evidently true.

\subsection{A proof of the implication $\Omega\longrightarrow \lbrack\Omega\rbrack^2_3\,\Rightarrow\,\sone(\Omega,\Omega)$}

\begin{proof}
Assume that $\Omega\longrightarrow \lbrack\Omega\rbrack^2_3$ holds. By Lemma \ref{lemma:subscriptup}  $\Omega \longrightarrow \lbrack \Omega\rbrack^2_{5/\le2}$ holds. 
Let $(\mathcal{U}_n: n\in\mathbb{N})$ be a sequence of $\omega$-covers of $X$. By hypothesis, each $\omega$-cover of $X$ has a countable subset which is still an $\omega$-cover of $X$. Thus we may assume that each of the $\mathcal{U}_n$s is countable. For each $n$, enumerate $\mathcal{U}_n$ bijectively as
$(U^n_m: m\in\mathbb{N})$. Define a new $\omega$-cover $\mathcal{V}$ by $\{U^1_k \cap U^k_n\cap U^n_{\ell}: k, n, \ell\in\mathbb{N}\} \setminus \{\emptyset\}$.

For each element $V$ of $\mathcal{V}$ choose a representation $V = U^1_k\cap U^k_n \cap U^n_{\ell}$. The term $U^k_n$ will be called the middle term of the representation, while the term $U^n_{\ell}$ will be called the end-term of the representation. Define a partition $f : \lbrack \mathcal{V}\rbrack^2 \longrightarrow \{0, 1, 2, 3, 4\}$ as follows: (in the displayed formula we are assuming that the two arguments of $f$ are
listed so that $(k_1, n_1)$ lexicographically precedes, or is equal to $(k_2, n_2)$).
\[
  f(\{U^1_{k_1}\cap U^{k_1}_{n_1}\cap U^{n_1}_{\ell_1},U^1_{k_2}\cap U^{k_2}_{n_2}\cap U^{n_2}_{\ell_2} \}) = \left\{\begin{array}{lcl}
0 & \mbox{if} & k_1 = k_2 \mbox{ and } n_1 = n_2, \\
1 & \mbox{ if } & k_1 = k_2 \mbox{ and } n_1 < n_2, \\
2 & \mbox{ if } & k_1 < k_2 \mbox{ and } n_1 < n_2, \\
3 & \mbox{ if } & k_1 < k_2 \mbox{ and } n_1 > n_2, \\
4 & \mbox{ if } & k_1 < k_2 \mbox{ and } n_1 = n_2
\end{array}\right.
\]
Choose an $\omega$-cover $\mathcal{W} \subseteq \mathcal{V}$ on which f is at most two-valued. List $\mathcal{W}$ as
$(U^1_{k_1}\cap U^{k_1}_{n_1}\cap U^{n_1}_{\ell_1}, \; U^1_{k_2}\cap U^{k_2}_{n_2}\cap U^{n_2}_{\ell_2},\;\cdots)$ according to the lexicographic order of the triples $(k_i,n_i,\ell_i)$ which occur in the chosen
representations of elements of $\mathcal{W}$. There are four main cases to be considered.

{\flushleft\underline{Case 1: $f$ does not have the values $0$ or $1$ on $\lbrack\mathcal{W}\rbrack^2$.}}  In this case we have $k_1 < k_2 <  \cdots < k_n < \cdots$
and so if we set $V_{k_i} = U^{k_i}_{n_i}$, and for $m\not\in\{k_i: i = 1,2,3,\cdots\}$ we choose $V_m \in \mathcal{U}_m$ arbitrarily, then the sequence $(V_m : m = 1, 2, 3,\cdots)$ constitutes an $\omega$-cover of $X$
and for each m we have $V_m\in \mathcal{U}_m$.

{\flushleft\underline{Case 2: $\{0,\;1\}\cap f\lbrack\lbrack\mathcal{W}\rbrack^2\rbrack = \{1\}$}.}

{\flushleft{\underline{Subcase 2.1: $f\lbrack\lbrack W\rbrack^2\rbrack \subseteq \{1, 2\}$}}}. In this case, we see that $n_i \neq n_j$ whenever $i \neq j$. But then $\mathcal{W}$ refines the sequence $(U^{n_i}_{\ell_i}: i = 1, 2, 3, \cdots)$, whence the latter constitutes an $\omega$-cover of $X$. This sequence can then be augmented to one of the form $(U_n:n < \omega)$ where for each $n$, $U_n \in \mathcal{U}_n$, and still $\{U_n:n\in\naturals\}$ is an $\omega$-cover of $X$.\\

{\flushleft{\underline{Subcase 2.2: $f\lbrack\lbrack W\rbrack^2\rbrack \subseteq \{1, 3\}$}}}.
This case doesn't occur. To see this, first observe that 3 cannot be attained with only finitely many different $k$'s, since then $\mathcal{W}$ would be a refinement of a finite subset
of $\mathcal{U}_1$, and thus not an $\omega$-cover. But for Subcase 2.2, value 3 also cannot be attained with infinitely many $k$'s. To see this, list the subscripts of the $U^1_k$'s occurring in the representations of elements of $\mathcal{W}$ monotonically, say
\[
   k_1 \le k_2 \le \cdots \le k_j < \cdots.
\]
If 3 occurs with infinitely many $k$'s, then it happens infinitely often that $k_j < k_{j+1}$. Define $i_1 = 1$ and $i_{n+1} > i_n$, to be minimal such that $k_{i_{n+1}} > k_{i_n}$. Look at the
subset $\{U^1_{k_{ i_j}}\cap U^{k_{i_j}}_{n_{i_j}} \cap U^{n_{i_j}}_{\ell_{i_j}}:1, 2, 3, \cdots \}$ of $\mathcal{W}$.. It is homogeneous of color 3, and thus we find
$n_{i_l} > n_{i_2} > \cdots > n_{i_j} > \cdots$, 
an infinite descending sequence of natural numbers, a contradiction.

{\flushleft{\underline{Subcase 2.3: $f\lbrack\lbrack W\rbrack^2\rbrack \subseteq \{1, 4\}$}}}.
If the value 4 is taken with only finitely many k's,
then $\mathcal{W}$ would be a refinement of a finite subset of $\mathcal{U}_1$, and thus not an $\omega$-cover of
$X$. If the value 1 is taken with only finitely many k's, then all but finitely many
of the middle terms $U^{k_i}_{n_i}$ could be assigned to distinct ones of the covers $\mathcal{U}_n$, and
would constitute an $\omega$-cover, in which case we would be done. Thus we must treat
the case where each of 1 and 4 is attained with infinitely many $k$'s. Then
\[
  (k_1, n_1 ), (k_2, n_2), \cdots, (k_r, n_r), \cdots
\]
  forms (in the lexicographic order) a strictly increasing sequence such that for each
$i$, either $k_i < k_{i+l}$ and $n_i = n_{i+1}$ (in which case $f(\{U^1_{k_{i}}\cap U^{k_{i}}_{n_{i}} \cap U^{n_{i}}_{\ell_{i}},U^1_{k_{i+1}}\cap U^{k_{i+1}}_{n_{i+1}} \cap U^{n_{i+1}}_{\ell_{i}+1},\}) = 4$), or else $k_i = k_{i+1}$ and $n_i < n_{i+1}$ (in which case $f(\{U^1_{k_{i}}\cap U^{k_{i}}_{n_{i}} \cap U^{n_{i}}_{\ell_{i}},U^1_{k_{i+1}}\cap U^{k_{i+1}}_{n_{i+1}} \cap U^{n_{i+1}}_{\ell_{i}+1},\}) = 1$). 
  
  Consider the two sequences
\[
  (U^{k_i}_{n_i}: i = 1, 2,3,\cdots \mbox{ and } k_{i-1} < k_i \mbox{ or } k_i < k_{i+1})
\]  
which consists of certain middle terms of the three-set intersections composing the elements of $\mathcal{W}$, and
\[
  (U^{n_i}_{\ell_i}: i = 1, 2,3,\cdots \mbox{ and } n_{i-1} < n_i \mbox{ or } n_i < n_{i+1})
\]  
which consists of certain end terms of the three-set intersections composing the elements of $\mathcal{W}$. Since $\mathcal{W}$ refines the totality of sets belonging to these two sequences,
these two sequences constitute an $\omega$-cover of $X$. For each $m$ the set $\{i: k_i = m \mbox{ or } n_i = m\}$ has at most four elements.

Thus this $\omega$-cover can be partitioned into four-element sets each of which could be assigned to distinct terms of the original sequence $(\mathcal{U}_n: n < \omega)$. Being an
$\omega$-cover, we can find a new $\omega$-cover by selecting one term per each of these four-element sets (i.e., $\Omega\longrightarrow(\Omega)^1_4$ holds). In this way we find a selector for the original
sequence of $\omega$-covers in such a way that the selector is also an $\omega$-cover.

{\flushleft{\underline{ Case 3. $\{0, 1\} \cap f\lbrack\lbrack \mathcal{W}\rbrack^2\rbrack = \{0\}$ .}}}

{\flushleft{\underline {Subcase3 .1. $f\lbrack\lbrack\mathcal{W}\rbrack^2\rbrack \subseteq \{0,2\}$.}}}

If the value 2 occurs at only a finite number of distinct $k$'s, then $\mathcal{W}$ is a refinement of a finite subset of $\mathcal{U}_1$, and thus not an $\omega$-cover of $X$. So, the value $2$ is achieved at
infinitely many $k$'s. Each time it is achieved and only then, both $k_i$ and $n_i$ increase in value. Let $i_l < i_2 < \cdots$ be such that
\begin{itemize}
\item{$i_l = 1$,}
\item{for each $j$, if $i_j \le t < i_{j+1}$, then $k_{i_j} = k_t$ and $n_{i_j} = n_t$,}
\item{for each $j$, $k_{i_j} < k_{i_{j +1}}$, and}
\item{$\{(k_{I_j} ,n_{i_j}): j = 1, 2,3, \cdots\} =\{(k_j ,n_j.): j = 1, 2,3, \cdots\}$}
\end{itemize}
For each $j$ put $V_{k_{i_j}} = U^{k_{i_j}}_{n_{i_j}}$ and for $n$ not in $\{k_{i_j}: j = 1, 2, 3,\cdots\}$, choose $V_n$ from $\mathcal{U}_n$ arbitrarily. 
Then for each $n$, $V_n \in \mathcal{U}_n$, and $\mathcal{W}$ is a refinement of $\{V_n: n = 1, 2, 3, \cdots\}$, and thus the latter is an $\omega$-cover of $X$.

{\flushleft\underline{Subcase3 .2. $f\lbrack\lbrack \mathcal{W}]\rbrack^2\rbrack\subseteq \{0,3\}$.}} 
The value 3 can be attained at only a finite number of distinct k's, lest we have an infinite descending sequence of natural numbers.
But then $\mathcal{W}$ is a refinement of a finite subset of $\mathcal{U}_1$, and so not an $\omega$-cover of $X$. It follows that this case doesn't occur. 

{\flushleft\underline{Subcase 3.3. $f\lbrack\lbrack \mathcal{W}\rbrack^2\rbrack \subseteq \{0, 4\}$.}}
 The value 4 must be attained at an infinite number
of distinct $k$'s, else $\mathcal{W}$ would be a refinement of a finite subset of $\mathcal{U}_1$, hence not an $\omega$-cover of $X$. But then an argument as in Subcase 3.1 shows that there is a
sequence $(U_n: n \in\mathbb{N})$ such that for each $n$ $U_n \in \mathcal{U}_n$, and $\{U_n: n \in\mathbb{N}\}$ is an $\omega$-cover of $X$.
 
{\flushleft\underline{Case 4. $f\lbrack\lbrack \mathcal{W}\rbrack^2\rbrack \subseteq \{0, 1\}$.}}
Then there is a fixed $k$ such that each element of $\mathcal{W}$ is a subset of $U^1_k$. Since
$X \neq U^1_k$ it follows that $\mathcal{W}$ doesn't even cover $X$. Consequently, this case doesn't
occur.
\end{proof}

\subsection{A proof of the implication $\Omega\longrightarrow (\Omega,4)^3_2\,\Rightarrow\,(\Omega\longrightarrow (\Omega)^2_2$}

The proof of this implication is modeled after the proof of (iii)$\Rightarrow$(i) of Theorem 2.1 of \cite{BT}. Let $\mathcal{U}$ be a countable $\omega$-cover of $X$, enumerated bijectively as $(U_n:n\in\naturals)$. Assume that the partition relation 
$\Omega\longrightarrow (\Omega,4)^3_2$ holds, and let a function $f:\lbrack\mathcal{U}\rbrack^2\longrightarrow\{1,\;2\}$ be given.

Define a new function $g:\lbrack\mathcal{U}\rbrack^3\longrightarrow\{1,2\}$ as follows: Let $\{U_{n_1},\; U_{n_2},\; U_{n_3}\}$ be an element of $\lbrack\mathcal{U}\rbrack^3$, written so that $n_1<n_2<n_3$. Define
\[
  g(\{U_{n_1},\;U_{n_2},\;U_{n_3}\}) = \left\{
                                                               \begin{tabular}{ll}
                                                                 2 & if $f(\{U_{n_1},U_{n_2}\}) = 1$ and $f(\{U_{n_2},U_{n_3}\}) = 2$\\
                                                                 1 & otherwise
                                                              \end{tabular}
                                                              \right.
\]
Now apply the partition relation $\Omega\longrightarrow (\Omega,4)^3_2$ to the function $g$.
{\flushleft{\underline{Claim 1:} There is no four-element subset of $\mathcal{U}$ on whose triples $g$ has value $2$.}}

For consider a 4-tuple $\mathcal{F} = \{U_{n_1},\; U_{n_2},\; U_{n_3},\; U_{n_4}\}$ with $n_1<n_2<n_3<n_4$ and suppose that $g$ has value $2$ at each three element subset of the set $\mathcal{F}$. Then, considering the subset $\{U_{n_1},\;U_{n_2},\; U_{n_3}\}$ we have $f(\{U_{n_1},\;U_{n_2}\}) = 1$ and $f(\{U_{n_2},\;U_{n_3}\}) = 2$, and considering $\{U_{n_2},\; U_{n_3},\; U_{n_4}\}$, we have $f(\{U_{n_2},\;U_{n_3}\}) = 1$ and $f(\{U_{n_3},\;U_{n_4}\}) = 2$, a contradiction. This completes the proof of Claim 1. 

With Claim 1 established, we find a subset $\mathcal{V}$ of $\mathcal{U}$ such that $\mathcal{V}$ is an $\omega$-cover, and $f$ is constant of value $1$ on $\lbrack\mathcal{V}\rbrack^3$. If $f$ is constant of value $2$ on $\lbrack\mathcal{V}\rbrack^2$, there is nothing more to prove. Thus, assume that there is a two-element subset, say $\{U_{n_1},\;U_{n_2}\}$ with $n_1<n_2$, of $\mathcal{V}$ for which $f(\{U_{n_1},\;U_{n_2}\}) =1$. Then the set $\mathcal{W} = \{U_n\in\mathcal{V}:n>n_2\}$ is still an $\omega$-cover of $X$.

{\flushleft{\underline{Claim 2:} $f$ has value $1$ on the set $\lbrack\mathcal{W}\rbrack^2$.}}
For if not, choose a two-element set $\{U_k,\;U_m\}\subset\mathcal{W}$ with $k<m$ and $f(\{U_k,\;U_m\})=2$. By the construction of $\mathcal{W}$ we have $n_1 < n_2 < k < m$. Also, we have $f(\{U_{n_1},U_{n_2}\}) = 1$ and $f(\{U_k,\;U_m\}) = 2$.
If $f(\{U_{n_2},U_k\}) = 1$, then by definition of the function $g$ we have  $g(\{U_{n_2},\;U_m,\;U_k\}) = 2$, contradicting the fact that $g$ has value $1$ on $\lbrack\mathcal{V}\rbrack^3$. Thus, we must have $f(\{U_{n_2},U_k\}) = 2$. But in this case we have $f(\{U_{n_1},\;U_{n_2}\}) = 1$, so that by definition $g(\{U_{n_1},\;U_{n_2},\;U_k\}) = 2$ also contradicting the fact that $g$ has value $1$ on $\lbrack\mathcal{V}\rbrack^3$. These contradictions follow from assuming that $f$ is not constant of value $1$ on the set $\lbrack\mathcal{W}\rbrack^2$. This completes the proof of Claim 2, and of the implication $\Omega\longrightarrow (\Omega,4)^3_2\,\Rightarrow\,(\Omega\longrightarrow (\Omega)^2_2$.

\subsection{A proof of the implication $\Omega\longrightarrow (\Omega)^2_2\,\Rightarrow\,(\Omega\longrightarrow (\Omega,\open)^2$}
 
 Since $\Omega$ is a subset of $\open$, this implication holds.

\subsection{A proof of the implication $\Omega\longrightarrow (\Omega,\open)^2_2\,\Rightarrow\,\sone(\Omega,\Omega)$}

Let a sequence $(\mathcal{U}_n:n\in\naturals)$ of $\omega$-covers of the space $X$ be given. We may assume that each $\mathcal{U}_n$ is countable, and enumerate it bijectively as $(U^n_m:m\in\naturals)$. Then define the $\omega$-cover
$\mathcal{V} = \{U^1_m\cap U^m_k: 1<m,\; k\in\naturals\}\setminus\{\emptyset\}$.  Also, for each element of $\mathcal{V}$ fix a specific choice of $k$ and $m$ to represent the element as $U^1_k\cap U^k_m$.

Define a function $f:\lbrack\mathcal{V}\rbrack^2\rightarrow\{0,\; 1\}$ so that
\[
   f(\{U^1_{k_1}\cap U^{k_1}_{m_1}, U^1_{k_2}\cap U^{k_2}_{m_2}\}) = \left\{
                                                                                                                           \begin{tabular}{ll}
                                                                                                                              $0$ & if $k_1\neq k_2$\\
                                                                                                                              $1$ & otherwise
                                                                                                                            \end{tabular}
                                                                                                                            \right.
\]
Applying the hypothesis that the partition relation $\Omega\longrightarrow (\Omega,\open)^2_2$ holds, choose a subset $\mathcal{W}$ of $\mathcal{V}$ such that $f$ is constant on $\lbrack\mathcal{W}\rbrack^2$, and $\mathcal{W}$ witnesses this partition relation for $f$. Assume that $f$ does not have value $0$ on $\lbrack\mathcal{W}\rbrack^2$, and thus that $\mathcal{W}$ is an open cover of $X$ but not necessarily an $\omega$-cover. Then $f$ has value $1$ on $\lbrack\mathcal{W}\rbrack^2$. Then for any $U^1_{k_1}\cap U^{k_1}_{m_1} \neq U^1_{k_2}\cap U^{k_2}_{m_2}$ in the set $\mathcal{W}$ we have $k_1 = k_2$, and thus all elements of $\mathcal{W}$ are subsets of one and the same $U^1_k$. But $U^1_k$ being an element of an $\omega$-cover of $X$ is a proper subset of $X$, and thus $\mathcal{W}$ is not a cover of $X$, a contradiction. Consequently $f$ has value $0$ on $\lbrack\mathcal{W}\rbrack^2$, and $\mathcal{W}$ is an $\omega$-cover of $X$. By the definition of the function $f$ we also find that if $U^1_{k_1}\cap U^{k_1}_{m_1}$ and $U^1_{k_2}\cap U^{k_2}_{m_2}$ are two distinct elements of $\mathcal{W}$, then $k_1\neq k_2$. 

For each $k$ such that $U^1_k\cap U^k_m$, with indices as in the originally chosen representation, appears in $\mathcal{W}$, pick $V_k = U^k_m\in\mathcal{U}_k$. For all other $k$ choose $V_k\in\mathcal{U}_k$ arbitrarily. Then the set $\{V_k:k\in\naturals\}$ witnesses the statement $\sone(\Omega,\Omega)$ for the given sequence $(\mathcal{U}_n:n\in\naturals)$ of $\omega$-covers of $X$.

The proof of Theorem \ref{egprothberger} is complete. $\Box$

\section{The Rothberger covering property for $\sigma$-compact spaces.}

As shown in Theorem \ref{th:MSUniform}, in the case of a $\sigma$-compact uniform space $(X,\Psi)$, for a subset $Y$ of $X$ the selection principles $\sone(\open_{\Psi},\open_Y)$ and $\sone(\open,\open_Y)$ are equivalent. In the case when $Y=X$, the latter is the Rothberger covering property. In this section we focus on this scenario: The topological space $(X,\tau)$ is $\sigma$-compact and has the covering property $\sone(\open,\open)$.
 According to Theorem \ref{th:MSUniform} for such a space, ONE has no winning strategy in the game $\gone(\open,\open)$. In fact, much stronger conclusions - see Theorem \ref{th:sigmacptRothb} -  can be drawn. 
 \begin{lemma}\label{th:cptRothbScattered}
 If $(X,\tau)$ is a compact $\textsf{T}_2$-space which has the property $\sone(\open,\open)$, then the space has an isolated point.
 \end{lemma}
\begin{proof}
For assume the contrary. Let $(X,\tau)$ be a compact space which has no isolated points. We show that ONE has a winning strategy in the game $\gone(\open,\open)$, implying by Pawlikowski's Theorem that $(X,\tau)$ does not have the property $\sone(\open,\open)$. 

Define a strategy $F$ for ONE in the game $\gone(\open,\open)$ on $X$ as follows:
SInce $X$ has no isolated points, it is infinite. Choose two distinct points $x_0$ and $_1$ from $X$. As $X$ is compact and $\textsf{T}_2$, choose open sets $U_0$ and $U_1$ such that $x_0\in U_0$ and $x_1\in U_1$, and the closures of $U_0$ and $U_1$ are disjoint. ONE's first move is
\[
  F(\emptyset) = \{X\setminus\overline{U}_0,\; X\setminus \overline{U}_1\}.
\]
When TWO chooses $T_1\in F(\emptyset)$, say $T_1 = X\setminus\overline{U}_{i_1}$, then as $X$ has no isolated points, the nonempty open set $U_{i_1}$ has no isolated points. Thus choose two distinct points $x_{i_1,0}$ and $x_{i_1,1}$ in $U_{i_1}$
and apply the fact that $X$ is compact $\textsf{T}_2$ again to find two open sets $U_{i_1,0}$ and $U_{i_1,1}$ for which their closures are disjoint subsets of $U_{i_1}$, and $x_{i_1,0}\in U_{i_1,0}$ and $x_{i_1,1}\in U_{i_1,1}$. ONE's response is
\[
  F(T_1) = \{X\setminus\overline{U}_{i_1,0},\; X\setminus \overline{U}_{i_1,1}\}
\]

When TWO chooses $T_2\in F(T_1)$, say $T_2 = X\setminus\overline{U}_{i_1,i_2}$, then as $X$ has no isolated points, the nonempty open set $U_{i_1,i_2}$ has no isolated points. Thus choose two distinct points $x_{i_1,i_2,0}$ and $x_{i_1,i_2,1}$ in $U_{i_1,i_2}$
and apply the fact that $X$ is compact $\textsf{T}_2$ again to find two open sets $U_{i_1,i_2,0}$ and $U_{i_1,i_2,1}$ for which their closures are disjoint subsets of $U_{i_1,i_2}$, and $x_{i_1,i_2,0}\in U_{i_1,i_2,0}$ and $x_{i_1,i_2,1}\in U_{i_1,i_2,1}$. ONE's response is
\[
  F(T_1,T_2) = \{X\setminus\overline{U}_{i_1,i_2,0},\; X\setminus \overline{U}_{i_1,i_2,1}\}
\]

When TWO chooses $T_3\in F(T_1,\;T_2)$, say $T_3 = X\setminus\overline{U}_{i_1,i_2,i_3}$, then as $X$ has no isolated points, the nonempty open set $U_{i_1,i_2,i_3}$ has no isolated points. Thus choose two distinct points $x_{i_1,i_2,i_3,0}$ and $x_{i_1,i_2,i_3,1}$ in $U_{i_1,i_2,i_3}$
and apply the fact that $X$ is compact $\textsf{T}_2$ again to find two open sets $U_{i_1,i_2,i_3,0}$ and $U_{i_1,i_2,i_3,1}$ for which their closures are disjoint subsets of $U_{i_1,i_2,i_3}$, and $x_{i_1,i_2,i_3,0}\in U_{i_1,i_2,i_3,0}$ and $x_{i_1,i_2,i_3,1}\in U_{i_1,i_2,i_3,1}$. ONE's response is
\[
  F(T_1,T_2,T_3) = \{X\setminus\overline{U}_{i_1,i_2,i_3,0},\; X\setminus \overline{U}_{i_1,i_2,i_3,1}\}
\]
and so on.

Consider an $F$-play
\[
  F(\emptyset),\; T_1,\; F(T_1),\; T_2,\; \cdots.\; F(T_1,\cdots,T_n),\; T_{n+1},\cdots
\]
of the game $\gone(\open,\open)$ on $X$. From the definition of $F$ there is for each $n$ a nonempty open subset $U_{i_1,\cdots,i_n}$ of $X$ such that 
\begin{enumerate}
\item{$T_n = X\setminus\overline{U}_{i_1,\cdots,i_n}$}
\item{ the set $\overline{U}_{i_1,\cdots,i_n}$ is a nonempty subset of $U_{i_1,\cdots,i_{n-1}}$, }
\item{and $U_{i_1,\cdots,i_n}$ is an open set containing the point $x_{i_1,\cdots,i_n}$ of $X$.}
\end{enumerate}
Since $(X,\tau)$ is compact, the sequence $(\overline{U}_{i_1,\cdots,i_n}:n\in\mathbb{N})$ as a descending sequence of nonempty compact subsets of $X$ has a nonempty intersection. Let the point $y$ be a member of this intersection. Then for each $n$, $y$ is not a member of $T_n$. It follows that ONE wins the $F$ play, and as we considered an arbitrary $F$ play, $F$ is a winning strategy for ONE in the game $\gone(\open,\open)$ on $X$. By Pawlikowski's Theorem, $(X,\tau)$ does not have the selection property $\sone(\open,\open)$.
\end{proof}

\begin{lemma}\label{lemma:scattered1}
If $(X,\tau)$ is a compact space with property $\sone(\open,\open)$, then each nonempty subset has an isolated point. 
\end{lemma}
\begin{proof}
If not, suppose that $(X,\tau)$ is a compact space with property $\sone(\open,\open)$ and let $U$ be a nonempty set with no isolated points. Then also $\overline{U}$ has no isolated points, but is a compact space with the property $\sone(\open,\open)$, contradicting Theorem \ref{th:cptRothbScattered}.
\end{proof}

A topological space is scattered if each nonempty set has an isolated point. Thus, compact spaces with property $\sone(\open,\open)$ are scattered. For a topological space let $\mathcal{I}(X)$ denote the set of isolated points of $X$. Then we define the Cantor-Bendixson sequence of $X$ as follows:
\begin{itemize}
\item{ $X^0 = X$}
\item{For an ordinal $\alpha$, $X^{\alpha+1} = X\setminus X^{\alpha}$.}
\item{For a limit ordinal $\alpha$, $X^{\alpha} = \bigcap_{\beta<\alpha}X^{\beta}$}
\end{itemize}
Observe that $(X,\tau)$ is a scattered space if there is an ordinal $\alpha$ for which $X^{\alpha} = \emptyset$. The least value of $\alpha$ for which $X^{\alpha}=\emptyset$ is said to be the scattered height of $X$ and is denoted $ht(X)$. One can show that of $(X,\tau)$ is a compact scattered space, then $ht(X)$ is a successor ordinal.

\begin{lemma}\label{lemma:scattered2}
A compact scattered space has the property $\sone(\open,\open)$.
\end{lemma}
\begin{proof}
We prove the lemma by induction on the ordinal $\alpha$ for which $ht(X)=\alpha+1$.
When $\alpha = 0$, then $X$ is a discrete finite set, and there is nothing to prove.

Assume that the lemma is proven for all compact scattered spaces $(X,\tau)$ with $ht(X) = \beta+1$ where $\beta<\alpha$. Now consider a compact scattered space $(X,\tau)$ with $ht(X) = \alpha+1$.  Then $X$ has a finite set of isolated points, being compact. Let $\{x_1,\cdots,x_n\}$ be the set of isolated points of $X$. Let $(\mathcal{U}_n:n\in\mathbb{N})$ be a given sequence of open covers of $X$. For $1\le i\le n$ choose $U_i\in\mathcal{U}_i$ so that $x_i\in {U}_i$. If $X = \cup\{U_i:i\le n\}$, we are done. Else, the subspace $Y = X\setminus \cup\{U_i:i\le n\}$ is a compact scattered space and $ht(Y) < \alpha+1$. Thus, by the induction hypothesis, $Y$ has the property $\sone(\open,\open)$. Applying this property to the sequence $(\mathcal{U}_m:n<m \mbox{ and }m\in\mathbb{N})$, we find for each $m$ with $n<m$ a set $U_m\in\mathcal{U}_m$ such that $Y$ is covered by the sets $U_i,\; n<i$. But then the selection $(U_j:j\in\mathbb{N})$ witnesses for teh given sequence of covers that $(X,\tau)$ has the property $\sone(\open,\open)$.
\end{proof}

Lemmas \ref{lemma:scattered1} and \ref{lemma:scattered2} prove that a compact $T_2$-space is scattered if, and only if, it has the property $\sone(\open,\open)$. 

\begin{lemma}\label{lemma:scattered3}
For a compact $T_2$ space $(X,\tau)$ the following statements are equivalent:
\begin{enumerate}
\item{$(X,\tau)$ has the property $\sone(\open,\open)$}
\item{ TWO has a winning strategy in the game $\gone(\open,\open)$}
\end{enumerate}
\end{lemma}
\begin{proof}
We must show that (1) implies (2). Thus, assume that $(X,\tau)$ has the property $\sone(\open,\open)$. By Lemma \ref{lemma:scattered1} the space $(X,\tau)$ is scattered, and has scattered height $ht(X) = \alpha+1$ for some ordinal $\alpha$. We prove the Lemma by induction on $\alpha$.

When $\alpha=0$ then $X$ is a finite set, and TWO's winning strategy is to in each inning cover a previously uncovered point. Thus, assume that $\alpha>0$ and that for each $\beta<\alpha$ it has been proven that for any compact $T_2$ space of scattered height $\beta+1$, TWO has a winning strategy in the game $\gone(\open,\open)$. Now let $(X,\tau)$ be a compact scattered space of height $\alpha+1$. Since $X$ has isolated points and is compact, its set of isolated points is finite, say the set of isolated points of $X$ is $\{x_1,\cdots,x_n\}$. TWO's winning strategy on $X$ proceeds as follows. During inning $i$ for $i\le n$, TWO chooses an element $T_i$ of ONE's open cover $\mathcal{U}_i$ such that $x_i\in T_i$. Then after the first $n$ innings passed, the space $Y = X\setminus \cup\{T_i:i\le n\}$ is compact, scattered, and of height $\beta+1<\alpha+1$. By the induction hypothesis TWO has a winning strategy $F_Y$ in the game $\gone(\open,\open)$ on $Y$. For the remainder of the game on $X$, TWO follows the winning strategy $F_Y$ on $Y$.
\end{proof}

\begin{theorem}\label{th:sigmacptRothb}
Consider a $\sigma$-compact topological space $(X,\tau)$. The following statements are equivalent:
\begin{enumerate}
\item{$(X,\tau)$ has the property $\sone(\open,\open)$}
\item{TWO has a winning strategy in the game $\gone(\open,\open)$}
\end{enumerate}
\end{theorem}
\begin{proof}
We must prove that (1) implies (2). Fix a well-ordering $\prec$ of the family of open subsets of $X$. Since $X$ is $\sigma$-compact, write $X$ as an increasing union of countably many subsets $X_n$, $n\in\mathbb{N}$, where each $X_n$ is compact. By Lemma \ref{lemma:scattered3}, for each $n$ TWO has a winning strategy $F_n$ in the game $\gone(\open,\open)$ played on $X_n$. We now define a strategy for TWO in the game on $X$, utilizing the strategies $F_n$. First, write $\mathbb{N}$ as a union of countably many pairwise disjoint infinite sets, $S_n$, $n\in \mathbb{N}$. For each $n$ enumerate $S_n$ as $(s^n_k:k\in\mathbb{N})$ so that for $k<\ell$, $s^n_k<s^n_{\ell}$. Now define TWO's strategy $F$ for the game $\gone(\open,\open)$ on $X$ as follows:

Given inning $k$, and the first $k$ moves $O_1,\cdots,O_k$ of ONE, first find $n$ so that $k\in S_n$, and then identify $\ell$ so that $k = s^n_{\ell}$. Then define $F(O_1,\cdots,O_k)$ as follows:
For $1\le i\le \ell$ define $\mathcal{V}_i$ to be the set
$\{U\cap X_n: U\in O_{s^n_i}\}$, and then compute
\[
  T = F_n(O_{s^n_1},\cdots,O_{s^n_{\ell}}),
\] 
TWO's response in the game $\gone(\open_{X_n},\open_{X_n})$ using the winning strategy $F_n$. Then choose the $\prec$-first $U\in O_k$ for which $U\cap X_n = T$. Declare $F(O_1,\cdots,O_k) = U$.

We claim that $F$ as defined here is a winning strategy for TWO in the game $\gone(\open,\open)$ played on $X$. For consider any $F$-play, say
\[
  O_1, F(O_1), O_2, F(O_1,O_2), \cdots, O_k, F(O_1,\cdots O_k),\cdots
\]
Consider a point $x\in X$, and choose the smallest $n$ with $x\in X_n$. Recall that $S_n$ is enumerated as $s^n_1 < s^n_2 < \cdots < s^n_{\ell} < \cdots$. Consider the subsequence
\[
  O_{s^n_1}, F(O_1,\cdots,O_{s^n_1}), O_{s^n_2}, F(O_1,\cdots,O_{s^n_2}), \cdots, O_{s^n_k}, F(O_1,\cdots,O_{s^n_k}), \cdots
\]
of the given $F$-play.

In then notation above used in the definition of $F$ we have $\mathcal{V}_i = \{U\cap X_n: U\in O_{s^n_i}\}$ is for each $i$ an open cover of $X_n$, and $X_n\cap F(O_1,\cdots,O_{s^n_i}) = F_n(\mathcal{V}_1,\cdots,\mathcal{V}_{i})$ for each$i$. Since
\[
 \mathcal{V}_1, F_n(\mathcal{V}_1), \mathcal{V}_2, F_n(\mathcal{V}_1,\mathcal{V}_2), \cdots
\]
is an $F_n$ play of the game $\gone(\open_{X_n},\open_{X_n})$ on $X_n$, and $F_n$ is a winning strategy for TWO in this game, we find that for some $k$ the point $x$ is an element of $F(O_1,\cdots,O_{s^n_k})$.
\end{proof}

For $\sigma$-compact spaces the fact that TWO has a winning strategy in the game $\gone(\open,\open)$ is in in turn equivalent to a strong Ramseyan partition relation on the $\omega$-covers of the space. Towards formulating this stronger property we introduce yet another class of open covers for topological spaces, first studied by Gerlits and Nagy in \cite{GN}. 
An open cover $\mathcal{U}$ of a topological space $(X,\tau)$ is said to be a $\gamma$-cover if it is infinite and for each $x\in X$, for all but finitely many $U\in\mathcal{U}$ it is the case that $x\in U$. The symbol $\Gamma$ denotes the set of all $\gamma$-covers of $X$.
In Theorem 1 of \cite{GN} the authors prove
\begin{theorem}[Gerlis and Nagy] \label{th:GNTWO} For a $\textsf{T}_{3\frac{1}{2}}$-space $(X,\tau)$ the following are equivalent:
\begin{enumerate}
\item{TWO has a winning strategy in the game $\gone(\open,\open)$ on $X$}
\item{TWO has a winning strategy in the game $\gone(\Omega,\Gamma)$ on $X$}
\end{enumerate}
\end{theorem}
It is evident that when TWO has a winning strategy in the game $\gone(\Omega,\Gamma)$, then ONE does not have a winning strategy in that game, whence in turn $\sone(\Omega,\Gamma)$ holds. 

The main theorem of this section is as follows:
\begin{theorem}\label{th:gammasets} For a $\sigma$-compact space $(X,\tau)$ the following are equivalent:
\begin{enumerate}
\item{$(X,\tau)$ has the property $\sone(\open,\open)$.}
\item{TWO has a winning strategy in the game $\gone(\open,\open)$.}
\item{$(X,\tau)$ has the property $\sone(\Omega,\Gamma)$.}
\item{For all $m$ and $k$ the partition relation $\Omega\rightarrow(\Gamma)^m_k$ holds.}
\end{enumerate}
\end{theorem}
\begin{proof}
The equivalence of (1) and (2) is proven in Theorem \ref{th:sigmacptRothb}. That (2) implies (3) follows from Theorem \ref{th:GNTWO}. To see that (3) implies (4), let $\mathcal{U}$ be an $\omega$-cover of $X$. We may assume that $\mathcal{U}$ is countable. Let a function $f:\lbrack\mathcal{U}\rbrack^m\rightarrow\{1,\cdots,k\}$ be given. By (3) let $\mathcal{V}\subset\mathcal{U}$ be such that $\mathcal{V}$ is a $\gamma$-cover of $X$. Note that any infinite subset of $\mathcal{V}$ is also a $\gamma$-cover of $X$. Now apply Ramsey's Theorem to the coloring $f$ restricted to the set of $m$-tuples of $\mathcal{V}$. We find an infinite subset $\mathcal{W}$ of $\mathcal{V}$, and an $i\in\{1,\cdots,k\}$, such that on $\lbrack\mathcal{W}\rbrack^m$ the function $f$ is constant and has value $i$. Finally, to see that (4) implies (1), note that since $\Gamma\subset\Omega$, (4) implies that for all $m$ and $k$ the partition relation $\Omega\rightarrow(\Omega)^m_k$ holds. But then by Theorem \ref{egprothberger} the space $(X,\tau)$ has the property $\sone(\Omega,\Omega)$, and thus by Theorem \ref{thm:XisYRothberger} has the property $\sone(\open,\open)$.
\end{proof}
To see that the hypothesis of being a compact space with the property $\sone(\open,\open)$ is not ``trivial", note the following: Metrizable spaces that are compact and have property $\sone(\open,\open)$ are necessarily countable. But for each infinite cardinal number $\kappa$ there is a $\textsf{T}_2$ topological space of cardinality $\kappa$ that is $\sigma$-compact and has the property $\sone(\open,\open)$. The following examples of this phenomenon are based on a result of Corson (Proposition 4 of \cite{Corson}), and were pointed out in \cite{MSUniform}: 
\begin{theorem} [Corson]\label{th:Corson}
If $\{X_i:i\in I\}$ is a family of $\sigma$-compact topological groups where for each $i$ the identity element of $X_i$ is $e_i$, then the group
\[
  G = \{f\in\prod_{i\in I}X_i:\vert\{j\in I:f(j)\neq e_j\}\vert <\aleph_0\}
\]
is $\sigma$-compact.
\end{theorem}
\begin{proposition}\label{prop:sigmacompactRothberger}
For each infinite cardinal number $\kappa$ there is a $\textsf{T}_{3\frac{1}{2}}$ $\sigma$-compact topological group of cardinality $\kappa$ which has the property $\sone(\open,\open)$.
\end{proposition}
\begin{proof}
Let cardinal number $\kappa$ be given. In Theorem \ref{th:Corson} take $I$ to be $\kappa$ and for each $i\in I$ take $X_i$ to be the additive group of integers, $\integers$. Now observe that the group $G$ as in Corson's Theorem in fact has the property $\sone(\open,\open)$. To see this take for each $n$ a neighborhood $U_n$ of the identity element of $G$. We may assume that each $U_n$ is a basic open set, and thus that there is a finite set $F_n\subseteq I$ such that $U_n$ is of the form $\{f\in G:(\forall i\in F_n)(f(i)= 0)\}$.
 
The set $C = \bigcup_{n<\omega}F_n$ is a countable subset of $I$ and thus the subgroup $G_C = \{f\lceil_C:f \in G\}$ of the group $\prod_{i\in C}X_i$ is a countable set. Thus,  the property $\sone(\open_{nbd},\open_{G_C})$ evidently holds. For each $n$ choose a $g_n\in G_C$ such that $G_C\subseteq \bigcup_{n<\omega}g_n*U_n\lceil_C$. For each $n$ choose $f_n\in G$ with $f_n\lceil_C = g_n$. Then it follows that $G\subseteq \bigcup_{n<\omega}f_n*U_n$. Thus $G$ satisfies the property $\sone(\open_{nbd},\open)$. As $G$ is a $\sigma$-compact Theorem \ref{th:MSUniform} implies that $G$ satisfies the property $\sone(\open,\open)$. 
\end{proof}

Theorem 14 of  \cite{MSUniform} gave an \emph{ad hoc} argument that TWO has a winning strategy in the game $\gone(\open,\open)$ in these specific examples. Theorem \ref{th:sigmacptRothb} generalizes that prior result. 

It is further worth noting that the existence of a winning strategy for TWO in the game $\gone(\open,\open)$ on an uncountable $\textsf{T}_{3\frac{1}{2}}$ space does not imply that the space is $\sigma$-compact. Thus is illustrated using the following example of Comfort and Ross (\cite{CR}, Example 3.2). Consider the set 
  $G:=\{f\in\,^{\omega_1}2:\vert\{\alpha: f(\alpha)\neq 0\}\vert<\aleph_0\}$ endowed with the operation $\oplus$ of coordinate-wise addition modulo $2$. Then $(G,\oplus)$ is an Abelian group. 
Endow $G$ with the {\sf G}$_{\delta}$ topology. In \cite{CR} it is proven that the group $(G,\oplus)$ endowed with this $\textsf{G}_{\delta}$ topology is Lindel\"of 
and thus ${\sf T}_4$. 

For convenience we review a few facts about so-called $P$-spaces - see for example Section 5 of \cite{GilmannHenriksen}. A topological space is said to be a $P$-space if each ${\sf G}_{\delta}$ set is open. Every subspace of a $P$-space is a $P$-space. Every countably infinite subspace of a ${\sf T}_2$ $P$-space is closed and discrete. It follows that a compact $P$-space is finite, and thus a $\sigma$-compact $P$-space is countable. Thus, no uncountable $P$-space is a closed subspace of a $\sigma$-compact ${\sf T}_2$-space. If a topological group $(G,*)$ is a Lindel\"of $P$-space then it is has the property $\sone(\open_{nbd},\open)$ in a strong sense: Let $(U_n:n<\omega)$ be a sequence of neighborhoods for the identity. Then $U = \bigcap_{n<\omega}U_n$ is a neighborhood for the identity. Since the group is Lindel\"of, fix a sequence $(x_n:n<\omega)$ of elements of the group such that $x_n*U,\, n<\omega$ covers the group. Then the sequence $(x_n*U_n:n<\omega)$ witnesses that the uncountable group $(G,\oplus)$ has the property $\sone(\open_{nbd},\open)$. Moreover this group is not contained as a closed subspace of any $\sigma$-compact group.

Theorem 2.3 of Comfort in \cite{Comfort} implies the following generalization of the example just given:
\begin{theorem}[Comfort]\label{Comfortcountable}
Let $(G_i,*_i)$, $i\in I$, be a family  of countable topological groups. Endow the product $\prod_{i\in I}G_i$ with the ${\sf G}_{\delta}$ topology. Then the subgroup
\[
  G:=\{f\in\prod_{i\in I}G_i: \vert\{j\in I:f(j)\neq {\sf id}_j\}\vert<\aleph_0\}
\]
is a Lindel\"of $P$-group.
\end{theorem}
In particular, for each uncountable cardinal number $\kappa$ there is a ${\sf T}_0$ Lindel\"of $P$ group of cardinality $\kappa$.
Now Galvin proved a result - see the Lemma in Section 2 of \cite{GN} - that implies that if a space is a Lindel\"of P-space then has the property $\sone(\open,\open)$. Thus, Lindel\"of $P$ groups satisfy the stronger selection principle $\sone(\open,\open)$. 
\begin{proposition}\label{twowinsrothberger} For each uncountable cardinal $\kappa$ there is a ${\sf T}_0$ Lindel\"of $P$-group of cardinality $\kappa$ such that TWO has a winning strategy in $\gone(\open,\open)$.
\end{proposition}
Thus for the examples based on Comfort's work cited above, none is homeomorphic to a closed subspace of a $\sigma$-compact space, yet TWO has a winning strategy in the game $\gone(\open,\open)$. By arguments as in the proof of Theorem \ref{th:gammasets}, this class of examples also satisfies the partition relation $\Omega\rightarrow(\Gamma)^m_k$ for all natural numbers $m$ and $k$.

\section{A question about the polarized partition relation}

The polarized partition relation was studied by Rado in \cite{Rado}, and formally introduced by Erd\"os and Rado in Section 9 of \cite{ER}. Generalizing their notion we define the following: 
Let $\mathcal{A}_1$, $\mathcal{A}_2$ $\mathcal{B}_{i,1}$ and $\mathcal{B}_{i,2}$ be families of sets for $1\le i\le n$.

\begin{definition} The symbol
 \[
   \left( {\begin{array}{c}
   \mathcal{A}_1 \\
   \mathcal{A}_2 \\
  \end{array} } \right) 
  \longrightarrow
   \left( {\begin{array}{ccc}
   \mathcal{B}_{1,1} & \cdots & \mathcal{B}_{n,1} \\
   \mathcal{B}_{1,2} & \cdots & \mathcal{B}_{n,2} \\
  \end{array} } \right)^{1,1} 
\]
denotes the statement that for each $A_1\in\mathcal{A}_1$ and each $A_2\in\mathcal{A}_2$, and for each function $f:A_1\times A_2 \rightarrow\{1,\cdots,n\}$ there exists an $i$ with $1\le i\le n$, and sets $B_{i,1}\in \mathcal{B}_{i,1}$ and $B_{i,2}\in\mathcal{B}_{i,2}$ such that $B_{i,1}\subseteq A_1$ and $B_{i,2}\subseteq A_2$, and $f$ has the value $i$ on the set $B_{i,1}\times B_{i,2}$.
\end{definition}

First observe that by Theorem 47 of \cite{ER}, for each infinite cardinal number $\kappa$ the negative relation
 \[
   \left( {\begin{array}{c}
   \kappa \\
   \kappa \\
  \end{array} } \right) 
  \not\longrightarrow
   \left( {\begin{array}{cc}
   \kappa & 1 \\
   1 & \kappa\\
  \end{array} } \right)^{1,1} 
\]
holds. This in turn implies that for infinite topological spaces $(X,\tau)$ the relation
 \[
   \left( {\begin{array}{c}
   \Omega \\
   \Omega \\
  \end{array} } \right) 
  \not\longrightarrow
   \left( {\begin{array}{cc}
   \Omega & 1 \\
   1 & \Omega\\
  \end{array} } \right)^{1,1} 
\]
holds. Thus, we relax the polarized partition relation to a square bracket version, analogous to the weakening of the ordinary partition relation.


\begin{definition} Let $\mathcal{A}_1$, $\mathcal{A}_2$, $\mathcal{B}_1$ and $\mathcal{B}_2$ be families of sets. The symbol
\[
     \left( {\begin{array}{c}
   \mathcal{A}_1 \\
   \mathcal{A}_2 \\
  \end{array} } \right) 
  \longrightarrow
   \left[ 
   \begin{array}{c}
   \mathcal{B}_{1}  \\
   \mathcal{B}_{2}  \\
  \end{array}  
  \right]^{1,1}_{k/\le \ell} 
\]
denotes the statement that for each $A_1\in\mathcal{A}_1$ and $A_2\in\mathcal{A}_2$, and for each function $f:A_1\times A_2\rightarrow\{1,\cdots,k\}$ there are sets $B_1\subseteq A_1$ and $B_2\subseteq A_2$ such that $B_1\in\mathcal{B}_1$, $B_2\in\mathcal{B}_2$, and $\vert\{f(b_1,b_2):b_1\in B_1 \mbox{ and }b_2\in B_2\}\vert\le \ell$.
\end{definition}

It is not clear if there is a version of the square bracket polarized partition relation $\left(\begin{array}{c} \Omega\\ \Omega \end{array}\right) \longrightarrow \left [ \begin{array}{c}\Omega\\ \Omega \end{array}\right]^{1,1}_{k/\le \ell}$ that characterizes whether a space has the property $\sone(\Omega,\Omega)$. We conjecture:
\begin{conjecture}\label{conj:polarizedconjecture}
If a topological space $(X,\tau)$ satisfies $\left(\begin{array}{c} \Omega\\ \Omega \end{array}\right) \longrightarrow \left [ \begin{array}{c}\Omega\\ \Omega \end{array}\right]^{1,1}_{k/<3}$ for any positive integer $k$, then $(X,\tau)$ has the property $\sone(\Omega,\Omega)$.
\end{conjecture}
We give a complete analysis in the proof of Theorem \ref{th:polarizedimpl}  to indicate to the reader the single case of the analysis which does not support the conjecture. We believe that this one ``counter" case may be a shortcoming of the proof technique. The following concept, and corresponding facts about its relation to the conjecture are important.  
Space $X$ satisfies property $\textsf{Split}(\Omega, \Omega)$ if there is for each $\omega$-cover $\mathcal{U}$ of $X$ two $\omega$-covers $\mathcal{B}_1$ and $\mathcal{B}_2$ such that $\mathcal{B}_1 \cap \mathcal{B}_2 = \emptyset$ and $\mathcal{B}_1 \cup \mathcal{B}_2 = \mathcal{U}$. This property was introduced in [19] and studied further in [15].

\begin{theorem}\label{th:polarizesplit} Let $X$ be a space which has the Lindel\"of covering property in all finite powers. If for  $X$ the polarized partition relation $\left(\begin{array}{c} \Omega\\ \Omega \end{array}\right) \longrightarrow \left [ \begin{array}{c}\Omega\\ \Omega \end{array}\right]^{1,1}_{k/<3}$ holds, then $X$ has property
$\textsf{Split} (\Omega, \Omega)$.
\end{theorem}
\begin{proof}
Let $\mathcal{U}$ be an $\omega$-cover of $X$ . We may assume that it is countable. Enumerate it bijectively as $(U_m : m \in \mathbb{N} )$. Then define $f :\mathcal{U} \times \mathcal{U} \rightarrow \{0,1,2\}$ so that
\[
f(U_m,U_n) = \left\{
         \begin{tabular}{ll}
           0  & i f $m = n$\\
           1 & if $m<n$ \\
           2 & if $m > n$
          \end{tabular}
          \right.
\]
Apply the polarized relation to find $\omega$-covers $\mathcal{V}$ and $\mathcal{W}$ contained in $\mathcal{U}$ such that on $\mathcal{V} \times \mathcal{W}$  $f$ is two-valued. We show that $\mathcal{V}$ and $\mathcal{W}$ are disjoint
by showing that $0$ is not a value of $f$. Suppose on the contrary that $0$ is a value of $f$ on $\mathcal{V} \times \mathcal{W}$. Choose $n$ such that $(U_n,U_n) \in  \mathcal{V} \times \mathcal{W}$. Since $\mathcal{V}$ and $\mathcal{W}$ are $\omega$-covers they are infinite. Thus there are $U_m\in \mathcal{V}$ with $m > n$ and $U_k \in \mathcal{W}$ with $k > $n. This implies that $1$ and $2$ are also values of $f$, so that $f$ takes three distinct values on $\mathcal{V} \times \mathcal{W}$ instead of just two.
\end{proof}

\begin{theorem}\label{th:polarizedimpl}
Let $X$ be a space. Of the following statements, each implies the next.
\begin{enumerate}
\item{$\sone(\Omega,\Omega)$}
\item{For each positive integer $k$, $\left({\begin{array}{c}
   \Omega \\
   \Omega \\
  \end{array}}\right) 
  \longrightarrow
   \left[ 
   \begin{array}{c}
   \Omega  \\
   \Omega  \\
  \end{array}  
  \right]^{1,1}_{k/< 3} 
$}
\item{$\sfin(\Omega,\Omega)$}
\end{enumerate}
\end{theorem}
\begin{proof} 
{\flushleft\underline{$(1)\Rightarrow(2)$: }}  Let $\mathcal{U}$ and $\mathcal{V}$ be $\omega$-covers of $X$. Since $X$ has property $\sone(\Omega, \Omega)$ we may assume that each of these is countable. Enumerate $\mathcal{U}$ bijectively as
$(U_n: n \in\mathbb{N})$. Let $k$ be a positive integer and let $f: \mathcal{U} \times \mathcal{V} \rightarrow \{1,. . . , k\} $ be given.

Choose a chain of subsets of $\mathcal{V}$ as follows: Put $\mathcal{V}_{0} = \mathcal{V}$ and for $0 \le n < \omega$ choose $i_{n+1} \in \{1,. . . , k\}$ such that 
\[
 \mathcal{V}_{n+1} = \{ V \in \mathcal{V}_n: f(U_{n+1},V) = i_{n+1}\} \in \Omega
\]
Since $\mathcal{U}$ is an $\omega$-cover, fix an $i \in \{1,\cdots , k\}$ and $n_1 < n_2 < \cdots$ such that $\{U_{n_1}, U_{n_2}, \cdots \} \in \Omega$ and for each $j$, $i_{n_j} = i$. Apply property $\sone(\Omega, \Omega)$
to the sequence $(\mathcal{V}_{n_1}, \mathcal{V}_{n2}, \cdots, \mathcal{V}_{n_k}, \cdots )$ of $\omega$-covers of $X$ and select for
each $m$ a $V_m \in \mathcal{V}_{n_m}$, such that $\{V_m : m \in \mathbb{N}\}$  is an $\omega$-cover of
$X$ . We may assume that this enumeration is bijective. At this stage we have that if $j \le m$, then $f(U_{n_j},V_m ) = i$.

Put $\mathcal{U}_{-1} = \{U_{n_l}, U_{n_2},\cdots \}$ and choose for each $m$ a $j_{m+l} \in \{ 1, \cdots, k\}$ such that $\mathcal{U}_m = \{U_{n_r} \in \mathcal{U}_m : f(U_{n_r}, V_{m+l}) = j_{m+l} \} \in \Omega$. Fix a $j$ and
$m_l < m_2 <\cdots $ such that for each $r$ $j_{m_r} = j$, and $\{V_{m_1}, V_{m_2},\cdots\}$ is an
$\omega$-cover of $X$ . Next we apply $\sone(\Omega, \Omega)$ to the sequence $(\mathcal{U}_{m_]}, \mathcal{U}_{m_2}, \cdots )$
of $\omega$-covers of $X$, as follows: From $\mathcal{U}_{m_t}$ choose $U_{n_{h_t}}$, such that for each $t$,
$m_t < h_t < h_{t+l}$ and $\{U_{n_{h_t}} : t \in\mathbb{N}\}$ is an $\omega$-cover of $X$ . One way
of seeing that this can be done is as follows. In [21] it was shown that $X$
has property $\sone(\Omega, \Omega)$ if, and only if, ONE has no winning strategy in
the following game, denoted $\gone (\Omega, \Omega)$: In the $n$-th inning ONE chooses
an $\omega$-cover $O_n$, of $X$ ; TWO responds by selecting $T_n \in O_n$. A play
$(O_1,T_1 ,\cdots , O_n,T_n ,\cdots. )$ is won by TWO if $\{T_I,T_2, T_3,\cdots \}$ is an $\omega$-cover
of $X$ ; otherwise, ONE wins.

Now consider the strategy which calls on ONE to play in the $t$-th inning the set of elements of $\mathcal{U}_{m_t}$ of the form $U_{n_h}$, where $h$ exceeds
$\max \{m_1,\cdots, m_t\}$ as well as the maximum of all subscripts of elements
chosen thus far by TWO. Since this is not a winning strategy, there is
a play according to it which is won by TWO. Such a play provides us
with a sequence of $U_{n_{h_t}}$'s as above.
But then on $\{U_{n_{h_l}}, U_{n_{h_2}},\cdots \} \times \{V_{m_1},V_{m_2}, \cdots \}$ the function $f$ takes
only the two values, $i$ and $j$ .

{\flushleft\underline{$(2)\Rightarrow(3)$: }} Let $(\mathcal{U}_n : n \in \naturals )$ be a sequence of $\omega$-covers of $X$. We may assume each $\mathcal{U}_n$ is countable and enumerate it as $(U^n_m : m \in \naturals )$.
Define $\mathcal{V}$ to be the non-empty intersections of the form $U^1_k\cap U^k_n$. Then $\mathcal{V}$ is an $\omega$--cover of $X$ . For each element of $\mathcal{V}$ choose a specific representation
of the form $U^1_k \cap U^k_n$. Define a coloring $f : \mathcal{V} \times \mathcal{V} \rightarrow \{0,1,2,3,4,5,6,7,8\}$ such that
\[
  f((U^1_{k_1}\cap U^{k_1}_{n_1}, U^1_{k_2}\cap U^{k_2}_{n_2})) = \left\{
         \begin{tabular}{lll}
                 0 & if & $k_1 = k_2$ and $n_1 = n_2$\\
                 1 & if & $k_1 = k_2$ and $n_1 < n_2$\\
                 2 & if & $k_l = k_2$ and  $n_1 > n_2$\\
                 3 & if & $k_1 <  k_2$ and $n_1 = n_2$\\
                 4 & if & $k_1 <  k_2$ and $n_1 < n_2$\\
                 5 & if & $k_1 < k_2$  and $n_1 > n_2$\\
                 6 & if & $k_l > k_2$ and $n_1 = n_2$\\
                 7 & if & $k_1 > k_2$ and $n_1 < n_2$\\
                 8 & if & $k_1 > k_2$ and $n_1 > n_2$.
     \end{tabular}
     \right.
\]
Applying the partition relation to this $f$ we find $\omega$-covers $\mathcal{K}$ and $\mathcal{L}$, subsets of $\mathcal{V}$ , such that $f$ has no more than two values on $\mathcal{K} \times \mathcal{L}$. First, note that $f$ cannot have all its values in $\{0,1,2\}$. For in this case there is a fixed $k$ such that all elements of $\mathcal{K}$ and of $\mathcal{L}$ are subset of $U^1_k$, contradicting the fact that $\mathcal{K}$ and $\mathcal{L}$ are $\omega$-covers.

By Theorem \ref{th:polarizesplit} choose $\omega$-covers $\mathcal{A}$ and $\mathcal{B}$ such that $\mathcal{A}\cup\mathcal{B}\subset\mathcal{V}$ and $\mathcal{A}\cap\mathcal{B} = \emptyset$, and consider the function $f$ restricted to $\mathcal{A}\times\mathcal{B}$.
By the hypothesis choose $\omega$-covers $\mathcal{K}\subset\mathcal{A}$ and $\mathcal{L}\subset\mathcal{B}$ such that $f$ has an image set of at most two values on the set $\mathcal{K}\times\mathcal{L}$.
First note that on $\mathcal{K}\times\mathcal{L}$ the function $f$ does not have value $0$ as $\mathcal{K}$ and $\mathcal{L}$ are disjoint. Secondly, by a prior remark, the range of $f$ cannot be in $\{0,1,2\}$.

The rest of the argument consists of an extensive case analysis starting from a hypothesis like $range(f)\subseteq \{a,b\}$ where $a$ and $b$ are elements of $\{0,1,\cdots,8\}$. In all cases, except one, the conclusion is that either that case is not possible, or else that case produces a witness that $\sone(\Omega,\Omega)$ holds for the given sequence $(\mathcal{U}_n:n\in\naturals)$ of $\omega$-covers of the space $X$. We leave the examination of these cases to the reader. The one exceptional case is the case when $range(f) = \{4,8\}$:

Observe that in this case, for any $(U^1_{k_1}\cap U^{k_1}_{n_1},\; U^1_{k_2}\cap U^{k_2}_{n_2}) \in \mathcal{K}\times\mathcal{L}$ it is the case that either $k_1<k_2$ and $n_1<n_2$, or else $k_1>k_2$ and $n_1>n_2$. Thus, if we define
$M = \{k:U^1_k\cap U^k_n\in\mathcal{K}\}$ and $N=\{k\in\naturals:U^1_{k_2}\cap U^{k_2}_{n_2}\in\mathcal{L}\}$, then $M$ and $N$ are disjoint. Moreover, both $M$ and $N$ are infinite. Enumerate $N$ in increasing order as $(\ell_1,\;\ell_2,\; \cdots,\ell_m,\cdots)$. Moreover, for each $k$ with a $U^1_k\cap U^k_n\in\mathcal{L}$, let $n(k)$ denote the least such $n$.

Now proceed to choose for each $\mathcal{U}_k$ a finite subset $\mathcal{F}_k$ as follows:

For $k\in M$ with. $k\le \ell_1$ choose $U^k_n\in\mathcal{U}_k$ such that $n\le n(\ell_1)$ and let $\mathcal{F}_k$ be the (finite) set of such chosen elements of $\mathcal{U}_k$. In general, for $k\in M$ such that $\ell_{j}<k\le \ell_{j+1}$ choose $U^k_n\in\mathcal{U}_k$ such that $n\le n(\ell_{j+1})$, and let $\mathcal{F}_k$ be the (finite) set of such chosen elements of $\mathcal{U}_k$. 
For $k$ not in $M$, let $\mathcal{F}_k$ consist of any single element of $\mathcal{U}_k$.

We claim that the sequence $(\mathcal{F}_k:k\in\naturals)$ witnesses $\sfin(\Omega,\Omega)$ for the given sequence of $\omega$-covers of $X$. The reason is that for any $U^1_{k_1}\cap U^{k_1}_{n_1} \in\mathcal{K}$, the set $U^{k_1}_{n_1}$ is an element of $\mathcal{F}_{k_1}$, and thus $\mathcal{K}$ is a refinement of $\bigcup\{\mathcal{F}_k:k\in\naturals\}$, and the latter is an $\omega$-cover of $X$ since $\mathcal{K}$ is.
\end{proof}

As suggested by the conjecture above, we expect that in fact $\sone(\Omega,\Omega)$ is a consequence of the polarized square bracket relation, and most likely a different approach to a proof is needed.



\section{Remarks}

The results above are given for the family $\mathcal{A} = \Omega$ of $\omega$-covers of certain topological spaces. A study of the proofs will reveal that several of these equivalences also hold for several other families $\mathcal{A}$. The main requirements on a family $\mathcal{A}$ are that each element of $\mathcal{A}$ has a countable subset in $\mathcal{A}$, that for each $k$ $\mathcal{A}\rightarrow (\mathcal{A})^1_k$ holds, and that $\sone(\mathcal{A},\mathcal{A})$ is equivalent to ONE not having a winning strategy in $\gone(\mathcal{A},\mathcal{A})$, and that this is equivalent to the Ramseyan statement $\mathcal{A}\rightarrow(\mathcal{A})^2_2$.

As a non-topological example, consider the following: A collection $\mathcal{C}$ of subsets of a set $S$ is said to be a \emph{combinatorial} $\omega$-\emph{cover} of $S$ if $S\not\in\mathcal{C}$, but for each finite subset $F$ of $S$ there is a $C\in\mathcal{C}$ with $F\subseteq C$. For an infinite cardinal number $\kappa$ let $\Omega_{\kappa}$ be the set of \emph{countable} combinatorial $\omega$-covers of $\kappa$. Let ${\sf cov}(\mathcal{M})$ be the least infinite cardinal number $\kappa$ such that the real line is a union of $\kappa$ first category sets. By the Baire Category Theorem ${\sf cov}(\mathcal{M})$ is uncountable. 
Using the techniques of this paper one can prove:
\begin{theorem}\label{cardinality} For an infinite cardinal number $\kappa$ the following are equivalent:
\begin{enumerate}
\item{$\kappa<{\sf cov}(\mathcal{M})$.}
\item{$\sone(\Omega_{\kappa},\Omega_{\kappa})$.}
\item{$\egp(\Omega_{\kappa},\Omega_{\kappa})$.}
\item{${\sf GP}(\Omega_{\kappa},\Omega_{\kappa})$.}
\item{${\sf FG}(\Omega_{\kappa},\Omega_{\kappa})$.}
\item{${\sf NW}(\Omega_{\kappa},\Omega_{\kappa})$.}
\item{For all positive integers $n$ and $k$, $\Omega_{\kappa}\rightarrow(\Omega_{\kappa})^n_k$.}
\end{enumerate}
\end{theorem}

Here are three more, among many, topological examples.
For a topological space $X$ and an element $x\in X$, define $\Omega_x = \{A\subset X\setminus\{x\}:\, x\in\overline{A}\}$. According to \cite{Sakai} $X$ has  
strong countable fan tightness at $x$ if the selection principle $\sone(\Omega_x,\Omega_x)$ holds. Consider for a Tychonoff space $X$ the subspace of the Tychonoff product $\Pi_{x\in X}\mathbb{R}$ consisting of the continuous functions from $X$ to $\mathbb{R}$. The symbol ${\sf C}_p(X)$ denotes this subspace with the inherited topology. Since ${\sf C}_p(X)$ is homogeneous, the truth of $\sone(\Omega_f,\Omega_f)$ at some point $f$ implies the truth of $\sone(\Omega_f,\Omega_f)$ at any point $f$. Thus we may confine attention to $\Omega_{\bf o}$, where ${\bf o}$ is the function which is zero on $X$. Using the techniques above one can prove:
\begin{theorem}\label{strfantight} For a Tychonoff space $X$ the following are equivalent for ${\sf C}_p(X)$:
\begin{enumerate}
\item{$\sone(\Omega_{\bf o},\Omega_{\bf o})$.}
\item{$\egp(\Omega_{\bf o},\Omega_{\bf o})$.}
\item{${\sf GP}(\Omega_{\bf o},\Omega_{\bf o})$.}
\item{${\sf FG}(\Omega_{\bf o},\Omega_{\bf o})$.}
\item{${\sf NW}(\Omega_{\bf o},\Omega_{\bf o})$.}
\item{For all $n$ and $k$, $\Omega_{\bf o}\rightarrow(\Omega_{\bf o})^n_k$.}
\end{enumerate}
\end{theorem}

For a topological space $X$ let $\mathcal{D}$ denote the collection whose members are of the form $\mathcal{U}$, a family of open subsets of $X$, such that no element of $\mathcal{U}$ is dense in $X$, but $\cup\mathcal{U}$ is dense in $X$. And let $\mathcal{D}_{\Omega}$ be the set of $\mathcal{U}\in\mathcal{D}$ such that for each finite family $\mathcal{F}$ of nonempty open subsets of $X$ there is a $U\in\mathcal{U}$ with $U\cap F\neq\emptyset$ for each $F\in\mathcal{F}$. The families $\mathcal{D}$ and $\mathcal{D}_{\Omega}$ were considered in \cite{coc5} where it was proved that for $X$ a set of real numbers, and ${\sf PR}(X)$ the Pixley-Roy space over $X$, the following holds:
\begin{theorem}\label{prreals} If $X$ is a set of real numbers, the following are equivalent for ${\sf PR}(X)$:
\begin{enumerate}
\item{$\sone(\mathcal{D}_{\Omega},\mathcal{D}_{\Omega})$.}
\item{ONE has no winning strategy in the game $\gone(\mathcal{D}_{\Omega},\mathcal{D}_{\Omega})$.}
\item{For each $n$ and $k$ $\mathcal{D}_{\Omega}\rightarrow(\mathcal{D}_{\Omega})^n_k$.}
\end{enumerate}
Each of these statements is equivalent to $X$ having $\sone(\Omega_X,\Omega_X)$.
\end{theorem}
Using the techniques above one can prove:
\begin{theorem}\label{prRamsey} For a set $X$ of reals the following are equivalent for ${\sf PR}(X)$:
\begin{enumerate}
\item{$\sone(\mathcal{D}_{\Omega},\mathcal{D}_{\Omega})$.}
\item{$\egp(\mathcal{D}_{\Omega},\mathcal{D}_{\Omega})$.}
\item{${\sf GP}(\mathcal{D}_{\Omega},\mathcal{D}_{\Omega})$.}
\item{${\sf FG}(\mathcal{D}_{\Omega},\mathcal{D}_{\Omega})$.}
\item{${\sf NW}(\mathcal{D}_{\Omega},\mathcal{D}_{\Omega})$.}
\item{For all $n$ and $k$, $\mathcal{D}_{\Omega}\rightarrow(\mathcal{D}_{\Omega})^n_k$.}
\end{enumerate}
\end{theorem}

For a non-compact topological space $X$ call an open cover $\mathcal{U}$ a {\sf k}-cover if there is for each compact $C\subset X$ a $U\in\mathcal{U}$ such that $C\subseteq U$, and if $X\not\in\mathcal{U}$. Let $\mathcal{K}$ denote the collection of {\sf k}-covers of such an $X$. If $X$ is a separable metric space then each member of $\mathcal{K}$ has a countable subset which still is a member of $\mathcal{K}$. Using the techniques above one can prove:
\begin{theorem}\label{kcovRamsey} For separable metric spaces $X$ the following are equivalent:
\begin{enumerate}
\item{ONE has no winning strategy in $\gone(\mathcal{K},\mathcal{K})$.}
\item{$\sone(\mathcal{K},\mathcal{K})$.}
\item{$\egp(\mathcal{K},\mathcal{K})$.}
\item{${\sf GP}(\mathcal{K},\mathcal{K})$.}
\item{${\sf FG}(\mathcal{K},\mathcal{K})$.}
\item{${\sf NW}(\mathcal{K},\mathcal{K})$.}
\item{For all $n$ and $k$, $\mathcal{K}\rightarrow(\mathcal{K})^n_k$.}
\end{enumerate}
\end{theorem}
The equivalence of (2) and (7) for n=2 and k=2 is Theorem 8 of \cite{dkm}. The equivalence of (1) and (2) is a result of \cite{ST}. The remaining equivalences are then derived as was done above for $\Omega$.

In the above examples we did not list all the equivalences proved in Theorem \ref{egprothberger}. Indeed, we did not check all of these. To illustrate that the remaining equivalences are worth checking, consider the following example: When $\mathcal{A}$ is a free ultrafilter on $\naturals$, the following theorem holds:
\begin{theorem}\label{thm:ultrafilters} For a non-trivial ultrafilter $\mathcal{U}$ on $\naturals$ the following statements are equivalent:
\begin{enumerate}
\item{Selection principle $\sone(\mathcal{U},\mathcal{U})$ holds.}
\item{ONE has no winning strategy in the game $\gone(\mathcal{U},\mathcal{U})$.}
\item{For all natural numbers $m$ and $k$ the partition relation $\mathcal{U}\rightarrow(\mathcal{U})^m_k$ holds.}
\end{enumerate}
\end{theorem}
Ultrafilters $\mathcal{U}$ satisfying (1) of Theorem \ref{thm:ultrafilters} have been called selective ultrafilters, while ones satisfying (3) have been called Ramsey ultrafilters. It is known that the existence of an ultrafilter as in Theorem \ref{thm:ultrafilters}
is independent of \textsf{ZFC}. The equivalence of statements (1) and (3) is due to Kunen, while the equivalence of (1) and (2) is due to Galvin and McKenzie. Note that these implications are analogous to the ones for $\Omega$ appearing in Theorem \ref{egprothberger} among others. But not all equivalences in Theorem \ref{egprothberger} also hold for ultrafilters $\mathcal{U}$ on $\naturals$. For example: In \cite{Blass} Blass proved that under the Continuum Hypothesis there is an ultrafilter $\mathcal{U}$ on $\naturals$ such that $\mathcal{U}\rightarrow\lbrack\mathcal{U}\rbrack^2_3$, but $\mathcal{U}$ does not satisfy the selection principle $\sone(\mathcal{U},\mathcal{U})$. Thus $(8)\Rightarrow(1)$ of Theorem \ref{egprothberger} for $\Omega$, does not hold for ultrafilters on $\naturals$.

\bibliographystyle{amsplain}

\end{document}